\documentclass[11pt,a4paper,twoside]{article}
\usepackage{latexsym,amsfonts,amsmath, amsthm, amssymb}
\usepackage{fullpage}
\usepackage{xcolor,bm, bbm}
\usepackage[utf8x]{inputenc}
\usepackage{array}
\usepackage{mathrsfs}
\usepackage{tikz}
\usepackage{pgfplots}

\usepackage{fourier}

\newcommand{\Ham}{\mathbb H}
\newcommand{\R}{\mathbb R}
\newcommand{\N}{\mathbb N}
\newcommand{\C}{\mathbb C}
\newcommand{\E}{\mathbb E}
\newcommand{\Pro}{\mathbb P}
\newcommand{\dif}{\,\mathrm{d}}

\newcommand{\Var}{\mathrm{Var}}
\newcommand{\Uni}{\mathrm{Unif}}
\newcommand{\Sc}{\mathcal{S}}
\newcommand{\vol}{\mathrm{vol}}
\newcommand{\Mat}{\mathrm{Mat}}

\newcommand{\Sph}{\ensuremath{{\mathbb S}}}

\def\dint{\textup{d}}
\newcommand{\SSS}{\ensuremath{{\mathbb S}}}

\newcommand{\B}{\ensuremath{{\mathbb B}}}

\newcommand{\eqdistr}{\stackrel{d}{=}}
\newcommand{\toweak}{\overset{w}{\underset{n\to\infty}\longrightarrow}}

\DeclareMathOperator{\id}{Id}

\DeclareMathOperator{\Tr}{Tr}
\DeclareMathOperator{\rate}{\mathbb I}
\DeclareMathOperator{\ratej}{\mathbb J}

\renewcommand{\Re}{\operatorname{Re}}  

\newtheorem{thm}{Theorem}[section]

\newtheorem{cor}[thm]{Corollary}
\newtheorem{lemma}[thm]{Lemma}
\newtheorem{df}[thm]{Definition}
\newtheorem{proposition}[thm]{Proposition}

\newtheorem{thmalpha}{Theorem}

{
\theoremstyle{definition}

\newtheorem{rmk}[thm]{Remark}

}

\allowdisplaybreaks
\begin{document}

\title{\bf The large and moderate deviations approach in \\geometric functional analysis} 

\medskip

\author{Joscha Prochno}



\date{}

\maketitle

\begin{abstract}
\small
The work of Gantert, Kim, and Ramanan [Large deviations for random projections of {$\ell^p$} balls, Ann. Probab. 45 (6B), 2017] has initiated and inspired a new direction of research in the asymptotic theory of geometric functional analysis. The moderate deviations perspective, describing the asymptotic behavior between the scale of a central limit theorem and a large deviations principle, was later added by Kabluchko, Prochno, and Th\"ale in [High-dimensional limit theorems for random vectors in $\ell_p^n$ balls.~II, Commun. Contemp. Math. 23(3), 2021]. These two approaches nicely complement the classical study of central limit phenomena or non-asymptotic concentration bounds for high-dimensional random geometric quantities. Beyond studying large and moderate deviations principles for random geometric quantities that appear in geometric functional analysis, other ideas emerged from the theory of large deviations and the closely related field of statistical mechanics, and have provided new insight and become a source for new developments. Within less than a decade, a variety of results have appeared and formed this direction of research. Among others, a connection to the famous Kannan--Lov\'asz--Simonovits conjecture and the study of moderate and large deviations for isotropic log-concave random vectors was discovered. In this manuscript, we introduce the basic principles, survey the work that has been done, and aim to manifest this direction of research, at the same time making it more accessible to a wider community of researchers.   
\medspace
\vskip 1mm
\noindent{\bf Keywords}. {Cram\'er's theorem, Gibbs measure, moderate deviations principle, large deviations principle, Orlicz ball, random projection, Sanov's theorem, Schatten class}\\
{\bf MSC}. Primary 46B20, 60F10; Secondary 47B10, 52A23 
\end{abstract}

\tableofcontents

\section{Introduction}

Already the early years in the local theory of Banach spaces and geometric functional analysis have demonstrated a deep connection between the geometry of high-dimensional normed spaces, and thus of high-dimensional convex bodies, and probability theory. This fruitful interplay has already been nicely displayed in the wonderful and influential monographs \emph{Asymptotic theory of finite-dimen-sional normed spaces} by V.~D.~Milman and G.~Schechtman \cite{MS1986}, \emph{The volume of convex bodies and Banach space geometry} by G.~ Pisier \cite{P1989}, and \emph{Probability in Banach Spaces} by M.~Ledoux and M.~Talagrand \cite{LT1991}, and has recently taken a new direction as nicely presented in \cite{V2018}; we do not even make the attempt here to list the numerous influential and groundbreaking papers in this direction. Even though high-dimensional structures and randomness seem inevitably to lead to the existence of more or too many possibilities, resulting in an increased diversity and complexity as the dimension increases, the past decades have shown that this is often compensated and even patterns are created and concentration phenomena evoked, its source being convexity. Along the way, powerful methods have been developed by a large community of gifted mathematicians, both on the geometric and analytic, as well as on the probabilistic side. In fact, new fields at the crossroads between functional analysis, discrete and convex geometry, and probability theory have emerged, today known as asymptotic geometric analysis (also asymptotic convex geometry) and high-dimensional probability theory. Due to their common origin both fields overlap in many different ways, while still leaning towards the mathematical disciplines the names already indicate. The methodical diversity, drawing on ideas from analysis, geometry, and probability, is a common characteristic for both. Today the methods have become indispensable in numerous problems of (multivariate) statistics, computer science, theoretical numerical analysis, or compressed sensing, and quite often research activities are driven by questions that arose in the applied sciences, for instance, the famous {K}annan--{L}ov\'{a}sz--{S}imonovits conjecture \cite{KLS1995}, which stems from an algorithmic question about the complexity of volume computation of a high-dimensional convex body (we will encounter this conjecture again later). We refer to the monographs \cite{AGM2015_I,AGM2021_II} and \cite{IsotropicConvexBodies}, which provide an outstanding and extensive overview of modern asymptotic geometric analysis until the year $2021$, the recent excellent introduction to high-dimensional probability with applications in data science \cite{V2018}, and the comprehensive textbook \cite{FR2013} about the mathematical aspects of compressed sensing, all of which are complemented by their nearly exhaustive lists of references. In particular, the monographs \cite{AGM2015_I,AGM2021_II,IsotropicConvexBodies} treat the classical central limit perspective, most prominently Klartag's central limit theorem for convex bodies \cite{K2007}, saying that given any isotropic convex body in high dimension its typical random projections will be approximately standard Gaussian (see \cite{ABP2003} for the origin of the problem). This is complemented by other central limit theorems such as \cite{APT2019}, \cite{JP2019}, \cite{KPT2019_I}, \cite{PPZ14}, or \cite{SS1991}, just to mention a few (please consult the respective lists of references for more in this direction). In this manuscript, we shall be concerned with other types of limit theorems going beyond the normal fluctuations covered by a central limit theorem.

\subsection{The large and moderate deviations perspective}

The beautiful universality that emerges from the central limit perspective unfortunately restricts the information that may be retrieved, for instance, from lower-dimensional projections of a high-dimensional convex body. But which aspects of, for instance, random projections can be used to distinguish between different convex bodies? 

In order to answer this question, let us recall that a sequence $(X_n)_{n\in\N}$ of random vectors in $\R^d$ satisfies a large deviations principle (LDP) with lower-semicontinuous rate function $\rate$ and speed $s_n\uparrow \infty$ if and only if for all Borel sets $A\subset \R^d$
\[
\Pro[X_n\in A] \approx e^{-s_n \inf_{x\in A}\rate(x)}.
\]
A moderate deviations principle (MDP) formally resembles an LDP but on a different scale and, typically, with significant differences in the behavior. More precisely, the scale for an MDP of a random quantity is between that of a weak limit theorem (like a central limit theorem) and that of a law of large numbers. 
It is instructive to consider the case of independent and identically distributed (iid) $\R$-valued random variables $X_i$, $i\in\N$, with finite exponential moments for which Cram\'er's theorem \cite[Theorem I.4]{dH2000} guarantees that, for all $a>\E[X_1]$,
\[
\lim_{n\to\infty} \frac{1}{n}\log \Pro\Big[\frac{1}{n}\sum_{i=1}^nX_n\geq a\Big] = -\rate(a),
\] 
with $\rate$ being the Legendre transform of the log-moment generating function of $X_1$. While this LDP occurs at scale $n$ of a law of large numbers, as seen from the expression $\frac{1}{n}\sum_{i=1}^nX_i$, an MDP occurs (under suitable conditions, e.g., \cite{EL2003}) on a scale between $\sqrt{n}$ and $n$, more precisely, $\frac{1}{b_n}\sum_{i=1}^nX_i$ with $b_n/n\to 0$ and $b_n/\sqrt{n}\to\infty$.   

In contrast to a central limit theorem, large and moderate deviations principles are known to be typically non-universal and distribution-dependent, more precisely, they are parametric in speed and/or rate function; MDPs are usually less parametric in the rate (up to a variance dependence that may appear). 
In this sense, the large and moderate deviations behavior of, for instance, a random projection of a convex body, depends in a subtle way on the geometry of the underlying body. In other words, LDPs and MDPs allow one to distinguish high-dimensional probability measures, like the uniform distributions on a high-dimensional convex body, via their lower-dimensional projections. We shall describe below the
development of this large and moderate deviations approach in the asymptotic theory of geometric functional analysis, but also the influence of large deviations (and in particular statistical mechanics) ideas and techniques, before we move on with a short introduction to large deviations theory and a more detailed presentation and discussion of some of the results presented now.  
\vskip 1mm

The contributions we have seen and the progress that has been made in the past seven years started with an inspiring work of N.~Gantert, S.~S.~Kim, and K.~Ramanan \cite{GKR2017} who studied the question about the large deviations behavior of one-dimensional marginals of (scaled) $\ell_p^d$-balls defined as $\B_p^d:=\{x\in\R^d\,:\,\|x\|_p\leq 1 \}$, $1\leq p\leq \infty$, both in an annealed (i.e., marginal in a random direction) and quenched (i.e., marginal in a fixed deterministic direction) setting. In the annealed setting, their work reveals a stark difference in the behavior both speed- and rate-wise, depending on whether $1\leq p \leq 2$ or $2<p\leq \infty$. The authors also identify a variational formula that relates the annealed and quenched LDPs.

This is where D.~Alonso-Guti\'errez, J.~Prochno, and C.Th\"ale started, looking at higher-dimensional marginals, studying LDPs for the Euclidean norm of random orthogonal $k_d$-dimensional projections of random vectors uniformly distributed on $\ell_p^d$-balls; here $k_d\in\N$ may tend to infinity as $d\to\infty$. The observations are two-fold. On the one hand, we see again a significantly different large deviations behavior, depending on whether $1\leq p\leq 2$ or $2< p\leq \infty$. On the other hand, the form of the rate function that governs the LDP depends on the asymptotic behavior of the random subspace dimensions $k_d$ relative to $d$. One should particularly notice that for $1\leq p<2$, both the speed and the rate function differ from those for the $1$-dimensional random projections in \cite{GKR2017}. Notably, when $p=2$ the LDP does not feel the dimensions $k_d$ of the random subspaces as long as they do not increase too fast with $d$ and we obtain the same LDP as in \cite{GKR2017}. The difference to the $1$-dimensional projections becomes visible only in the truly high-dimensional regime in which $k_d$ is eventually proportional to the dimension $d$ of the ambient space. When $2<p\leq \infty$ the LDP occurs at the same speed as in \cite{GKR2017}, but with a different rate function.

The case of general $\ell_q$-norms of random vectors uniformly distributed on $\B_p^d$, i.e., without involving any projections, was studied by Z.~Kabluchko, J.~Prochno, and C.~Th\"ale in \cite{KPT2019_I}. The LDPs obtained in \cite{KPT2019_I} significantly differ depending on the position of $p$ relative to $q$. It is shown that when $q\leq p$, the speed is $d$ and the rate function feels the radial part of the uniform distribution on $\B_p^d$, while a lack of exponential moments makes the analysis for $q>p$ more delicate, resulting in a slower speed $d^{p/q}$ and a non-convex rate function. In particular, the LDP in this case is not susceptible to the radial component. In the follow-up paper \cite{KPT2019_II}, the same authors studied $\ell_q$-norms of random vectors with distributions from a more general class related to $\ell_p$-balls. In particular, they considered MDPs for those random quantities in the case where $q<p$. A somehow related work, where the random points are chosen uniformly at random in a centered regular simplex in high dimensions, is \cite{BKPST2020}. There the authors obtained Berry--Esseen bounds in the regime $1\leq q <\infty$, complemented those by a non-central limit theorem together with moderate and large deviations in the case where $q=\infty$, and compared those results to the previously mentioned ones for the crosspolytope $\B_1^d$. We also refer to the work \cite{GKT2019} of J.~Grote, Z.~Kabluchko, and C.~Thäle for several types of limit theorems, including MDPs and LDPs, for high-dimensional random simplices.

The results mentioned so far are often referred to as level--1 large deviations. The level--2 large deviations or Sanov-type large deviations concern LDPs for random empirical measures. In \cite{KR2018}, S.~S.~Kim and K.~Ramanan studied conditional limit theorems for $\ell_p$-spheres in the flavor of the classical and famous Poincar\'e--Maxwell--Borel Lemma (see \cite{B1914} or \cite{DF1987}). Central to their analysis is an LDP that the authors establish for the sequence of empirical measures of the coordinates of a random vector uniformly distributed on the sphere $\SSS_p^{d-1}:=\{x\in\R^d\,:\, \|x\|_p=1 \}$ (both, random vector and sphere are scaled with a factor $d^{1/p}$). Notably, the LDP not only holds in the weak topology (on the space of probability measures), but even in the stronger $q$-Wasserstein topology as long as $q<p$. A nice combination with another element of large deviations theory, the Gibbs conditioning principle, allows them to obtain the conditional limit theorem.  

\vskip 1mm

Let us briefly stop here and explain a central tool in all the work presented so far, which will then also serve as motivation for the research discussed in a moment. In the world of $\ell_p^d$-balls, there exists an inconspicuous
probabilistic representation of the uniform measure on $\B_p^d$, relating it for any $1\leq p<\infty$ to the so-called $p$-Gaussian distribution with Lebesgue density of the form
\[
\R \ni x\mapsto C_p e^{-|x|^p/p},
\]
where $C_p\in(0,\infty)$ is a normalization constant.
More precisely, the work of G.~Schechtman and J.~Zinn \cite{SchechtmanZinn} (see also the independent work \cite{RR1991} by S.~T.~Rachev and L.~Rüschendorf) shows that if $X$ is uniformly distributed on $\B_p^d$, and $Z_1,\dots,Z_d$ are independent $p$-generalized Gaussian random variables, then 
\[
X \stackrel{\dint}{=} U^{1/d}\frac{(Z_1,\dots,Z_d)}{\|(Z_1,\dots,Z_d)\|_p},
\]
where $U$ is uniformly distributed on $[0,1]$ and independent of $Z_1,\dots,Z_d$. In view of a probabilistic approach to the geometry of $\ell_p^d$-balls, this representation is quite powerful, because it allows to make a transition from a random vector with dependent coordinates (like $X$) to one with independent coordinates (like $(Z_1,\dots,Z_d)$); this makes many problems computationally feasible. Behind any of the large and moderate deviations results for $\ell_p^d$-balls stands this Schechtman--Zinn probabilistic representation, also in \cite{GKR2017}. 

\vskip 1mm

Motivated by the work in \cite{KR2018}, Z.~Kabluchko, J.~Prochno, and C.~Th\"ale investigated Sanov-type LDPs in the non-commutative setting of Schatten classes \cite{KPT2019_sanov}, where no such Schechtman--Zinn representation had been available. We recall that the Schatten $p$-class is the space of $d\times d$ real matrices  where the norm of a matrix $A$ is given by the $\ell_p^d$-norm of the vector of singular values of $A$. The authors established a non-commutative version of the Schechtman--Zinn representation and proved an LDP for the (properly scaled) empirical spectral measure of a matrix chosen uniformly at random from a unit ball in the Schatten $p$-class or from the boundary with respect to the cone measure (on the space of probability measures equipped with the weak topology). As a consequence they also obtained that the spectral measure converges weakly almost surely to a non-random limiting measure given by the Ullman distribution. It turns out that in this non-commutative setting, the analysis is considerably more delicate. This paper was the first going beyond the setting of $\ell_p$-balls. Moreover, the ideas and results developed in \cite{KPT2019_sanov} and the related works \cite{KPT2020_volume_ratio,KPT2020_ensembles} (see also \cite{ST2023}), which largely stem from large deviations theory and statistical mechanics, where crucial in \cite{DFGZ2023} to establish a generalized version of the variance conjecture for the family of Schatten $p$-class unit balls of self-adjoint matrices.

Another work dealing at least in parts with the non-commutative setting is \cite{KT2022_matrix_p_balls}. T.~Kaufmann and C.~Thäle derive a probabilistic representation for classes of weighted $p$-radial distributions, which, in essence, are based on mixtures of a weighted cone probability measure and a weighted uniform distribution on $\B_p^d$. The authors then prove LDPs for the empirical measure of the coordinates of vectors chosen at random from $\B_p^d$ with respect to distributions from the weighted measure class. Those distributions are then extended to $p$-balls in classical matrix spaces, both for self-adjoint and non-self-adjoint matrices, and their eigenvalue and singular value distributions are identified for the respective random matrices. As an application, LDPs for the empirical spectral measures for the eigenvalues and the singular values are proved.

In \cite{KPT2019_cube}, Z.~Kabluchko, J.~Prochno, and C.~Th\"ale proved a strong law of large numbers for $k$-dimensional random projections of the $d$-dimensional cube via the Artstein-Vitale law of large numbers for random compact sets. This result, which can also be derived from the celebrated Dvoretzky--Milman theorem, says that with respect to the Hausdorff distance, a uniform random projection of the $d$-dimensional cube onto $\R^k$ converges almost surely to a centered $k$-dimensional Euclidean ball of radius $\sqrt{2/\pi}$, as $d\to\infty$. For every point inside this, the authors determined the asymptotic number of vertices and the volume of the part of the cube projected `close' to this point. This result is a consequence of a more general large deviations result for random projections of arbitrary product measures obtained in this paper.

Related to the $d$-dimensional cube is also a work of S.~Johnston, Z.~Kabluchko, and J.~Prochno, who studied the large deviations behavior of random one-dimensional projections of the uniform distribution on the cube $[-1,1]^d$ \cite{JKP2021}. While a combination of Klartag's central limit theorem for convex bodies with the Borel--Cantelli lemma implies the almost sure weak convergence to a centered Gaussian distribution of variance $1/3$, and thus shows that the typical behavior is universal, the authors establish an LDP showing that the atypical behavior is not universal. In fact, the speed of the LDP is $d$ and the rate function is explicit, thus giving direct access to the dependence on the underlying distribution.

The previous results from \cite{JKP2021} have been considerably generalized by Z.~Kabluchko and J.~Prochno in \cite{KP2024}, where LDPs for random matrices in the orthogonal group and Stiefel manifold are proved. The good convex rate functions are explicitly given in terms of certain log-determinants of trace-class operators and are finite on the set of Hilbert--Schmidt operators $M$ satisfying $\|MM^*\|<1$. Those results are then applied to determine the precise large deviation behavior of $k$-dimensional random projections of high-dimensional product distributions using an appropriate interpretation in terms of point processes, also providing a characterizing of the space of all possible deviations. The case of uniform distributions on $\B_p^d$, $1\leq p\leq \infty$, is also considered and reduced to appropriate product measures.

A work partly generalizing the one of Z.~Kabluchko and J.~Prochno \cite{KP2024} is the recent paper \cite{KST2022} by T.~Kaufmann, H.~Sambale, and C.~Thäle. Here the authors extend the set of projected distributions for which an LDP is shown (going from the uniform distribution to quite general $p$-radial distributions) and observe that the large deviations behavior obtained in \cite{KP2024} is universal for a large class of probability measures on $\B^d_p$. They also describe geometrically motivated distributions on the balls $\B^d_p$ for which the LDP needs a suitable modification.

Another work dealing with large deviations induced by the Stiefel manifold is \cite{KR2023_stiefel}. In this paper, S.~S.~Kim and K.~Ramanan prove LDPs for $k$-dimensional, deterministic Stiefel projections of classes of random vectors in $\R^d$ ($k<d$ fixed) including uniform distributions on suitably normalized $\ell_p^d$-balls with $1<p\leq \infty$ and product measures with sufficiently light tails; the results hold under some convergence assumption on the empirical measures of the properly scaled Stiefel matrix. The results are then nicely applied to obtain large deviations for random projections in the quenched sense, where  the Stiefel matrices are chosen at random with respect to the Haar measure on the Stiefel manifold and quenched refers to the fact that one conditions on almost sure realizations of the random Stiefel matrix. The authors also obtain a variational formula for the rate function governing the LDP for annealed projections. It is important to note that the key difference compared to \cite{KP2024} is in the object of study, which in \cite{KR2023_stiefel} is the projection point itself, hence yielding an LDP on $\R^k$, whereas in \cite{KP2024} it is the projected distribution on $\R^k$ (and thus the main result yields an LDP on the space of probability measures on $\R^k$.

Having its origin in theoretical computer science, the Kannan--Lov\'asz--Simonovits conjecture (see, e.g., \cite{AGB2015}) is arguably the most famous open problem in asymptotic geometric analysis and high-dimensional probability theory today; the currently best known result is due to B.~Klartag \cite{K2023}. As shown by D.~Alonso-Guti\'{e}rrez, J.~Prochno, and C.~Thäle \cite{APT2020}, there is a connection between this conjecture and the study of large and moderate deviations for isotropic log-concave random vectors, therefore providing an alternative possibility to tackle the conjecture by presenting a potential route to disproving it. In their work, the authors study the moderate deviations for the Euclidean norm of random orthogonally projected random vectors in $\B_p^d$, leading to a  number of interesting observations, namely that the $\ell_1^d$-ball is already critical for the new approach, that for $p \geq 2$ the rate function in the MDP undergoes a phase transition, depending on whether the scaling is below the square-root of the subspace dimensions or comparable, and finally that for $1\leq p < 2$ and comparable subspace dimensions, the rate function again displays a phase transition depending on its growth relative to $d^{p/2}$.

While the work on Sanov-type large deviations in the non-commutative setting of Schatten $p$-classes \cite{KPT2019_sanov} was the first in which an LDP was derived without the classical Schechtman--Zinn probabilistic representation, it took a while to understand how to go beyond the framework of the (commutative) sequence spaces $\ell_p$. Quite recently, Z.~Kabluchko and J.~Prochno used a principle deeply manifested in large deviations theory and statistical mechanics, the so-called maximum entropy or Gibbs principle for non-interacting particles, to establish a probabilistic approach to the geometry of Orlicz balls \cite{KP2020}. More precisely, they extracted from this maximum entropy principle an asymptotic version of a probabilistic representation for the uniform distribution on an Orlicz ball. This allowed them, among other things, to compute the precise asymptotic volume of Orlicz balls (based on sharp large deviations techniques) and to determine the threshold behavior for the volume of intersections of Orlicz balls, a result extending the $\ell_p$-ball counterpart of G.~Schechtman and M.~Schmuckenschl\"ager \cite{SS1991}. In particular, it became clear that the role of the Schechtman--Zinn probabilistic representation is replaced, in an asymptotic form, by appropriate Gibbs distributions whose potentials are given by Orlicz functions. Let us note that, in essence, those ideas have been used by F.~Barthe and P.~Wolff in \cite{BW2022} to show for a large family of Orlicz balls that they verify the R.~Kannan, L.~Lovász, and M.~Simonovits spectral gap conjecture.

S.~T.~Rachev and L.~R\"uschendorf \cite{RR1991} and A.~Naor and D.~Romik \cite{NR2003} unified results obtained in the 1980s by P.~Diaconis and D.~Freedman \cite{DF1987}, establishing
a connection between $\ell_p^d$-balls and a $p$-generalized Gaussian distribution. In \cite{JP2023_maxwell_orlicz}, S.~Johnston and J.~Prochno studied similar
questions in a significantly broader setting, looking at low-dimensional projections of random
vectors uniformly distributed on sets of the form $\B_{\phi,t}^d := \{ (s_1,\ldots,s_d) \in\R^d \,:\, \sum_{i=1}^d\phi(s_i) \leq  t d\}$, where
$\phi:\R \to [0,\infty]$ is a function satisfying some fairly mild conditions; this framework includes the simplex as well as the $\ell_p^d$-balls for $p>0$ and the more general case
of Orlicz balls. Using an approach completely different from \cite{RR1991} and  \cite{NR2003}, which is again based on ideas from the theory of large deviations and statistical mechanics, specifically those centered around Cram\'er’s theorem, the Gibbs conditioning principle, and Gibbs measures, the authors showed that there is a critical parameter $t_{\text{crit}}$ at which there is a phase transition in the behavior of the low-dimensional projections. More precisely, for $t > t_{\text{crit}}$ the coordinates of projected random vectors sampled from $\B_{\phi,t}^d$ behave like independent uniform random variables, but for $t \leq t_{\text{crit}}$ however the Gibbs conditioning principle comes into play, and here there is a parameter $\beta_t > 0$ (the so-called inverse temperature) such that the coordinates are approximately distributed according to a density proportional to $e^{-\beta \phi(s)}$. The general ideas of this paper have been developed in parallel with the ones put forward by Z.~Kabluchko and J.~Prochno in \cite{KP2020}.

In the work \cite{AP2020}, D.~Alonso-Guti\'errez and J.~Prochno studied the asymptotic thin-shell concentration for random vectors uniformly distributed on Orlicz balls. Their approach was inspired by the ideas developed in \cite{KP2020} and combined with moderate deviations techniques. Among other things, this allowed the authors to determine the precise asymptotic value of the isotropic constant of Orlicz balls and to provide both upper and lower bounds on the probability of such a random vector being in a thin shell whose radius is a constant multiple of the dimension $\sqrt{d}$. In fact, they manage to prove that in certain cases their bounds are best possible and improve upon the currently best known general bound when the deviation parameter goes down to zero as the dimension $n$ of the ambient space increases.

The work \cite{KR2018} of S.~S.~Kim and K.~Ramanan on Sanov-type large deviations and conditional limit theorems for $\ell_p$-spheres has recently been lifted to the setting of Orlicz spaces by L.~Frühwirth and J.~Prochno in \cite{FP2024}; the underlying idea using appropriate Gibbs distributions again stems from \cite{KP2020}. In the paper, a Sanov-type large deviations principle for the sequence of empirical measures of vectors chosen uniformly at random from an Orlicz ball is obtained and then, from this level-2 large deviations result, in a combination with Gibbs conditioning, entropy maximization, and the Orlicz version of the Poincaré--Maxwell--Borel lemma due to S.~Johnston and J.~Prochno \cite{JP2023_maxwell_orlicz}, the authors deduce a conditional limit theorem for high-dimensional Orlicz balls. In geometric parlance, the latter shows that if $M_1$ and $M_2$ are Orlicz functions, then random points in the $M_1$-Orlicz ball, conditioned on having a small $M_2$-Orlicz radius, look like an appropriately scaled $M_2$-Orlicz ball. In fact, it is shown that the limiting distribution in their Poincaré--Maxwell--Borel lemma, and thus the geometric interpretation, undergoes a phase transition depending on the magnitude of the $M_2$-Orlicz radius.

A few years ago there was actually a breakthrough on the commutative front. S.~S.~Kim, Y.-T.~Liao, and K.~Ramanan established both level--1 and level--2 large deviations principles for sequences of random vectors under asymptotic thin-shell-type conditions  \cite{KLR2019}. In its most elementary form, this condition requires the sequence of Euclidean norms of the random vectors, normalized by a factor $d^{-1/2}$, to satisfy an LDP at speed $d$ and with rate function having a unique minimum. Notably, the authors manage to establish this asymptotic thin-shell condition in the case of random vectors uniformly distributed on Orlicz balls whose Orlicz functions are of super-quadratic growth. This is very nice for several reasons, but most prominently is maybe the following: the authors identify a similarity to Klartag's famous central limit theorem which follows from the classical thin-shell estimate. The mentioned work is quite impressive and the goal in the near future will be to establish the asymptotic thin-shell condition for other classes of random vectors, for instance, for those chosen uniformly at random from unit balls of Lorentz sequence spaces (these spaces are, besides Orlicz sequence spaces, the most natural sequence space generalizations of $\ell_p$-spaces and also belong to the important class of $1$-symmetric Banach spaces); we refer to \cite{KPS2023} for a recent and first probabilistic approach in this setting. 

Last but not least, let us mention that in recent years sharp large deviations results in the flavor of R.~R.~Bahadur and R.~R.~Ranga Rao \cite{BRR1960} or V.~V.~Petrov \cite{P1965} have been obtained for several random geometric quantities. While LDPs identify the asymptotic exponential decay rate of probabilities in terms of speed and rate function, sharp large deviation estimates also provide the subexponential prefactor in front of the exponentially decaying term, which is not visible on the logarithmic scale of an LDP. In that sense, sharp large deviations provide more accurate quantitative estimates in finite dimensions, which is more in the spirit of the local theory of Banach spaces. This line of research was initiated by Y.-T.~Liao and K.~Ramanan \cite{LR2020} and picked up by T.~Kaufmann in \cite{K2021} and L.~Frühwirth and J.~Prochno in \cite{FP2024_sharp}. Since the techniques differ in various aspects from classical large deviations techniques, we shall not discuss those results here and rather refer the interested reader to the references above.

\section{An outrageously short introduction to large deviations theory}

Let us continue with a brief introduction to large deviations theory and some of its fundamental concepts. The systematic study and unified formalization of large deviations theory, which in essence concerns the asymptotic behavior of remote tails of sequences of probability distributions, was developed by S.~Varadhan in 1966 \cite{V1966}. In fact, he was later awarded the prestigious Abel Prize ``for his fundamental contributions to probability theory and in particular for creating a unified theory of large deviations'' in 2007.

We shall state results here in their most elementary form (sketch some of their proofs) and provide references to more general ones where needed. However, the goal is to provide the uninitiated reader with a rough understanding of the underlying ideas and principles that are required later on. For a thorough introduction to this fascinating topic, we refer the reader to the survey of S.~Varadhan \cite{V2008} or the wonderful monographs by A.~Dembo and O.~Zeitouni \cite{DZ2010}, R.~S.~Ellis \cite{E2006}, F.~den Hollander \cite{dH2000}, or by F.~Rassoul-Agha and T.~Sepp\"{a}l\"{a}inen \cite{RAS2015} as well as the survey article by H.~Touchette \cite{T2009}.

\subsection{Motivation}

Assume we are given iid random variables $X_1,X_2,\dots$, say for simplicity $\E[X_1]=0$ and $\Var[X_1]=\sigma^2>0$, and let $S_n:=\sum_{i=1}^nX_i$, $n\in\N$, denote the sequence of partial sums. The law of large numbers tells us that, for any $\varepsilon\in(0,\infty)$,
\[
\lim_{n\to\infty} \Pro\Big[\Big|\frac{S_n}{n} \Big|\geq \varepsilon\Big|\Big] = 0.
\]
Chebyshev's inequality yields a convergence rate of $1/n$, but can we, at least under sufficient moment conditions, improve upon that bound? 

An event of the form $\{|S_n|\geq n\varepsilon \}$ is referred to as a large deviation or rare event, compared to the normal scale $\{|S_n|\geq \sqrt{n}\varepsilon \}$ considered in a weak limit theorem. By a direct computation one can check that, for instance, a sequence of independent standard Gaussians $g_1,g_2,\dots$ satisfies, for any $a>0$,
\[
\lim_{n\to\infty} \frac{1}{n} \log\Pro\Big[\sum_{i=1}^n g_i \geq a n\Big] = -\frac{a^2}{2},
\]
and that a sequence of independent Rademacher random variables $\xi_1,\xi_2,\dots$ satisfies
\[
\lim_{n\to\infty} \frac{1}{n} \log\Pro\Big[\sum_{i=1}^n \xi_i \geq a n\Big] = - I(a),
\]
where $I:\R\to[0,\infty]$ is defined as
\[
I(x):=
\begin{cases}
\frac{1+x}{2}\log(1+x) + \frac{1-x}{2}\log(1-x) &:\,x\in[-1,1] \\
\infty &:\, |x|>1.
\end{cases}
\]
If $X_1,X_2,\dots$ denotes either the sequence of Gaussians or the one of Rademachers, then both examples indicate a similar behavior of the form 
\[
\Pro\Big[\sum_{i=1}^nX_i \geq a n\Big] \approx e^{-n\rate(a)}
\]
for an appropriate function $\rate:\R\to[0,\infty]$ and $a>0$. In particular, the rate of convergence is exponential and thus way faster than the rate $1/n$ suggested by Chebyshev's inequality.

This raises the question whether there is a more general principle behind this? Let us assume that the random variables $X_1,X_2,\dots$ are iid, centered, and have all exponential moments, i.e., for all $t\in\R$, the moment generating function satisfies
\[
\lambda_X(t):=\E\Big[e^{tX_1}\Big] < \infty.
\]
It follows from Markov's inequality, together with the iid assumption, that, for any $a,t>0$,
\[
\Pro\Big[S_n \geq a n\Big] = \Pro\Big[e^{tS_n}\geq e^{tan}\Big] \leq e^{-tan} \E\Big[e^{tS_n}\Big] = e^{-atn}\prod_{i=1}^n \lambda_{X}(t) = e^{-n[ta - \log\lambda_X(t)]},
\]
 and a simple optimization in the parameter $t>0$ shows that, for any $a>0$,
\[
\Pro\Big[S_n \geq a n\Big] \leq e^{-n\sup_{t>0}[ta-\log\lambda_X(t)]}.
\]
In particular,
\[
\limsup_{n\to\infty}\frac{1}{n}\log \Pro\Big[S_n \geq an\Big] \leq -n\sup_{t>0}[ta-\log\lambda_X(t)],
\]
and the right-hand side is essentially the Legendre transform of the log-moment generating function of $X_1$. 

\subsection{Cram\'er type large deviations}

It turns out that the so-called Chernoff bound with the Legendre transform presented above is actually optimal and this is the assertion of a classical theorem due to H. Cram\'er \cite{C1938}. Let us state this result here in its simplest form.

\begin{thm}[H. Cram\'er, 1938]\label{thm:cramer}
Let $X_1,X_2,\dots$ be iid random variables such that, for all $t\in\R$,
\[
\Lambda_X(t) :=\log \E\Big[e^{tX_1}\Big]<\infty.
\]
Then, for every $a>\E[X_1]$,
\[
\lim_{n\to\infty}\log \Pro[S_n\geq a n] = -\Lambda_X^*(a),
\]
where $\Lambda_X^*(a):=\sup_{s\in\R}[as-\Lambda_X(s)]$ is the Legendre transform of $\Lambda_X$.
\end{thm}

\begin{proof}[Idea of Proof.]
First of all, it is easy to see that it is enough to consider the case $\E[X_1]<0$ and $a=0$, i.e., it is enough to prove that
\[
\lim_{n\to\infty} \frac{1}{n} \log \Pro\Big[\sum_{i=1}^nX_i \geq 0\Big] = - \Lambda^*_Y(0).
\]

As we have seen, the upper bound follows essentially from the exponential Markov inequality. To argue that the supremum is actually attained on $(0,\infty)$ one needs to take the distribution of mass of $X_1$ into account, which means consider separately the following cases: (1) $\Pro[X_1<0]$, (2) $\Pro[X_1\leq 0]$ and $\Pro[X_1=0]>0$ as well as (3) $\Pro[X_1<0]>0$ and $\Pro[X_1>0]>0$. Note that we used here the reduction to $\E[X_1]<0$.

To explain the idea behind the proof the lower bound, let us assume that for a set of nice random variables $(Y_i)_{i\in\N}$ and $a>\E[Y_1]$, $\sigma^2=\Var[X_1]>0$, we are interested in $\Pro[\sum_{i=1}^nY_i\geq a n]$. If we standarize the partial sum of random variables, then we obtain
\[
\Pro\Big[\sum_{i=1}^nY_i\geq a n\Big] = \Pro\Big[ \frac{S_n-n\E[X_1]}{\sigma \sqrt{n}} \geq \frac{\sqrt{n}(a-\E[X_1])}{\sigma}\Big].
\]
Now, for large $n\in\N$, Gaussian approximation tells us that
\[
\Pro\Big[\sum_{i=1}^nY_i\geq a n\Big] \approx \frac{1}{\sqrt{2\pi}}\int_{\frac{\sqrt{n}(a-\E[X_1])}{\sigma}}^\infty e^{-y^2/2}\,\dint y.
\]
Because of the $\sqrt{n}$ in the lower bound of the integration domain, this Gaussian approximation is not good if $a-\E[X_1]$ is not small. The idea, which actually goes back to F. Esscher \cite{E1932}, is to perform an exponential tilt of the original measure (today known as Esscher transform or Cram\'er transform) that adds more weight to larger values, thereby shifting the mean to the right, which eliminates the lower bound of the integration domain. As it turns out, the price that one pays in the integrand is acceptable and the central limit theorem can be properly applied to obtain the lower bound in Cram\'er's theorem.

More formally, in our set-up of $\mu:=\Pro^{X_1}$, we define the exponentially tilted measure to be 
\[
\widetilde \mu(\dint x) := \frac{e^{t_{\text{min}}x}}{\rho} \mu(\dint x)
\]
with $\rho:=e^{-\Lambda_X^*(0)} = \inf_{t\in\R} \lambda_X(t)$ and $t_{\textrm{min}} > 0$ such that $\lambda_X(t_{\textrm{min}})=\rho$ and $\lambda_X'(t_{\textrm{min}})=0$. Given this measure, we consider independent random variables $\widetilde X_1, \widetilde X_2,\dots$ with distribution $\widetilde \mu$. In particular, those random variables satisfy $\E[\widetilde X_1]=0=a$ and thus deal with the problematic part mentioned before. Moreover, a change of measure from $\mu$ to $\widetilde \mu$ shows that 
\[
\Pro[S_n \geq 0] = \rho^n \E\Bigg[e^{t_{\text{min}}\widetilde S_n}\mathbbm 1_{\{\widetilde S_n \geq 0\}}\Bigg],
\]
where $\widetilde S_n := \sum_{i=1}^n\widetilde X_i$. Thus,
\[
\frac{1}{n} \log \Pro[S_n \geq 0] = \underbrace{\log(\rho)}_{=-\Lambda_X^*(0)} + \frac{1}{n} \log \E\Bigg[e^{t_{\text{min}}\widetilde S_n}\mathbbm 1_{\{\widetilde S_n \geq 0\}}\Bigg]
\]
and it is therefore enough to show that the second summand on the right-hand side is non-negative. But this follows from the classical central limit theorem, since
\begin{align*}
\frac{1}{n} \log \E\Bigg[e^{t_{\text{min}}\widetilde S_n}\mathbbm 1_{\{\widetilde S_n \geq 0\}}\Bigg] & \geq \frac{1}{n} \log \E \Bigg[ e^{-t_{\text{min}}\widetilde S_n} \mathbbm 1_{\{0 \leq \widetilde S_n \leq 1 \}}\Bigg] \cr
& \geq -\frac{t_{\text{min}} \sqrt{n}}{n} + \frac{1}{n} \log \Pro\Bigg[\frac{\widetilde S_n}{\sqrt{n}} \in[0,1]\Bigg] 
\,\, \stackrel{n\to\infty}{\longrightarrow} 0,
\end{align*}
which then completes the proof.
\end{proof}

A natural question that arises is the one of what happens if we are not in the luxurious situation where we are working with random variables having all exponential moments. What happens, for instance, if we are faced with a sequence of heavy tailed random variables (those naturally appear, e.g., when studying large deviations related to $\ell_p$-balls)? As it turns out, we can still prove a version of Cram\'er's theorem, also known as \emph{Cram\'er's theorem for random variables with stretched exponential tails}.

\begin{thm}[Cram\'er for stretched exponential tails]\label{thm:cramer stretched}
Consider iid and integrable random variables $X_1, X_2,\dots$. Assume that there exist constants $r\in(0,1)$, $c\in(0,\infty)$, and $0<C_1<C_2<\infty$ such that, for all $t\geq 0$,
\begin{equation}\label{eq:stretched exponential}
C_1 e^{-ct^r} \leq \Pro[X_1 \geq t] \leq C_2 e^{-ct^r}.
\end{equation}
Then, for any $a>\E[X_1]$, we have that
\[
\lim_{n\to\infty} \frac{1}{n^r} \log \Pro[S_n \geq an ] = -c\big(a-\E[X_1]\big)^r.
\]
\end{thm}

The property of stretched exponential tails is expressed through \eqref{eq:stretched exponential}. In particular, one can easily check that the lower bound in \eqref{eq:stretched exponential} guarantees that $\E[e^{tX_1}]=\infty$ for all $t>0$. Note that the speed is $n^{r}$, which is much slower than the speed $n$ in Cram\'er's theorem (see Theorem \ref{thm:cramer}). Moreover, the rate function on the right-hand side is obviously non-convex, i.e., it does not arise as a Legendre transform. Since a proof is not contained in the standard literature, we provide at least a sketch here.

\begin{proof}[Idea of Proof of Theorem \ref{thm:cramer stretched}.]
Let $a>0$. Without loss of generality we may assume that $\E[X_1]=0$. We shall start with the lower bound, where one uses the philosophy that a sum of iid random variables with stretched exponential tails becomes large most easily if one of the random variables does, while all others have `normal' size. 
\vskip 2mm
\noindent \textit{Lower bound:} Let $\varepsilon>0$. Then, for each $n\in\N$, we obtain from
\[
X_1 \geq n(a+\varepsilon) \qquad\text{and}\qquad \sum_{i=2}^nX_i \geq -(n-1)\varepsilon
\]
that
\[
S_n = \sum_{i=1}^n X_i \geq n(a+\varepsilon) - (n-1)\varepsilon = an + \varepsilon.
\]
Hence, for all $n\in\N$,
\begin{align*}
\Pro[S_n \geq a n] 
& = \Pro[X_1 \geq n(a+\varepsilon)] \cdot \Pro\Big[\frac{S_{n-1}}{n-1} \geq -\varepsilon\Big],
\end{align*}
where we used the iid property of the sequence $X_1,X_2,\dots$. As a consequence of the weak law of large numbers,
\begin{equation}\label{eq:weak law in cramer without exp moments}
\Pro\Big[\frac{S_{n-1}}{n-1} \geq -\varepsilon\Big] \stackrel{n\to \infty}{\longrightarrow} 1.
\end{equation}
Thus, the lower bound in \eqref{eq:stretched exponential} yields that, for any $n\in\N$,
\begin{align*}
\frac{1}{n^r}\log \Pro[S_n \geq a n] 
& \geq  
\frac{1}{n^r}\log C_1 - c(a+\varepsilon)^r + \frac{1}{n^r}\log \Pro\Big[\frac{S_{n-1}}{n-1} \geq -\varepsilon\Big]\,.
\end{align*}
This together with \eqref{eq:weak law in cramer without exp moments} shows that
\[
\liminf_{n\to\infty} \frac{1}{n^r} \log \Pro[S_n \geq a n] \geq -c(a+\varepsilon)^r,
\]
which completes the proof of the lower bound for $\varepsilon\downarrow 0$.
\vskip 2mm
\noindent \textit{Upper bound:} The upper bound requires tedious computations, which we shall shorten considerably, leaving it to the reader to fill in the details. First, we note that by the very same philosophy already used in the lower bound, for all $n\in\N$,
\begin{align*}
\Pro[S_n \geq a n] & = \Pro\Big[S_n \geq a n\,,\, \max_{1\leq i \leq n}X_i \geq a n \Big] + \Pro\Big[S_n \geq a n\,,\, \max_{1\leq i \leq n}X_i < a n \Big] \cr
& \leq \Pro\Big[\max_{1\leq i \leq n} X_i \geq a n \Big] + \Pro\Big[S_n \geq  a n\,,\, \max_{1\leq i \leq n}X_i < a n\Big].
\end{align*}
We will show that the second summand does not dominate the first on the scale $n^r$.

Because of the identical distribution of the random variables, for the first summand we have
\[
\Pro\Big[\max_{1\leq i \leq n} X_i \geq a n \Big] \leq \sum_{i=1}^n \Pro[X_i \geq a n] = n \Pro[X_1 \geq a n].
\]
In combination with the upper bound in \eqref{eq:stretched exponential} this yields, for any $n\in\N$,
\begin{align*}
\frac{1}{n^r} \log \Pro\Big[\max_{1\leq i \leq n} X_i \geq a n \Big] 
& \leq  
\frac{1}{n^r}\log n + \frac{1}{n^r} \log C_2 - ca^r
\end{align*}
and so see that
\begin{equation}\label{eq:upper bound for max cramer no exp moments}
\limsup_{n\to\infty} \frac{1}{n^r} \log \Pro\Big[\max_{1\leq i \leq n} X_i \geq a n \Big] \leq -ca^r.
\end{equation}
Using the simple fact that, for all $N\in\N$ and every sequence $a_1(n),\dots,a_N(n) \geq 0$, $n\in\N$, 
\[
\limsup_{n\to\infty} \, \frac{1}{n^r} \,\log \sum_{i=1}^Na_i(n)
= \max_{1\leq i \leq N} \limsup_{n\to \infty}\frac{1}{n^r}\, \log a_i(n),
\]
we obtain for $N=2$ in combination with \eqref{eq:upper bound for max cramer no exp moments} that
\begin{align*}
\limsup_{n\to\infty} \frac{1}{n^r} \log \Pro[S_n \geq a n] 
& \leq \max\bigg\{ -ca^r\,,\, \limsup_{n\to\infty}\frac{1}{n^r}\log \Pro\Big[S_n \geq  a n\,,\, \max_{1\leq i \leq n}X_i < a n\Big]\bigg\}\,.
\end{align*}
Hence, it is enough to prove that the second entry in the maximum does not dominate the first. Using Markov's inequality, we obtain, for every $\alpha >0 $,
\begin{align*}
\E\bigg[ \prod_{i=1}^{n}e^{\alpha X_i}\mathbbm 1_{\{X_i< n a \}}\bigg]
&\geq 
e^{\alpha n a} \, \Pro\Big[S_n\geq  a n\,,\, \max_{1\leq i \leq n} X_i < a n \Big].
\end{align*} 
Equivalently, for every $\alpha>0$, we have
\[
\Pro\Big[S_n\geq  a n\,,\, \max_{1\leq i \leq n} X_i < a n \Big] \leq e^{-\alpha n a}\, \E\bigg[ \prod_{i=1}^{n}e^{\alpha X_i}\mathbbm 1_{\{X_i< n a \}}\bigg]
\]
and so the iid assumption yields that
\begin{align*}
\Pro\Big[S_n\geq  a n\,,\, \max_{1\leq i \leq n} X_i < a n \Big] 
& \leq 
e^{-\alpha n a}  \bigg( \E\Big[ e^{\alpha X_1}\mathbbm 1_{\{X_1< \alpha^{-1} \}}\Big] + \E\Big[ e^{\alpha X_1}\mathbbm 1_{\{\alpha^{-1}\leq X_1< a n \}}\Big] \bigg)^n.
\end{align*}
We will now show that for appropriate choice of $\alpha$, the factor $e^{-\alpha n a}$ has the desired rate $-ca^r$. 
\vskip 2mm
\emph{The first expectation:} Using that for all $x<1$, $e^x \leq 1 + x +x^2$, and that, for all $x\in\R$, $1+x \leq e^x$, we obtain from a direct computation that  (since $\E[X_1]=0$ and $\sigma^2=\E[X_1^2]$)
\begin{align}\label{eq:abschaetzung erster erwartungswert cramer ohne exp moments}
\E\Big[ e^{\alpha X_1}\mathbbm 1_{\{X_1< \alpha^{-1} \}}\Big]
&\leq 
e^{\alpha^2 \sigma^2}\,.
\end{align}

\vskip 2mm
\emph{The second expectation:} It follows from  a tedious computation involving integration by parts for Lebesgue-Stieltjes integrals that
\[
\E[e^{\alpha \xi}\mathbbm 1_{\{a \leq \xi \leq b \}}] = e^{\alpha a}\Pro[\xi\geq a] - e^{\alpha b}\Pro[\xi> b] + \alpha \int_a^b e^{\alpha x}\Pro[\xi>x]\,\dint x
\]
holds for any random variable $\xi$ and all $a<b$. In our setting, together with the upper bound in \eqref{eq:stretched exponential}, we obtain 
\begin{align*}
\E\Big[ e^{\alpha X_1}\mathbbm 1_{\{\alpha^{-1}\leq X_1< a n \}}\Big] 
& \leq 
C_2 \alpha \int_{\alpha^{-1}}^{a n} e^{\alpha x - cx^r}\,\dint x + C_2e^{1-c\alpha^{-r}},
\end{align*} 
with absolute constant $C_2\in(0,\infty)$.
Now we choose $\alpha:= (c-\varepsilon)(an)^{r-1}$ with $\varepsilon\in(0,c)$. Then, for all $x\in[\alpha^{-1}, an]$, $\alpha x - c x^r \leq - \varepsilon x^r$ (where we used $r\in(0,1)$). Thus, for sufficiently large $n\in\N$,
\begin{align*}
\E\Big[ e^{\alpha X_1}\mathbbm 1_{\{\alpha^{-1}\leq X_1< a n \}}\Big]
& \leq C_2 \alpha \int_{\alpha^{-1}}^{a n} e^{ - \varepsilon x^r}\,\dint x + e^{-Cn^{r(1-r)}},
\end{align*}
where $C:=C(a,c,C_2,\varepsilon,r)\in(0,\infty)$. This means that, for sufficiently large $n\in\N$, 
\begin{align*}
\E\Big[ e^{\alpha X_1}\mathbbm 1_{\{\alpha^{-1}\leq X_1< a n \}}\Big]
 & \leq 
 C_2\alpha \frac{1}{\varepsilon^{1/r}r} \Gamma(r^{-1};\varepsilon\alpha^{-r}) + e^{-Cn^{r(1-r)}} ,
\end{align*}
where $\Gamma(\cdot;\cdot)$ is the incomplete gamma function. Using the asymptotic $\Gamma(x;y) \sim y^{x-1}e^{-y}$ for $y\to\infty$, we obtain for sufficiently large $n\in\N$ (note that $\alpha=\alpha(n)$) that
\begin{eqnarray}\label{eq:abschaetzung zweiter erwartungswert cramer ohne exp moments}
\E\Big[ e^{\alpha X_1}\mathbbm 1_{\{\alpha^{-1}\leq X_1< a n \}}\Big]
 \leq e^{-\widetilde{C}n^{r(1-r)}},
\end{eqnarray}
with $\widetilde C:=\widetilde C(a,c,C_2,\varepsilon,r)\in(0,\infty)$.

\vskip 2mm
\emph{Completion of the proof:}
We will use the bounds \eqref{eq:abschaetzung erster erwartungswert cramer ohne exp moments} and \eqref{eq:abschaetzung zweiter erwartungswert cramer ohne exp moments}. For the choice $\alpha= (c-\varepsilon)(an)^{r-1}$ and sufficiently large $n\in\N$, we obtain
\begin{align*}
\Pro\Big[S_n\geq  a n\,,\, \max_{1\leq i \leq n} X_i < a n \Big] 
& \leq 
e^{-\alpha n a} \bigg( e^{\alpha^2 \sigma^2} +  e^{-\widetilde{C}n^{r(1-r)}} \bigg)^n \cr
& \leq e^{-(c-\varepsilon)(an)^r} e^{C_1n^{2r-1}} \bigg( 1 +  e^{-\widetilde{C}n^{r(1-r)}} \bigg)^n,
\end{align*}
where $C_1:=C_1(\sigma,c,\varepsilon,a,r) \in(0,\infty)$. Since 
\[
\bigg( 1 +  e^{-\widetilde{C}n^{r(1-r)}} \bigg)^n \stackrel{n\to\infty}{\longrightarrow} 1,
\]
we have that, for large $n\in\N$,
\[
 e^{C_1n^{2r-1}} \bigg( 1 +  e^{-\widetilde{C}n^{r(1-r)}} \bigg)^n \leq e^{2C_1n^{2r-1}}.
\]
Noting that $2r-1< r$ and so $2C_1n^{2r-1} = o(n^r)$ for $n\to\infty$, we see that
\begin{align*}
\limsup_{n\to\infty}\frac{1}{n^r}\log \Pro\Big[S_n \geq  a n\,,\, \max_{1\leq i \leq n}X_i < a n\Big] & \leq \limsup_{n\to\infty} \frac{1}{n^r} \log\Big(e^{-(c-\varepsilon)(an)^r}e^{o(n^r)}\Big) = -(c-\varepsilon)a^r.
\end{align*}
Taking $\varepsilon\downarrow 0$ completes the proof.
\end{proof}

\begin{rmk}
We chose to state this theorem here, because later we shall use a more general version of such a result (where the constant can dependent on the parameter $t$), which is due to N.~Gantert, K.~Ramanan, and F.~Rembart \cite{GRR2014}. But the philosophy behind it is more easily explained in the situation considered in Theorem \ref{thm:cramer stretched}.
\end{rmk}

\subsection{Sanov's theorem for the empirical measure}\label{sec:sanovs ldp theorem}

Another fundamental theorem in large deviations theory, aside from Cram\'er's theorem, is Sanov's theorem. This result concerns so-called level--$2$ large deviations for the empirical measure of random variables and goes back to I.~N.~Sanov \cite{S1961}. We shall present here a more general and today standard form, following the presentation in \cite{RAS2015}.

Let $E$ be a Polish space which will always be endowed with the Borel $\sigma$-field. Consider $\lambda\in\mathscr M_1(E)$, where $\mathscr M_1(E)$ is the set of probability measures on $E$. Consider iid and $E$-valued random variables $X_1,X_2,\dots$ having distribution $\lambda$ and define the corresponding empirical measure via
\[
\mu_n := \frac{1}{n}\sum_{i=1}^n \delta_{X_i},\qquad n\in\N.
\]
This is a random probability measure on $E$, i.e., each $\mu_n$ is an random element of $\mathscr M_1(E)$. 
\vskip 1mm
\begin{rmk}
Let us recall some well known facts. Since $E$ is a Polish space, $\mathscr M_1(E)$ is a Polish space itself in the weak topology generated by $\mathscr C_b(E)$, the space of bounded continuous functions on $E$ taking values in $\R$.

$\mathscr M(E)$ will denote the space of signed measures on $E$ of which $\mathscr M_1(E)$ is a subspace. Then $\mathscr M(E)$ and $\mathscr C_b(E)$ are in duality via 
\[
\langle \nu,f \rangle = \int_E f\,\dint\mu
\]
for $\nu\in\mathscr M(E)$, $f\in\mathscr C_b(E)$. 

Convergence of probability measures in the weak topology $\sigma(\mathscr M(E),\mathscr C_b(E))$ is the same as the weak convergence $\nu_n\stackrel{w}{\longrightarrow} \nu$, which holds by definition if and only if $\langle \nu_n,f\rangle \rightarrow \langle \nu,f \rangle$ for all $f\in\mathscr C_b(E)$ \footnote{Recall that the weak topology on $\mathscr M(E)$, $\sigma(\mathscr M(E),\mathscr C_b(E))$, is the minimal topology under which all functions $\nu\mapsto \langle \nu , f, \rangle$, $f\in \mathscr C_b(E)$ are continuous.}.
\end{rmk}
\vskip 1mm
Sanov's theorem describes the large deviations behavior of the distributions 
\[
\rho_n:= \Pro\big[\mu_n \in B\big],\qquad B\subset \mathscr M_1(E)\,\,\text{Borel set},
\]
of $\mu_n$, $n\in\N$. Prior to proving an LDP it is good to develop a sense for a law of large numbers, which we shall do now before stating Sanov's theorem. First, we note that by the classical strong law of large numbers, for all $f\in\mathscr C_b(E)$,
\begin{align}\label{eq:mu_n of f slln}
\mu_n(f) & := \E_{\mu_n}[f] = \frac{1}{n}\sum_{i=1}^n f(X_i) \xrightarrow[n\to\infty]{\Pro-a.s.} \E[f(X_1)] = \E_\lambda[f].
\end{align}
Let us recall that a sequence $(g_i)_{i\in\N}$ is said to determine weak convergence on $\mathscr M_1(E)$ if and only if $\nu_n\stackrel{w}{\longrightarrow} \nu$ is equivalent to the condition that, for all $i\in\N$, $\langle \nu_n , g_i \rangle \rightarrow \langle \nu,g_i\rangle$. Now, since our Polish space is by definition separable, there exists a countable, convergence determining collection of $\mathscr C_b(E)$-functions. Let us denote those functions by $(f_i)_{i\in\N}$ with $f_i\in\mathscr C_b(E)$ for all $i\in\N$. Because of \eqref{eq:mu_n of f slln}, for every $i\in\N$,
\[
\mu_n(f_i) \xrightarrow[n\to\infty]{\Pro-a.s.} \E_\lambda[f_i].
\]
In particular, because $(f_i)_{i\in\N}$ determines weak convergence on $\mathscr M_1(E)$, we obtain the following strong law of large numbers for $(\mu_n)_{n\in\N}$:
\[
\mu_n \xrightarrow[n\to\infty]{\Pro-a.s.} \lambda.
\]
Now that we have developed an understanding of the typical asymptotic behavior of the sequence of empirical measures corresponding to $X_1,X_2,\dots$, we shall look at the asymptotic probabilities of rare events treated in Sanov's theorem.

\begin{thm}[Sanov's theorem]
Let $E$ be a Polish space and, for each $n\in\N$, let $\rho_n$ be the distribution of the empirical measure 
\[
\mu_n:= \frac{1}{n}\sum_{i=1}^n \delta_{X_i},
\]
corresponding to the iid sequence $(X_i)_{i\in\N}$ having distribution $\lambda\in\mathscr M_1(E)$. Then $(\rho_n)_{n\in\N}$ satisfies an LDP on $\mathscr M_1(E)$ at speed $n$ with good convex rate function $H(\cdot|\lambda) : \mathscr M_1(E) \to [0,\infty]$ being the relative entropy 
\[
H(\nu|\lambda) := 
\begin{cases}
\int_E f \log f \,\dint\lambda &: \nu \ll \lambda \,\text{ and }\, f=\frac{\dint \nu}{\dint \lambda}\\
\infty & : \text{otherwise}.
\end{cases}
\]
\end{thm}

In words, Sanov's theorem tells us that the probability that the empirical measures resemble some distribution $\nu$ other than the underlying distribution $\lambda$ behaves like $e^{-n H(\nu|\lambda)}$, i.e., the rate is determined by the entropy of $\nu$ relative to $\lambda$. We shall not provide a proof of Sanov's theorem here, but rather refer the interested reader to the nice exposition of the proof in \cite{RAS2015}.

\subsection{The Gärtner--Ellis theorem for dependent random variables}\label{sec:gaertner ellis theorem}

Before presenting the general theory of large deviations, let us describe a classical generalization of Cram\'er's theorem due to J.~G\"artner \cite{G1977} and R.~S.~Ellis \cite{E1984}, which shows that we may allow random sequences with a kind of moderate dependence, and still have LDP behavior. More precisely, let us consider random vectors $(Y_n)_{n\in\N}$ taking  values in $\R^d$ and having moment generating functions
  \[
    \lambda_n(t) := \E\Big[e^{\langle t,Y_n \rangle} \Big], \quad t\in\R^d, \, n\in\N;
  \]
here $\langle\cdot,\cdot\rangle$ denotes the standard inner product on $\R^d$.  We shall assume throughout this subsection that the following two conditions are satisfied:
\vskip 2mm
(GE1)  $\lim\limits_{n\to\infty}\frac{1}{n}\log\big(\lambda_n(nt)\big) = \Lambda(t)\in[-\infty,\infty]$ exists,
\vskip 1mm
(GE2) $0\in\mathcal D_{\Lambda}^{\circ}$, with $\mathcal D_{\Lambda} := \{t\in\R^d\,:\, \Lambda(t)<\infty \}$.
\vskip 1mm
Under the previous assumptions, it is rather easy to show that $\Lambda$ is convex and satisfies $\Lambda>-\infty$. Moreover, the Legendre transform $\Lambda^*(\cdot):= \sup_{t\in\R^d}[ \langle \cdot,t\rangle - \Lambda(t) ]$ is a convex good rate function (i.e., not everywhere equal to $\infty$ with compact sublevel sets). 

Let us recall that a point $x\in\R^d$ is called an exposed point for $\Lambda^*$ if and only if there exists a point $t\in\R^d$ such that, for all $\R^d\ni y \neq x$,
  \[
    \Lambda^*(y) - \Lambda^*(x) > \langle y-x,t \rangle.
  \]
This $t$ is the normal vector to an exposing hyperplane for the point $x\in\R^d$. Note that, since $\Lambda^*$ is convex, for each $x\in\mathcal D_{\Lambda^*}$ there exists $t\in\R^d$ such that $\Lambda^*(y)-\Lambda^*(x) \geq \langle t, y-x\rangle$ for all $y\in\R^d$, i.e., a supporting hyperplane at the point $(x,\Lambda^*(x))$ of the graph of $\Lambda^*$. Thus, at an exposed point there exists a supporting hyperplane sharing no other point with the graph.

\begin{thm}[Gärtner--Ellis theorem]\label{thm:gaertnerellis}
Let $(Y_n)_{n\in\N}$ be a sequence of random vectors in $\R^d$ such that (GE1) and (GE2) are satisfied, and let $\mu_n := \Pro^{Y_n}$, $n\in\N$. Then, for all closed sets $C\subset \R^d$,
  \[
    \limsup_{n\to\infty}\frac{1}{n} \log\big(\mu_n(C)\big) \leq -\inf_{x\in C}\Lambda^*(x),
  \]
and, for all open sets $O\subset \R^d$,
  \[
    \liminf_{n\to\infty}\frac{1}{n} \log\big(\mu_n(O)\big) \geq  -\inf_{x\in O\cap E}\Lambda^*(x),
  \]
where $E$ is the collection of exposed points of $\Lambda^*$ whose exposing hyperplane belongs to $\mathcal D_{\Lambda}^{\circ}$. 

Suppose that, in addition, $\Lambda$ is lower semi-continuous on $\R^d$, differentiable on $\mathcal D_{\Lambda}^{\circ}$, and that $\mathcal D_{\Lambda}=\R^d$ or $\Lambda$ is steep on the boundary $\partial\mathcal D_{\Lambda}$ of $\mathcal D_{\Lambda}$ (meaning that $\lim\limits_{t\rightarrow\partial D_\Lambda \colon t\in D_\Lambda} \lVert \nabla \Lambda(t) \rVert_2 = \infty$). Then we may replace $O\cap E$ by $O$. In large deviations parlance, the sequence $(\mu_n)_{n\in\N}$ satisfies an LDP on $\R^d$ with speed $n$ and rate function $\Lambda^*$.
\end{thm}

The steepness assumption may be removed in the iid setting \cite[Corollary 6.1.6]{DZ2010}, but is generally required for dependent random sequences. Moreover, one can generalize the result to topological vector spaces and we refer to \cite[Section 4.5]{DZ2010}. For a more detailed discussion, we refer the interested reader to \cite{DZ2010} and \cite{dH2000}; it is also discussed there how the Gärtner--Ellis theorem implies both Cram\'er's and Sanov's theorem.

\subsection{The abstract theory of large deviations}

As already mentioned, a general theory of large deviations, going far beyond the ideas behind Cram\'er's theorem, was developed by S.~Varadhan in \cite{V1966}. So let us explore on what principles this unified theory is build upon. 

\begin{df}[Rate function and LDP]
Let $\mathscr X$ be a topological Hausdorff space and $\mathscr B_{\mathscr X}$ the Borel $\sigma$-field on $\mathscr X$. A rate function $\rate:\mathscr X\to[0,\infty]$ is a lower-semicontinuous mapping with $\rate \not\equiv \infty$, which is said to be a good rate function if the level sets $\{ x\in\mathscr X\,:\, \rate(x) \leq a \}$, $a\in[0,\infty)$ are not only closed, but compact. The effective domain of the rate function is given by
\[
\mathscr D_{\rate} = \big\{ x\in\mathscr X\,:\, \rate(x)<\infty \big\}. 
\] 
Let $(\nu_n)_{n\in\N}$ be a sequence of probability measures on $\mathscr B_{\mathscr X}$. Further, let $(s_n)_{n\in\N}$ be a sequence of positive reals
with $s_n\uparrow\infty$ and $\rate:\mathscr X\to[0,\infty]$ be a lower semi-continuous function.
We say that $(\nu_n)_{n\in\N}$ satisfies a (full) large deviation principle (LDP) at speed $s_n$ and with rate function $\rate$ if and only if
\begin{equation}\label{eq:LDPdefinition}
\begin{split}
-\inf_{x\in A^\circ}\rate(x)
\leq\liminf_{n\to\infty}{1\over s_n}\log\nu_n(A)
\leq\limsup_{n\to\infty}{1\over s_n}\log\nu_n(A)\leq -\inf_{x\in\overline{A}}\rate(x) 
\end{split}
\end{equation}
for all Borel sets $A\subset \mathscr X$, where $A^\circ$ and $\overline{A}$ denote the interior and closure of $A$, respectively.
We say that $(\nu_n)_{n\in\N}$ satisfies a weak LDP with speed $s_n$ and rate function $\rate$ if the rightmost upper bound in \eqref{eq:LDPdefinition} is valid only for compact sets $A\subset \mathscr X$.
\end{df}

\begin{rmk}
The notion of an LDP for a sequence of random variables is obtained by considering the LDP for the corresponding sequence of distributions.
\end{rmk}
\begin{rmk}
The estimates in \eqref{eq:LDPdefinition}, describing the LDP, can also be reformulated in terms of an upper and lower large deviations bound holding for closed and open sets respectively, which is very much in the spirit of the bounds in the Gärtner--Ellis theorem (see Theorem \ref{thm:gaertnerellis} above).  
\end{rmk}
\begin{rmk}\label{rem:good rate function}
Let us remark that the rate function controlling the LDP is always unique.
A consequence of the rate function being good is that it attains its infimum on closed sets. Moreover, since $\nu_n(\mathscr X)=1$ for any $n\in\N$, for the upper bound in \eqref{eq:LDPdefinition} to hold it is necessary that $\inf_{x\in\mathscr X}\rate(x) = 0$. So when $\rate$ is a good rate function, then there must exist an element $x\in\mathscr X$ such that $\rate(x)=0$.

Note that on metric spaces $M$, one may check the lower-semicontinuity on sequences, i.e., $\rate$ is lower-semicontinuous if and only if, for all $x\in M$,
\[ 
\liminf_{x_n\to x}\,\rate(x_n) \geq \rate(x). 
\]
\end{rmk}


\begin{rmk}
In view of the Portmanteau theorem, this definition reminds us of the weak convergence of measures. However, to the uninitiated reader this definition only mildly resembles what one might expect after having seen Cram\'er's theorem. Naively one is tempted to think of a definition like
\[
\lim_{n\to\infty} \frac{1}{s_n} \log \nu_n(A) = -\inf_{x\in A}\rate(x).
\]
As it turns out, this is far to restrictive to allow us to build a rich and powerful theory. Indeed, note that if $(\nu_n)_{n\in\N}$ is a sequence of non-atomic measures, then $\nu_n(\{x\})=0$ for any $x\in\mathscr X$. But this means that $\rate \equiv \infty$, which is not allowed in the definition of a rate function. So we really need some topological restrictions in the definition of an LDP.
\end{rmk}

As we shall see later, when discussing the work on large deviation principles in the asymptotic theory of geometric functional analysis, one way of proving the large deviation estimates in \eqref{eq:LDPdefinition} is to first prove the upper bound for compact sets or, in other words, to first prove a weak LDP. Heuristically, this can work if most of the mass in $\mathscr X$ is concentrated (on an exponential scale) on compact sets.

\begin{df}[Exponential tightness]
Let $\mathscr X$ be a Hausdorff topological space with Borel $\sigma$-field $\mathscr B_{\mathscr X}$ and assume that $(s_n)_{n\in\N}$is a sequence of positive reals with $s_n\uparrow\infty$. We say that a sequence $(\nu_n)_{n\in\N}$ of probability measures on $\mathscr B_{\mathscr X}$ is exponentially tight (at scale $s_n$) if and only if for every $C\in(0,\infty)$ there exists a compact set $K_C\subset \mathscr X$ such that
\[
\limsup_{n\to\infty} \frac{1}{s_n}\log \nu_n(K_C^c) < C.
\]
\end{df}

\begin{rmk}\label{rem:exponential tightness}
Let us note that exponential tightness of a sequence of measures is not necessary for the sequence to satisfy an LDP with a good rate function. In nice situations however, if we are working with locally compact Hausdorff topological spaces $\mathscr X$ or with Polish spaces (i.e., a separable completely metrizable topological space), then goodness of the rate function implies exponential tightness (see, e.g., \cite{DZ2010}). 
\end{rmk} 

As it turns out, the notion of exponential tightness is the key to lifting a weak LDP to a full one (see, e.g., \cite[Lemma 1.2.18]{DZ2010}). At the same time, this guarantees us a good rate function in the LDP.

\begin{lemma}
Let $\mathscr X$ be a Hausdorff topological space with Borel $\sigma$-field $\mathscr B_{\mathscr X}$ and $(\nu_n)_{n\in\N}$ a sequence of probability measures on $\mathscr B_{\mathscr X}$. Suppose that $(\nu_n)_{n\in\N}$ satisfies a weak LDP at speed $s_n$ and with rate function $\rate$. If $(\nu_n)_{n\in\N}$ is exponentially tight at scale $s_n$, then it satisfies a full LDP at speed $s_n$ and with rate function $\rate$. In particular, in this case the rate function $\rate$ is a good rate function.  
\end{lemma}

\subsubsection{Three fundamental principles}

There are several useful tools in the theory of large deviations that allow one to obtain an LDP on one space by a transformation of a known LDP on some other. Such transformations seem rather innocent at first, but are not to be underestimated in their power. Note that we will refrain from providing any examples here, because we shall see all principles in action later when presenting some of the works on large deviations elaborated upon in the introduction.

The first stability result concerns the preservation of LDPs under continuous mappings and is commonly known as the contraction principle. Since the proof is quite short, we will present it here.

\begin{lemma}[Contraction principle]\label{lem:contraction principle}
Let $\mathscr X$ and $\mathscr Y$ be Hausdorff topological spaces and let $T:\mathscr X \to \mathscr Y$ be a continuous mapping. Denote their respective Borel $\sigma$-fields by $\mathscr B_{\mathscr X}$ and $\mathscr B_{\mathscr Y}$. If $(\nu_n)_{n\in\N}$ is a sequence of probability measures on $\mathscr B_{\mathscr X}$ satisfying an LDP at speed $s_n$ with a good rate function $\rate:\mathscr X\to[0,\infty]$, then the sequence of pushforwards
\[
\mu_n := \nu_n \circ T^{-1},\qquad n\in\N
\]
satisfies an LDP on the space $\mathscr Y$ at the same speed $s_n$ and with good rate function $\mathbb J:\mathscr Y\to [0,\infty]$
\[
\mathbb J(y) := \inf_{x\in\mathscr X, T(x)=y} \rate(x),
\]
where we agree that the infimum over the empty set is equal to $\infty$.
\end{lemma}
\begin{proof}
Let us first argue that the function $\mathbb J$ is indeed a good rate function.
We start by showing that $\mathbb J\not\equiv \infty$. Since $\rate$ is a (good) rate function,
$
\mathscr D_{\rate} = \big\{x\in \mathscr X\,:\, \rate(x)<\infty \big\}\neq \emptyset.
$ 
And so it follows from the definition of $\mathbb J$ that
$
\mathscr D_{\mathbb J} = \{y\in \mathscr Y\,:\, \mathbb J(y)<\infty \}\neq \emptyset,
$
hence $\mathbb J\not\equiv \infty$. We now show that the function $\mathbb J$  is good. Clearly, $\mathbb J$ is non-negative and, in view of Remark \ref{rem:good rate function} and the continuity of $T$, attains its infimum for every $y\in T(\mathscr X)$ at some point in $\mathscr X$. We shall use this to show that the level sets are compact. Let $a\in[0,\infty)$. Since $\rate$ is good, 
$
\{ x\in \mathscr X\,:\, \rate(x)\leq a \}
$
is a compact set in $\mathscr X$. Therefore, since $T$ is continuous,
$
T\big(\{ x\in E\,:\, \rate(x)\leq a \}\big)
$
is a compact set in $\mathscr Y$. But since
\begin{align*}
T\big(\big\{ x\in \mathscr X\,:\, \rate(x)\leq a \big\}\big) & = \big\{T(x)\,:\, x\in \mathscr X,\,\rate(x)\leq a\big \} 
 = \big\{y\in \mathscr y\,:\, \exists x\in \mathscr X: T(x)=y\,\,\land\,\, \rate(x)\leq a \big\} \cr
& = \big\{y\in \mathscr Y\,:\, \inf_{x\in T^{-1}(y)}\rate(x)\leq a\big \} 
 = \big\{y\in \mathscr Y\,:\, \mathbb J(y) \leq a \big\},
\end{align*}
which shows the compactness of level sets. We used that $\mathbb J$ attains its minimum on closed sets in the equality from line one to line two.
\end{proof}

The next principle frequently used, often in combination with the contraction principle, is the following one describing the large deviations behavior of pairs of independent random objects. We shall restrict ourselves to the case of Polish spaces, enough for the purpose of this paper. However, the result can be formulated and proved for regular topological spaces (see, e.g., \cite[Exercise 4.2.7]{DZ2010}), in which case one needs to additionally assume the exponential tightness of measures; by Remark \ref{rem:exponential tightness} this follows from the goodness of rate functions in Polish spaces.

\begin{lemma}\label{lem: ldp pairs of independent random objects}
Let $E$ be a Polish space. Assume that $(\nu_n)_{n\in\N}$ and $(\mu_n)_{n\in\N}$ are sequences of probability measures on $E$. Suppose that for all $n\in\N$, $(X_n,Y_n)$ is distributed according to the product measure $\nu_n\otimes \mu_n$ on $\mathscr B_{E} \otimes \mathscr B_{E}$ (i.e., $X_n$ and $Y_n$ are independent). Assume that $(\nu_n)_{n\in\N}$ satisfies an LDP at speed $s_n$ with good rate function $\rate_X$ and $(\mu_n)_{n\in\N}$ an LDP at speed $s_n$ with good rate function $\rate_Y$. Then,\\
(a) 
$(X_n,Y_n)$, $n\in\N$, satisfies an LDP on $E\times E$ at speed $s_n$ with good rate function $\mathbb J: E\times E \to [0,\infty]$,
\[
\mathbb J(x,y) := \rate_X(x) + \rate_Y(y).
\] 
(b) If $T:E\times E \to F$ is a continuous mapping into another Polish space $F$, then the sequence $Z_n:=T(X_n,Y_n)$, $n\in\N$, satisfies an LDP on $F$ at speed $s_n$ with good rate function $\rate_Z: F \to [0,\infty]$,
\[
\rate_Z(z) = \inf_{(x,y)\in E\times E\atop{T(x,y)=z}} \rate_X(x) + \rate_Y(y). 
\]
\end{lemma} 

The last principle we shall present here concerns the possibility of transferring an LDP from one sequence of random variables to another as long as those sequences are close on an exponential scale. Let $M$ be a metric space with distance $\dint$, equipped with its Borel $\sigma$-field $\mathscr B_{M}$. We say that two sequences of $M$-valued random variables $(X_n)_{n\in\N}$ and $(Y_n)_{n\in\N}$ are exponentially equivalent at speed $s_n$ if and only if, for all $\delta\in(0,\infty)$,
\[
\limsup_{n\to\infty} \frac{1}{s_n} \log\Pro\big[\dint(X_n,Y_n) > \delta \big] = -\infty.
\]
The next result shows that, as far as the LDP is concerned, exponentially equivalent random variables cannot be distinguished (see, e.g., \cite[Theorem 4.2.13]{DZ2010}).

\begin{lemma}\label{lem:exponential equivalence same ldp}
Let $M$ be a metric space equipped with its Borel $\sigma$-field $\mathscr B_{M}$ and assume that $(X_n)_{n\in\N}$, $(Y_n)_{n\in\N}$ are sequences of $M$-valued random variables. If $(X_n)_{n\in\N}$ satisfies an LDP at speed $s_n$ with good rate function $\rate$ and $(X_n)_{n\in\N}$ and $(Y_n)_{n\in\N}$ are exponentially equivalent, then $(Y_n)_{n\in\N}$ satisfies an LDP at the same speed $s_n$ with the same good rate function $\rate$.
\end{lemma}  

\section{Large deviations results and techniques in the world of $\ell_p^n$-spaces}\label{sec:ldps in lp spaces}

In this section we shall present some of the large deviations results that have been obtained within the framework of $\ell_p^n$-balls, of course starting with the inspiring work of N.~Gantert, S.~S.~Kim, and K.~Ramanan \cite{GKR2017} on random one-dimensional projections of $\ell_p^n$-balls. As we have already mentioned in the introduction, in this setting, computations are facilitated by a probabilistic representation of the uniform distribution on an $\ell_p^n$-ball. More precisely, G. Schechtman and J. Zinn proved in  \cite{SchechtmanZinn} that a random vector $X$ uniformly distributed on $\B_p^n$, $1\leq p < \infty$, has a representation
\begin{align}\label{eq:schechtman-zinn representation} 
X \stackrel{\dint}{=} U^{1/n}\frac{(Z_1,\dots,Z_n)}{\|(Z_1,\dots,Z_n)\|_p},
\end{align}
where $U$ is uniformly distributed on $[0,1]$ and independent of $Z_1,\dots,Z_n$, which are iid $p$-generalized Gaussians, i.e., they have Lebesgue density
\begin{equation}\label{eq:pgaussian density}
f_p: \R\to[0,\infty),\qquad x\mapsto \frac{1}{2p^{1/p}\Gamma(1+1/p)} e^{-|x|^p/p}.
\end{equation}
This representation is powerful as it allows one to work with a random vector having independent coordinates; in contrast to the coordinates of $X$, which are obviously dependent. We refer the reader to the survey \cite{PTT2020} for a discussion of the probabilistic approach to the geometry of $\ell_p^n$-balls.

\subsection{Large deviations for random projections of $\ell_p^n$-balls}\label{subsec:ldps projections lp spaces}

In \cite{GKR2017}, N.~Gantert, S.~S.~Kim, and K.~Ramanan started analyzing distributional properties of the projections of high-dimensional random vectors in $\ell_p^n$-balls on the scale of large deviations. On the normal deviations scale, a central limit theorem for projections of convex bodies, i.e., for $X^{(n)}$ sampled from a log-concave measure (for instance the uniform measure on a convex body) that is also isotropic, had already been obtained by B.~Klartag \cite{K2007},  showing that for sufficiently large $n\in\N$ and most $\theta^{(n)}\in\SSS_2^{n-1}$, the projection of $X^{(n)}$ onto $\theta^{(n)}$ satisfies $  \langle X^{(n)}, \theta^{(n)} \rangle \approx \mathscr N(0,1)$ in some suitable sense.

We shall look at the results obtained in the annealed case, where a quite different large deviations behavior is revealed depending on whether $p< 2$ or $p\geq 2$. In particular, this shows an interesting sensitivity with respect to the underlying distribution of random vectors, which can be attributed to the stark difference in the geometry of $\ell_p^n$-balls in those two regimes. More precisely, the authors showed that if for each $n\in\N$, $\theta^{(n)}\in\SSS^{n-1}$ is a uniform random direction and $X^{(n,p)}$ is an independent random point uniformly distributed in the $\ell_p^n$-ball for some fixed $p\in[1,\infty]$, then the sequence of rescaled random variables (the normalized scalar random projections)
$$
W^{(n,p)}:=n^{{1\over p}-{1\over 2}}\langle X^{(n,p)},\theta^{(n)}\rangle = \frac{1}{n} \sum_{i=1}^n n^{1\over p} X_i^{(n,p)} n^{1\over 2}\theta^{(n)}_i,\quad n\in\N,
$$
satisfies an LDP with speed $n^{2p\over 2+p}$ if $p\in[1,2)$ and speed $n$ if $p\in[2,\infty]$, and with a certain rate function that also depends on the parameter $p$; for $p=\infty$, we use the convention $n^{1\over p}=1$. Below we only discuss the results for $p<\infty$ and refer to \cite[Section 8]{GKR2017} for a compilation of results for general product measures satisfying certain tail conditions (including the uniform measure on $\B_{\infty}^n$).

\subsubsection{Annealed LDP -- the case $p\in[2,\infty) $}

For $p\in[2,\infty)$, denoting by $\mu_p\in\mathcal M_1(\R)$ the $p$-Gaussian measure having Lebesgue density given by \eqref{eq:pgaussian density}, we consider the function
  \[
    \Phi_p(t_0,t_1,t_2) := \log \Bigg( \int_\R \int_\R e^{t_0z^2 + t_1zy + t_2|y|^p}\,\mu_2(\dint z)\mu_p(\dint y)  \Bigg), \qquad t_0,t_1,t_2\in\R, 
  \]
which satisfies $ \Phi_p(t_0,t_1,t_2)<\infty$ whenever $t_0<{1\over 2}$, $t_1\in\R$, and $t_2<{1\over p}$. The rate function that describes the large deviations behavior is given by the Legendre transform of $\Phi_p$, i.e., we define
  \[
    \rate_p(x) := \inf_{\tau_0>0,\tau_1\in\R,\tau_2>0\atop{\tau_0^{-1/2}\tau_1\tau_2^{-1/p} = x}} \Phi^*(\tau_0,\tau_1,\tau_2);
  \]
the infimum that appears already hints at an application of the contraction principle, i.e., the desired LDP will be contracted from another one. Let us now state the LDP for random one-dimensional projections corresponding to \cite[Theorem 2.2]{GKR2017}.

\begin{thmalpha}\label{thm:ldp annealed one dimensional projections p greater 2}
Let $p\in[2,\infty)$. Then the sequence of random variables $(W^{(n,p)})_{n\in\N}$ satisfies an LDP on $\R$ with speed $n$ and quasi-convex, symmetric, good rate function $\rate_p$.
\end{thmalpha}
\begin{proof}[Idea of Proof.]
First of all, for the simple and not quite so interesting case $p=2$, we refer to \cite[Theorem 2.12]{GKR2017}; the rate function is then given by $x\mapsto -\frac{1}{2}\log(1-x^2)$ for $x\in(-1,1)$ and is $\infty$ otherwise. 

So assume now that $p\in(2,\infty)$. Using the specific form of the $p$-Gaussian measures (see \eqref{eq:pgaussian density}), one easily checks that for $t_0<{1\over 2}$, $t_1\in\R$, and $t_2<{1\over p}$,
  \[
    \Phi_p(t_0,t_1,t_2) = -\frac{1}{p}\log(1-pt_2) - \frac{1}{2}\log(1-2t_0) + \log\Bigg( \int_\R e^{\frac{1}{2}t_1^2(1-pt_2)^{-2/p}(1-2t_0)^{-1}y^2}\,\mu_p(\dint y) \Bigg),
  \]
  which is finite, since $p>2$. Thus, the effective domain $\mathscr D_{\Phi_p}=\{ z\in\R^3\,:\, \Phi_p(z)<\infty \}$ of $\Phi_p$ satisfies
  \[
    \mathscr D_{\Phi_p}^\circ = \Big(-\infty,\frac{1}{2}\Big) \times \R \times \Big(-\infty,\frac{1}{p}\Big) \ni 0.
  \]
  It now follows from a general version of Cram\'er's theorem \cite[Theorem 2.2.30]{DZ2010} (which applies to the multivariate setting if $0$ is an interior point of the effective domain of the cumulant generating function of the considered random vectors) that the sequence 
  \[
    S^{(n,p)} := \frac{1}{n} \sum_{i=1}^n \big(|Z_i^{(n)}|^2, Y_i^{(n,p)}Z_i^{(n)}, |Y_i^{(n,p)}|^p  \big), \quad n\in\N,
  \]
with 
  \[
    Y^{(n,p)}= (Y^{(n,p)}_1,\dots,Y^{(n,p)}_n) \text{ having iid coordinates with common distribution $\mu_p$} 
  \]
 and 
   \[
     Z^{(n)}=(Z^{(n)}_1,\dots,Z^{(n)}_n) \text{ being a vector of independent standard Gaussians},
   \]
satisfies an LDP in $\R^3$ with speed $n$ and good rate function given by the Legendre transform $\Phi^*$; note that $\Phi_p$ is the log-moment generating function of the summand in the definition of $S^{(n,p)}$. The desired LDP will now be contracted from the latter one via the continuous mapping $T_p:(0,\infty)\times \R \to \R$, $T_p(\tau_0,\tau_1,\tau_2):= \tau_0^{-1/2}\tau_1\tau_2^{-1/p}$. More precisely, one readily checks that
  \[
    T(S^{(n,p)}) = \frac{n^{1/p}}{n^{1/2}} \frac{\sum_{i=1}^nY^{(n,p)}_iZ_i^{(n)}}{\| Y^{(n,p)} \|_p\|Z^{(n)}\|_2},
  \]
and so one obtains an LDP for the sequence $T(S^{(n,p)})$, $n\in\N$, with speed $n$ and rate function $\rate_p$. Now it is only left to show that the LDP for the latter sequence carries over (with identical speed and rate) to our sequence $W^{(n,p)}$, $n\in\N$. For this we refer to \cite[Lemma 3.4]{GKR2017}.
\end{proof}

\begin{rmk}
We refer the reader to \cite[Theorem 1.1]{APT2018} for the case of high-dimensional projections discussed in the introduction.
\end{rmk}

\subsubsection{Annealed LDP -- the case $p\in[1,2) $}

As we shall see now, when we are in the case $p\in[1,2)$, the random projections display a completely different large deviations behavior. For $p\in[1,2)$, we define the function $\mathbb J_p:\R \to [0,\infty)$ via
  \[
  \mathbb J_p (x):= \frac{1}{r_p} |x|^{r_p} ,
  \]
with $r_p:= \frac{2p}{2+p}$. It is important to note that $r_p<1$ for $p<2$ and so the following LDP occurs at a different speed $n^{r_p}$, which is slower than the speed $n$ in the LDP for $p\geq 2$. The following result corresponds to  \cite[Theorem 2.3]{GKR2017}.

\begin{thmalpha}\label{thm:ldp annealed one dimensional projections p less than 2}
Let $p\in[1,2)$. Then the sequence of random variables $(W^{(n,p)})_{n\in\N}$ satisfies an LDP on $\R$ with speed $n^{r_p}$ and quasi-convex, symmetric, good rate function $\mathbb J_p$.
\end{thmalpha}

It is immediately clear that this result cannot be proved in the same way as the LDP in Theorem \ref{thm:ldp annealed one dimensional projections p greater 2}, because for $p<2$, $ \Phi_p(t_0,t_1,t_2)=\infty$ for $t_1\neq 0$; this already suggests that the LDP, if it exists, occurs at a different speed slower than $n$. 

\begin{proof}[Idea of Proof of Theorem \ref{thm:ldp annealed one dimensional projections p less than 2}.]
In a first step, one proves that if $Y\sim \mu_p$ and $Z\sim\mu_2$, and $Y$ is independent of $Z$, then it holds that
  \begin{equation}\label{eq:limiting behavior product random variables}
    \lim_{t\to\infty} \frac{1}{t^{r_p}} \log\big( \Pro[YZ\geq t] \big) = -\frac{1}{r_p};
  \end{equation}
  to prove this, one uses the $p$-Gaussian tail bounds saying that, for any $x\geq 0$,
  \begin{equation}\label{eq: pgaussian tail bounds}
    \frac{x}{x^p+1}e^{-x^p/p} \leq \int_x^{\infty} e^{-y^p/p}\,\dint y \leq   \frac{1}{x^{p-1}}e^{-x^p/p}.  
  \end{equation}
  In the second step, one uses \eqref{eq:limiting behavior product random variables} together wit a result of M.~A.~Arcones on large deviations of empirical processes with non-standard rates \cite[Theorem 2.1]{A2002} (alternatively one may use \cite[Theorem 1]{GRR2014} on LDPs for random variables with stretched exponential tails) to show that the random sequence
  \[
    V^{(n,p)}:= \frac{1}{n} \sum_{i=1}^n Y_i^{(n,p)}Z_i^{(n)}, \quad n\in\N,
  \] 
  with $Y^{(n,p)}= (Y^{(n,p)}_1,\dots,Y^{(n,p)}_n)$ having iid coordinates with distribution $\mu_p$ and $Z^{(n)}=(Z^{(n)}_1,\dots,Z^{(n)}_n)$ being a vector of independent standard Gaussians, satisfies an LDP 
  with speed $n^{r_p}$ and quasi-convex good rate function $\mathbb J_p$. 
  
  In a third step, one proves that the sequence 
  \[
  \frac{n^{1/p}}{n^{1/2}} \frac{\sum_{i=1}^nY^{(n,p)}_iZ_i^{(n)}}{\| Y^{(n,p)} \|_p\|Z^{(n)}\|_2}, \quad n\in\N, 
  \]
  which we considered above in the proof of Theorem \ref{thm:ldp annealed one dimensional projections p greater 2}, and $V^{(n,p)}$, $n\in\N$, are exponentially equivalent with speed $n^{r_p}$; thus they satisfy the same LDP. Again, as mentioned in the proof of Theorem \ref{thm:ldp annealed one dimensional projections p greater 2}, the LDP now carries over to $(W^{(n,p)})_{n\in\N}$ with speed $n^{r_p}$ and quasi-convex, symmetric, good rate function $\mathbb J_p$ (see \cite[Lemma 3.4]{GKR2017}).
\end{proof}

\begin{rmk}
Again, we refer the reader to \cite[Theorem 1.2]{APT2018} for the case of high-dimensional projections, which shall not be discussed here.
\end{rmk}

\subsection{Large deviations for $\ell_q$-norms of random vectors in $\ell_p$-balls}\label{subsec:ldps for qnorms of random vectors in lp spaces}

We now present some LDPs obtained by Z.~Kabluchko, J.~Prochno, and C.~Thäle for $\ell_q^n$-norms of random vectors chosen uniformly at random from $\ell_p^n$-balls \cite{KPT2019_I}; we shall skip the central and non-central limit theorems the authors obtained and focus on the large deviations part. As a motivation, let us start with the classical deviation results of G.~Schechtman and J.~Zinn \cite[Theorem 3]{SchechtmanZinn} (see \cite[Corollary 4]{SchechtmanZinn} for the normalized Lebesgue measure) and of A.~Naor in \cite[Theorem 2]{N2007}. The authors proved that if $1\leq p < q < \infty$, then
\begin{equation}\label{eq:SZConcentration}
\Pro\big[n^{1/p-1/q}\|Z\|_q >z\big] \leq e^{-cn^{p/q}z^p}
\end{equation}
for all $z>T(p,q)$, with $c:=1/T(p,q)$. Here $Z$ can be either uniformly distributed in $\B_p^n$ or distributed according to the cone measure on the boundary of $\B_p^n$. 

The authors of \cite{KPT2019_I} obtained an LDP counterpart, i.e., an asymptotic version of \eqref{eq:SZConcentration}. In fact, they identify $n^{p/q}$ in the exponent on the right-hand side of \eqref{eq:SZConcentration} as the speed of the LDP and $z^p$ as the asymptotically leading term of the rate function as $z\to\infty$. This LDP is in a sense optimal, and one can, contrary to \cite{SchechtmanZinn}, identify the exact constant $1/p$ in the exponent. It is important to note, however, that in general neither one of the results implies the other. Only for deviation parameters $z$ from a fixed compact interval, the result of \cite{KPT2019_I} is optimal and indeed stronger. The following result is based on large deviations techniques for sums of random variables with stretched exponential tails obtained in \cite{GRR2014} (i.e., a generalized version of Theorem \ref{thm:cramer stretched} above) and corresponds to \cite[Theorem 1.3]{KPT2019_I}.
\begin{thmalpha}\label{thm:LDPp<q}
Assume that $1\leq p<\infty$ and, for $n\in\N$, let $Z^{(n,p)}$ be uniformly distributed on $\B_p^n$. If $p<q<\infty$, then the sequence $\|{\bf Z}\|:=(n^{1/p-1/q}\|Z^{(n,p)}\|_q)_{n\in\N}$ satisfies an LDP in $\R$ with speed $n^{p/q}$ and good rate function
$$
\rate_{\bf \|Z\|} (z) := \begin{cases}
{1\over p}\big(z^q-M_p(q)\big)^{p/q} &: z\geq M_p(q)^{1/q}\\
\infty &: \text{otherwise}\,,
\end{cases}
$$
where $M_p(q):=\frac{p^{q/p}}{q+1}\,\frac{\Gamma(1+\frac{q+1}{p})}{\Gamma(1+\frac{1}{p})}$.
\end{thmalpha} 
\begin{proof}[Idea of Proof.]
The  case $p<q$ is more delicate than $q\leq p$ because of the lack of finite exponential moments. Thus, instead of Cram\'er's theorem, one uses a version for random variables with heavy tails. More precisely, the following result due to N.~Gantert, K.~Ramanan, and F.~Rembart \cite[Theorem 1]{GRR2014} (which is actually reformulated in a suitable way here) is used: let $X,X_1,X_2,\ldots$ be non-negative iid random variables, and assume that there are constants $r\in(0,1)$, $t_0>0$, and slowly varying functions $c_1,c_2,b:(0,\infty)\to(0,\infty)$ such that
$$
c_1(t)e^{-b(t)t^r}\leq \Pro[X\geq t]\leq c_2(t)e^{-b(t)t^r}\,,\qquad t\geq t_0\,.
$$
Then, the sequence of random variables ${1\over n}\sum\limits_{i=1}^nX_i$, $n\in\N$, satisfies an LDP on $\R$ with speed $b(n)n^r$ and good rate function
$$
\rate(z) = \begin{cases}
(x-\E X)^r &: z\geq \E X\\
\infty &: \text{otherwise}\,.
\end{cases}
$$

For each $n\in\N$, let $Y_1,\ldots,Y_n\sim \mu_p$ be independent $p$-generalized Gaussians and define
$$
S_n:=\sum_{i=1}^n|Y_i|^q\,.
$$
In a first step, the authors apply Cram\'er's theorem for stretched exponentials from \cite{GRR2014} to prove an LDP for the random sequence ${\bf S}:=(S_n)_{n\in\N}$. Using the $p$-Gaussian tail bounds from \eqref{eq: pgaussian tail bounds},
one can check that, for sufficiently large $z\in\R$,
$$
c_1(z) e^{-b(z)z^{p/q}} \leq \Pro[|Y_1|^q\geq z]\leq c_2(z) e^{-b(z)z^{p/q}},
$$
with appropriate slowly varying functions $c_1,c_2,b:(0,\infty)\to(0,\infty)$, where $b(z)$ satisfies $b(z)\to{1\over p}$, as $z\to\infty$. Thus, $\bf S$ satisfies an LDP on $\R$ with speed $n^{p/q}$ and good rate function
$$
\rate_{\bf S}(z) = \begin{cases}
{1\over p}(z-M_p(q))^{p/q} &: z\geq M_p(q)\\
+\infty &:\text{otherwise}\,.
\end{cases}
$$
In a second step, one applies the contraction principle of Lemma \ref{lem:contraction principle} with the continuous function $F(z)=z^{1/q}$, $z\in\R$, to conclude that  $(F(S_n))_{n\in\N}=((S_n)^{1/q})_{n\in\N}$ satisfies an LDP on $\R$ with speed $n^{p/q}$ and good rate function $\rate_{\bf \|Z\|} $ from the statement of the theorem.

In a third and last step, one shows that $((S_n)^{1/q})_{n\in\N}$ is exponentially equivalent to the sequence $(n^{1/p-1/q}\|Z^{(n,p)}\|_q)_{n\in\N}$ of interest. Here one uses the Schechtman--Zinn probabilistic representation from \eqref{eq:schechtman-zinn representation} to obtain the distributional identity
$$
n^{1/p-1/q}\|Z^{(n,p)}\|_q \overset{d}{=}U^{1/n}{\Big({1\over n}\sum\limits_{i=1}^n|Y_i|^q\Big)^{1/q}\over\Big({1\over n}\sum\limits_{i=1}^n|Y_i|^p\Big)^{1/p}}\,. 
$$
Now, for fixed $\delta>0$ and $\varepsilon\in(0,1)$, one can show that 
\begin{align*}
&\Pro\left[\,\left|\Big({1\over n}\sum\limits_{i=1}^n|Y_i|^q\Big)^{1/q}-U^{1/n}{\Big({1\over n}\sum\limits_{i=1}^n|Y_i|^q\Big)^{1/q}\over\Big({1\over n}\sum\limits_{i=1}^n|Y_i|^p\Big)^{1/p}}\right|>\delta\right]\cr
&\leq \Pro\bigg[\Big({1\over n}\sum\limits_{i=1}^n|Y_i|^q\Big)^{1/q}>{\delta\over\varepsilon}\bigg] + \Pro\big[U^{1/n}<\sqrt{1-\varepsilon}\,\big]+\Pro\bigg[{1\over n}\sum\limits_{i=1}^n|Y_i|^p>(1-\varepsilon)^{-p/2}\bigg]\cr
&\hspace{4.6cm}+\Pro\big[U^{1/n}>\sqrt{1+\varepsilon}\,\big]+\Pro\bigg[{1\over n}\sum\limits_{i=1}^n|Y_i|^p<(1+\varepsilon)^{-p/2}\bigg]\,.
\end{align*}
From Cram\'er's theorem \cite[Theorem 2.2.30]{DZ2010} and \cite[Lemma 3.3]{GKR2017} one deduces that the four last terms decay exponentially with speed $n$,
while the exponential asymptotic of the first term has already been determined above. As a consequence, and since $p/q<1$, one has that
\begin{align*}
&\limsup_{n\to\infty}{1\over n^{p/q}}\log\Pro\left[\,\left|\Big({1\over n}\sum\limits_{i=1}^n|Y_i|^q\Big)^{1/q}-U^{1/n}{\Big({1\over n}\sum\limits_{i=1}^n|Y_i|^q\Big)^{1/q}\over\Big({1\over n}\sum\limits_{i=1}^n|Y_i|^p\Big)^{1/p}}\right|>\delta\right]\cr
&\qquad\qquad\leq\limsup_{n\to\infty}{1\over n^{p/q}}\log\Pro\bigg[\Big({1\over n}\sum\limits_{i=1}^n|Y_i|^q\Big)^{1/q}>{\delta\over\varepsilon}\bigg]\,,
\end{align*}
which is $-{1\over p}(({\delta\over\varepsilon})^q-M_p(q))^{p/q}$ if $\delta/\varepsilon\geq M_p(q)^{1/q}$ and $-\infty$ otherwise. Letting $\varepsilon\to 0$, one concludes that the above limit is equal to $-\infty$. Thus, the two sequences $(n^{1/p-1/q}\|Z^{(n,p)}\|_q)_{n\in\N}$ and $((S_n)^{1/q})_{n\in\N}$ are exponentially equivalent, thus obeying the same LDP (see Lemma \ref{lem:exponential equivalence same ldp}). 
\end{proof}

Let us now present the LDP in the regime $1\leq q\leq p<\infty$, which corresponds to  \cite[Theorem 1.2]{KPT2019_I} . As already pointed out, in this case exponential moments exist and make the large deviations analysis rather straight forward, but we will also see that the two regimes exhibit a different behavior from a rate function and speed point of view. In particular, the form of the rate function already hints at an application of the large deviations result for pairs of independent random objects presented in Lemma \ref{lem: ldp pairs of independent random objects}. Let us note that the case $p = q$ is rather immediate (see \cite[Theorem 1.4]{KPT2019_I}) and so below we work with $q<p$; the special case $p = \infty$ is treated in (see \cite[Theorem 1.5]{KPT2019_I}).

\begin{thmalpha}\label{thm:LDPp>q}
Assume that $1\leq p <\infty$ and, for $n\in\N$, let $Z^{(n,p)}$ be uniformly distributed on $\B_p^n$. If $1\leq q<p$, the sequence $\|{\bf Z}\|:=(n^{1/p-1/q}\|Z^{(n,p)}\|_q)_{n\in\N}$ satisfies an LDP with speed $n$ and good rate function
$$
\ratej_{\|\textbf{Z }\|} (z)  := \begin{cases}
\inf\limits_{z=z_1z_2\atop z_1,z_2\geq 0}[ {\rate}_1(z_1)+{\rate}_2(z_2)] &: z\geq 0\\
\infty &:\text{otherwise}\,.
\end{cases}
$$
Here
$$
{\rate}_1(z) := \begin{cases}
-\log z &: z\in(0,1]\\
\infty &:\text{otherwise}
\end{cases}
\qquad\text{and}\qquad
\rate_2(z) :=
 \begin{cases}
\inf\limits_{x\geq 0,y>0\atop{x^{1/q}y^{-1/p}}=z} \Lambda^*(x,y) &: z\geq 0\\
\infty &: z<0\,,
\end{cases}
$$
where $\Lambda^*$ is the Legendre transform of the function
\[
\Lambda(t_1,t_2):=\log \int_{\R} e^{t_1|s|^q+t_2|s|^p}\frac{e^{-|s|^p/p}}{2p^{1/p}\Gamma(1+1/p)}\dif s\,.
\]
\end{thmalpha}
\begin{proof}[Idea of Proof.]
First of all, one employs again the Schechtmann--Zinn probabilistic representation from \eqref{eq:schechtman-zinn representation} to obtain, for each $n\in\N$,
\[
n^{1/p-1/q}\|Z^{(n,p)}\|_q\stackrel{\text{d}}{=}n^{1/p-1/q}U^{1/n}\frac{\|Y^{(n,p)}\|_q}{\|Y^{(n,p)}\|_p},
\]
where $Y^{(n,p)}=(Y^{(n,p)}_1,\ldots,Y^{(n,p)}_n)$ is a vector of iid $p$-generalized Gaussians, and $U$ is an independent random variable uniformly distributed on $[0,1]$. Then, for $n\in\N$, one defines the random quantities
\[
W_n := n^{1/p-1/q}\frac{\|Y^{(n,p)}\|_q}{\|Y^{(n,p)}\|_p} \qquad\text{and}\qquad S_n:= \frac{1}{n}\sum_{i=1}^n \big(|Y_i^{(n,p)}|^q,|Y_i^{(n,p)}|^p\big)\,.
\]
Now, for $t=(t_1,t_2)\in\R^2$,
\[
\Lambda(t_1,t_2)= \log\Big( \E \Big[e^{\langle t, (|Y_1^{(n,p)}|^q,|Y_1^{(n,p)}|^p) \rangle} \Big]\Big)= \log \int_0^\infty e^{t_1s^q+(t_2-1/p)s^p}\frac{\dif s}{2p^{1/p}\Gamma(1+1/p)}
\]
is the log-moment generating function of $S_n$. Its effective domain is $\R\times(-\infty,1/p)$ if $q<p$ and $\{(t_1,t_2)\in\R^2\,:\,t_1+t_2<1/p\}$ for $q=p$. Thus, by Cram\'er's theorem \cite[Theorem 2.2.30]{DZ2010}, the sequence ${\bf S}:=(S_n)_{n\in\N}$ satisfies an LDP in $\R^2$ with speed $n$ and good rate function $\Lambda^*$ being the Legendre transform of $\Lambda$. One can check that $\Lambda^*(t_1,t_2)= + \infty$ if $t_1\leq 0$ or $t_2\leq 0$. This implies that the sequence ${\bf S}$ also satisfies an LDP on $[0,\infty)\times(0,\infty)$ with the same good rate function $\Lambda^*$.

In a next step, one defines the continuous function
\[
F:[0,\infty)\times(0,\infty)\to\R,\quad (x,y)\mapsto x^{1/q}y^{-1/p},
\]
and notes that, for each $n\in\N$, $W_n\overset{d}{=}F(S_n)$. The contraction principle of Lemma \ref{lem:contraction principle} implies that the random sequence ${\bf W}:=(W_n)_{n\in\N}$ satisfies an LDP on $\R$ with speed $n$ and good rate function
\[
\rate_{\bf W}(z) =
 \begin{cases}
\inf\limits_{x\geq 0,y>0\atop{x^{1/q}y^{-1/p}}=z}\Lambda^*(x,y) &: z\geq 0\\
+\infty &: z<0\,.
\end{cases}
\]
In the third and last step, one considers the random variables $V_n:=(U^{1/n},W_n)$, $n\in\N$, and uses again the fact (see \cite[Lemma 3.3]{GKR2017}) that the sequence ${\bf U}:=(U^{1/n})_{n\in\N}$ satisfies an LDP on $\R$ with speed $n$ and rate function
$$
\rate_{\bf U}(z) = \begin{cases}
-\log z &: z\in(0,1]\\
\infty &:\text{otherwise}\,.
\end{cases} 
$$
Now, since $U^{1/n}$ and $W_n$ are independent, the sequence ${\bf V}:=(V_n)_{n\in\N}$ satisfies an LDP on $\R^2$ with speed $n$ and good rate function
$$
\rate_{\bf V} (z_1,z_2) := \rate_{\bf U}(z_1)+\rate_{\bf W}(z_2)\,,\qquad (z_1,z_2)\in\R^2\,,
$$
by Lemma \ref{lem: ldp pairs of independent random objects} (see also \cite[Lemma 2.2]{A2002}). Applying once more the contraction principle, but this time with the continuous function
$F:\R^2\to\R$, $(x,y)\mapsto xy$, one concludes that the sequence $U^{1/n}W_n\overset{d}{=}n^{1/p-1/q}\|Z^{(n,p)}\|_q$, $n\in\N$, satisfies an LDP on $\R$ with speed $n$ and good rate function
$$
\ratej_{\|\textbf{Z }\|} (z) = \inf_{z=z_1z_2}{\rate}_{\bf V} (z_1,z_2) = \begin{cases}
\inf\limits_{z=z_1z_2\atop z_1,z_2\geq 0}{\rate}_{\bf V} (z_1,z_2) &: z\geq 0\\
\infty &:\text{otherwise}.
\end{cases}
$$
This completes the argument.
\end{proof}

\subsection{Moderate deviations for $\ell_q$-norms of random vectors in $\ell_p$-balls}

In \cite{KPT2019_II}, Z.~Kabluchko, J.~Prochno, and C.~Thäle complemented the LDPs presented above by MDPs (which occur on scalings between the one of the CLT and LDP). More precisely, we refer the reader to \cite[Theorem C]{KPT2019_II}, where the result is actually obtained for more general distributions on the $\ell_p^n$-balls as introduced by F.~Barthe, O.~Gu{\'e}don, S.~Mendelson, and A.~Naor in \cite{BartheGuedonEtAl}. We shall not sketch the proof of the MDP and instead refer to \cite{KPT2019_II} directly.

Let us recall that an MDP is formally nothing else than an LDP, but with important differences in the behavior of the two principles. For instance, while LDPs provide estimates on the scale of a law of large numbers, MDPs describe the probabilities at scales between a law of large numbers and a distributional limit theorem (like a CLT). Moreover, while the rate function in an LDP depends in a subtle way on the distribution of the underlying random variables, the rate function in an MDP, in typical situations, is only parametric in the variance that occurs and is inherited from a CLT.
Let us recall that a sequence $(X_n)_{n\in\N}$ of random vectors in $\R^d$ ($d\in\N$) satisfies an LDP with {\emph speed} $s_n$ and good rate function $\rate:\R^d\to[0,\infty]$ if and only if
\begin{equation*}
\begin{split}
-\inf_{x\in A^\circ}\rate(x) &\leq\liminf_{n\to\infty}s_n^{-1}\log\Pro[X_n\in A]\leq\limsup_{n\to\infty}s_n^{-1}\log\Pro[X_n\in A]\leq-\inf_{x\in\overline{A}}\rate(x)
\end{split}
\end{equation*}
for all measurable $A\subset\R^d$. One says that a sequence $(X_n)_{n\in\N}$ satisfies an MDP if the speed sequence $(s_n)_{n\in\N}$ is given by $s_n=b_n\sqrt{n}$ with a positive sequence $(b_n)_{n\in\N}$ satisfying $b_n=\omega(1)$ and $b_n=o(\sqrt{n})$, where for two sequences $(x_n)_{n\in\N}$ and $(y_n)_{n\in\N}$ we use the Landau notation $x_n=o(y_n)$ if $\lim_{n\to\infty} \frac{x_n}{y_n}=0$ and $x_n=\omega(y_n)$ if $\lim_{n\to\infty}|\frac{x_n}{y_n}|=\infty$. We refer the reader to \cite[Chapter 3.7]{DZ2010} and especially part (c) of the remark after Theorem 3.7.1 for further details on terminology.

\begin{thmalpha}\label{thm:MDP}
Let $0<p< \infty$ and $0<q<\infty$ with $q<p$. Let $(b_n)_{n\in\N}$ be a sequence of positive real numbers satisfying $b_n=\omega(1)$ and $b_n=o(\sqrt{n})$.
For each $n\in\N$, let $Z^{(n,p)}$ be uniformly distributed on $\B_p^n$.
Then the sequence of random variables
  \begin{align*}
    {\sqrt{n}\over b_n}\bigg({n^{{1/p}-{1/q}}\over M_p(q)^{1/q}}\|Z^{(n,p)}\|_q-1\bigg), \qquad n\in\N,
  \end{align*}
satisfies an MDP with speed $b_n^2$ and good rate function $\rate(t)={t^2\over 2\sigma^2}$, $t\in\R$, where 
  \[
    \sigma^2 := {1\over q^2}\bigg({\Gamma({1\over p})\Gamma({2q+1\over p})\over\Gamma({q+1\over p})^2}-1\bigg)-{1\over p} \qquad\text{and}\qquad M_p(q):=\frac{p^{q/p}}{q+1}\,\frac{\Gamma(1+\frac{q+1}{p})}{\Gamma(1+\frac{1}{p})}.
  \]
\end{thmalpha} 

In particular, Theorem \ref{thm:MDP} implies that, for all $t\in\R$,
$$
\lim_{n\to\infty}{1\over b_n^2}\log\Pro\Bigg[{\sqrt{n}\over b_n}\bigg({n^{{1/ p}-{1/q}}\over M_p(q)^{1/q}}\|Z^{(n,p)}\|_q-1\bigg)\geq t\Bigg] = -{t^2\over 2\sigma^2}.
$$

\begin{rmk}
Note that already for $p=q$ the core term leading the MDP vanishes and therefore, one does not obtain an MDP with a non-trivial rate function. 
\end{rmk}

\subsection{Large deviations, moderate deviations, and the KLS conjecture}

Arguably, the major open problem in asymptotic geometric analysis is the Kannan--Lov\'asz--Simonovits (KLS) conjecture. Its origin lies in theoretical computer science, where it arose in the study of sampling algorithms for high-dimensional convex bodies. A specific situation of interest is to design an algorithm efficiently computing the volume of a $d$-dimensional convex body. Such an algorithm is given a convex body $K\subseteq \R^d$ and a quality parameter $\varepsilon\in(0,\infty)$. The body $K$ is represented by a membership oracle, which, for each $x\in \R^d$, can decide whether or not $x\in K$. The algorithm returns a number $V:=V(K,\varepsilon)\in\R$ such that $(1-\varepsilon)\vol_d(K) \leq V \leq (1+\varepsilon)\vol_d(K)$. The efficiency of the designed algorithm is measured in terms of the number of arithmetic operations and calls to the membership oracle. 

While it is known through the works of I.~B\'ar\'any and Z.~F\"uredi \cite{BF1987} and of G.~Elekes \cite{Elekes1986} that deterministic algorithms are inefficient, there are reasonable randomized algorithms that can accurately compute the volume of convex bodies with high probability in polynomial time (see, e.g., the work of M.~Dyer, A.~Frieze, and R.~Kannan \cite{DFK1991}). The construction of this randomized algorithm is connected to an isoperimetric inequality for log-concave probability measures on $\R^d$ (see \cite{AGB2015} for an introduction to the KLS conjecture and \cite{VEMP2010}  for a detailed explanation on its connection with problems in computer science). The constant appearing in this isoperimetric  inequality, referred to as Cheeger's constant, is of particular interest and directly linked to the KLS conjecture, which can be stated as follows (see \cite[Theorem 1.1]{BH1997}):

\bigskip

\textbf{KLS Conjecture.} \textit{There exists an absolute constant $C\in(0,\infty)$ such that for all $d\in\N$, every centered random vector $X$ with log-concave distribution, and any locally Lipschitz function $f:\R^d\to\R$ such that $f(X)$ is of finite variance,
\[
\Var[f(X)] \leq C\, \lambda_X^2\, \E\big[\| \nabla f(X)\|_2^2\big],
\]
where $\lambda_X^2:=\sup_{\theta\in\SSS^{d-1}}\E\langle X,\theta\rangle^2$.}

\bigskip

Using the localization lemma, R.~Kannan, L.~Lov\'asz, and M.~Simonovits proved the conjecture with a factor $(\E\|X\|_2)^2$ instead of $\lambda_X^2$ (see \cite{KLS1995}), and later S.~Bobkov \cite{BOB2007} improved this estimate. The currently best known bound has recently been obtained by B.~Klartag \cite{K2023}, who showed that the conjecture in $\R^d$ holds true up to a factor of $\sqrt{\log d}$. 

Even though the bound has been improved considerably in recent years, it is not clear whether the conjecture is indeed true. In \cite{APT2020}, D.~Alonso-Guti\'{e}rrez, J.~Prochno, and C.~Thäle established a new connection between the KLS conjecture and the study of LDPs and MDPs for isotropic log-concave random vectors, providing a potential route to disproving the conjecture.   

As mentioned before, an MDP is nothing else than an LDP for a differently normalized sequence of random variables, i.e., if a sequence $(t_n^{-1}X_n)_{n\in\N}$ of random variables satisfies an LDP with some speed and some rate function, an LDP for the sequence of random variables $\tilde{t}_n^{-1}X_n$ with $\tilde{t}_n\to\infty$ and $\tilde{t}_n=o(t_n)$ is referred to as an MDP.

The following result corresponds to \cite[Theorem A]{APT2020} and establishes a connection between the KLS conjecture and moderate/large deviations principles. Essentially, it is a consequence of the fact, known from \cite{AGB2015,GM1987,L1994}, that the KLS conjecture would imply a sharp concentration bound for the Euclidean norm of isotropic log-concave random vectors. Below, for two positive sequences of quantities, we write $\approx$ to denote equivalence up to absolute constants. 

\begin{thmalpha}\label{thm:KLS}
Let $1\leq k_n\leq n$ be a sequence of integers such that $k_n=\omega(1)$ and let $(\xi_n)_{n=1}^\infty$ be a sequence of isotropic log-concave random vectors in $\R^{k_n}$. Consider a sequence of random variables $X_n=\frac{\Vert \xi_n \Vert_2}{\sqrt{k_n}}$, $n\in\N$, which satisfies \eqref{eq:LDPdefinition} with speed $s_n$ and rate function $\rate$. Assume that one of the following two conditions is satisfied: 
\begin{itemize}
\item[(a)] $s_n=o(\sqrt{k_n})$ and $\rate$ is non-singular, i.e., $\rate(x)\neq\rate_0(x):=
\begin{cases}
0&: x=1\\
\infty&:\text{otherwise}.
\end{cases}$
\item[(b)] $s_n\approx \sqrt{k_n}$ and $\inf\limits_{t>t_0}\frac{\inf_{x\in(t,\infty)}\rate(x)}{t}=0$ for some absolute constant $t_0\in(1,\infty)$.
\end{itemize}
Then the KLS conjecture is false.
\end{thmalpha}

For a sequence of random vectors $\xi_n\in\R^{k_n}$, $n\in\N$, as given in Theorem \ref{thm:KLS}, it is well known that $\E\|\xi_n\|_2/\sqrt{k_n}\to 1$, as $n\to\infty$. Against this light, the natural scale for a law of large numbers, and hence an LDP, is $\sqrt{k_n}$, while scales with $t_n=o(\sqrt{k_n})$ and $t_n=\omega(1)$ correspond to an MDP. Thus, Theorem \ref{thm:KLS} establishes a connection between LDPs or MDPs for isotropic log-concave random vectors and the KLS conjecture, showing that by constructing a sequence of random vectors $\xi_n$, $n\in\N$, such that $X_n$, $n\in\N$, follows an LDP or MDP with an appropriate speed and/or rate function would disprove the variance and hence the KLS conjecture. Moreover, the authors show that (rescaled) crosspolytopes constitute a critical case in Theorem \ref{thm:KLS} as they just fail to satisfy condition (b) there. 

\begin{rmk}
Beyond the connection to the KLS conjecture, the authors of \cite{APT2020} prove a number of other results (see  \cite[Theorems B and C]{APT2020}), more precisely, they study moderate deviations for the Euclidean norm of random orthogonally projected random vectors in $\ell_p^d$-balls. This leads to a number of interesting observations: (1) for $p\geq 2$ the rate function in the MDP undergoes a phase transition, depending on whether the scaling is below the square-root of the subspace dimensions or comparable; (2) for $1\leq p<2$ and comparable subspace dimensions, the rate function again displays a phase transition depending on its growth relative to $d^{p/2}$.  
\end{rmk}

\begin{proof}[Idea of Proof of Theorem \ref{thm:KLS}.]
First, we recall the following fact (see \cite[Theorem 1.15]{AGB2015}, \cite{GM1987} or \cite{L1994}: if the KLS conjecture were true, then there would exist an absolute constant $C\in(0,\infty)$ such that, for every $n\in\N$ and all $t>0$,
  $$
    \Pro\left[\left|\frac{\Vert \xi_n\Vert_2}{\sqrt{k_n}}-1\right|>t\right]\leq 2e^{-Ct\sqrt{k_n}}.
  $$
But this would mean that
\begin{eqnarray*}
\log \Pro\left[\frac{\Vert \xi_n\Vert_2}{\sqrt{k_n}}-1>t\right]&\leq&\log\Pro\left[\left|\frac{\Vert \xi_n\Vert_2}{\sqrt{k_n}}-1\right|>t\right]\leq \log 2-Ct\sqrt{k_n}
\end{eqnarray*}
and
\begin{eqnarray*}
\log \Pro\left[\frac{\Vert \xi_n\Vert_2}{\sqrt{k_n}}<1-t\right]&\leq&\log\Pro\left[\left|\frac{\Vert \xi_n\Vert_2}{\sqrt{k_n}}-1\right|>t\right]\leq \log 2-Ct\sqrt{k_n},
\end{eqnarray*}
and thus, for every $n\in\N$ and $t>0$,
\begin{eqnarray*}
\frac{\log \Pro\left[\frac{\Vert \xi_n\Vert_2}{\sqrt{k_n}}>1+t\right]}{s_n}&\leq&\frac{\log 2}{s_n}-\frac{Ct\sqrt{k_n}}{s_n}
\end{eqnarray*}
and
\begin{eqnarray*}
\frac{\log \Pro\left[\frac{\Vert \xi_n\Vert_2}{\sqrt{k_n}}<1-t\right]}{s_n}&\leq&\frac{\log 2}{s_n}-\frac{Ct\sqrt{k_n}}{s_n}.
\end{eqnarray*}
Taking the limit inferior as $n\to\infty$, and using the assumption that $\sqrt{k_n}^{-1}\Vert X_n \Vert_2$ satisfies \eqref{eq:LDPdefinition} {with speed $s_n$ and rate function $\rate$}, one makes the following observation: if $s_n=o(\sqrt{k_n})$, then, for every $t>0$,
$$
-\inf_{x\in(1+t,\infty)}\rate(x)\leq -\infty \quad \text{and}\quad -\inf_{x\in(-\infty,1-t)}\rate(x)\leq -\infty,
$$
which implies that $\rate$ is identically equal to $\infty$ on $\R\setminus\{1\}$. Since
$$
\Pro\left[\frac{\Vert \xi_n\Vert_2}{\sqrt{k_n}}\in\R\right]=1,
$$
one has that, for every $n\in\N$,
$$
\frac{\log \Pro\left[\frac{\Vert \xi_n\Vert_2}{\sqrt{k_n}}\in\R\right]}{s_n}=0.
$$
Hence, taking the limit as $n\to\infty$,
$$
0 = -\inf_{x\in\R}\rate(x)=-\rate(1),
$$
{which implies that $\rate$ would coincide with the singular rate function $\rate_0$, a contradiction to the assumption that $\rate\neq\rate_0$.}

If otherwise
$s_n\approx\sqrt{k_n}$, then, for any $t>0$, there exists a constant $C_1\in(0,\infty)$ such that
$$
-\inf_{x\in(t+1,\infty)}\rate(x)\leq -C_1t,
$$
which implies that
$$
\frac{\inf_{x\in(t+1,\infty)}\rate(x)}{t+1}>C_1{t\over t+1}
$$
for all $t>0$. This is equivalent to the fact that, for every $t>1$,
$$
\frac{\inf_{x\in(t,\infty)}\rate(x)}{t}>C_1{t-1\over t}.
$$
Now, for every $t_0>1$, one has that, for all $t\in[t_0,\infty)$, $C_1{t-1\over t}>C_2(t_0)$. Thus, for any $t_0>1$,
$$
\inf_{t>t_0}\frac{\inf_{x\in (t,\infty)} \rate(x)}{t} \geq {C_2(t_0)}.
$$
However, this is a contradiction to the assumption.
\end{proof}

\section{Large deviations results and techniques in the world of Schatten classes}\label{sec:ldps in schatten classes}

The Schatten $p$-class $\Sc_p$ ($0<p\leq \infty$) is the collection of all compact operators between Hilbert spaces for which the sequence of their singular values belongs to the sequence space $\ell_p$, including the important cases of the nuclear or trace class operators ($p=1$) and Hilbert--Schmidt operators ($p=2$). Having been introduced by R.~Schatten in \cite[Chapter 6, pp. 71]{Schatten1960}, back then referring to a compact operator as a completely continuous one, they are among the most prominent unitary operator ideals studied in functional analysis today. In fact, he worked in the more general setting of symmetric gauge functions and the unitarily invariant crossnorms on the subalgebra of finite rank operators generated by them. 
The very origin of Schatten's work can be traced back to his paper \cite{S1946} and the subsequent works \cite{SvN1946_II,SvN1946_III} with J.~von Neumann, studying nuclear operators on Hilbert spaces; on Banach spaces this had later been considered by A.~F.~Ruston \cite{R1951} and on locally convex spaces by A.~Grothendieck \cite{G233}. Before Schatten's monograph appeared, \emph{spaces of completely continuous operators} had
received comparably little attention in the literature. Today, Schatten classes provide the mathematical framework for modern applied mathematics around low-rank matrix recovery (the non-commutative analogue to the classical compressed sensing approach) and completion (see, e.g., \cite{CR2009, CDK2015, FR2013, KS2018, RT2011}), which is only one reason for the increased interest in their structure in recent years. The Schatten class $\Sc_p$ may be considered a non-commutative version of the classical $\ell_p$ sequence space and both share various structural characteristics, but while there are several similarities on different levels, there are also many differences in their analytic, geometric, and probabilistic behavior. In fact, there are several situations in which the matrix spaces are harder to handle and where arguments are considerably more delicate and complicated.

From both the local and the global point of view the study of Schatten classes has a long tradition in geometric functional analysis and their structure has been investigated intensively in the past 50 years. For instance, Y.~Gordon and D.~R.~Lewis proved that $\Sc_p$ for $p\neq 2$ fails to have local unconditional structure and therefore does not have an unconditional basis \cite{GL1974}. This answered a question of S.~Kwapie\'n and A.~Pe{\l}czy\'nski, who had previously shown in \cite{KP1970} that the Schatten trace class $\Sc_1$ (naturally identified with $\ell_2\otimes_{\pi}\ell_2$) as well as $\Sc_\infty$ are not isomorphic to subspaces with an unconditional basis. 
In 1974, N.~Tomczak-Jaegermann succeeded in \cite{TJ1974} to prove that $\Sc_1$ has Rademacher cotype $2$. More recently, H.~K\"onig, M.~Meyer, and A.~Pajor obtained in \cite{KMP1998} that the isotropic constants of the unit balls in $\Sc_p^N$ are bounded above by absolute constants for all $1\leq p\leq \infty$. The concentration of mass properties of unit balls in the Schatten $p$-classes were studied by O.~Gu\'edon and G.~Paouris in \cite{GP2007} and J.~Radke and B.-H.~Vritsiou were able to prove the thin-shell conjecture for $\Sc_\infty^N$ \cite{RV2016}. In \cite{HPV2017}, A.~Hinrichs, J.~Prochno, and J.~Vyb\'iral succeeded in determining the asymptotic behavior of entropy numbers for natural identities $\Sc_p^N\hookrightarrow\Sc_q^N$ up to absolute constants for all $0<p,q\leq \infty$ and in \cite{HPV2021_gelfand} studied the sequences of Gelfand numbers for natural embeddings $\Sc_p^N\hookrightarrow\Sc_q^N$ of Schatten classes, providing asymptotically sharp bounds (up to constants depending on $p$ and $q$) for almost all regimes, thereby extending classical results for finite-dimensional $\ell_p$ sequence spaces by E.~D.~Gluskin \cite{G81,G83} and  A.~Y.~Garnaev and E.~D.~Gluskin \cite{GG1984} to the non-commutative setting. Their results have then been complemented by J.~Prochno and M.~Strzelecki in \cite{PS2022}. In a series of papers, Z.~Kabluchko, J.~Prochno, and C.~Th\"ale computed the precise asymptotic volume and the volume ratio for Schatten $p$-classes for $0<p\leq \infty$ \cite{KPT2020_volume_ratio} and proved a Schechtman--Schmuckenschl\"ager type result for the volume of intersections of unit balls \cite{KPT2020_ensembles}.

Given that there had already been obtained a number of large deviations results for $\ell_p^n$-spaces, it was only natural from the functional analytic point of view to study LDPs in the non-commutative setting. This was done by Z.~Kabluchko, J.~Prochno, and C.~Th\"ale, who obtained Sanov-type large deviations for the empirical spectral measures of random matrices in Schatten unit balls \cite{KPT2019_sanov}. In this section, we shall focus on the results obtained there.  

\subsection{Self-adjoint and non self-adjoint Schatten $p$-classes}

Let us describe the general setting, going beyond the mere case of Schatten $p$-classes, in which the results of \cite{KPT2019_sanov} were obtained.

Consider $\beta\in\{1,2,4\}$ and let $\mathscr H_n(\mathbb{F}_\beta)$ be the collection of all self-adjoint $(n\times n)$-matrices with entries from the skew field $\mathbb{F}_\beta$, where $\mathbb{F}_1=\R$, $\mathbb{F}_2=\C$, or $\mathbb{F}_4=\Ham$, the set of Hamiltonian quaternions. The standard Gaussian distribution on $\mathscr H_n(\mathbb{F}_\beta)$ is known as the GOE (Gaussian orthogonal ensemble) if $\beta=1$, the GUE (Gaussian unitary ensemble) if $\beta=2$, and the GSE (Gaussian symplectic ensemble) if $\beta=4$. By $\lambda_1(A),\ldots,\lambda_n(A)$ we shall denote the (real) eigenvalues of a matrix $A$ from $\mathscr H_n(\mathbb{F}_\beta)$ and consider the following Schatten-type unit ball, which can be regarded as a matrix analogue to the classical $\ell_p^n$-balls and is defined as
$$
\B_{p,\beta}^n := \left\{A\in \mathscr H_n(\mathbb{F}_\beta):\sum_{j=1}^n|\lambda_j(A)|^p \leq 1\right\},\qquad \beta \in\{1,2,4 \}\quad\text{and}\quad  0 < p \leq \infty.
$$
If $p=\infty$, then the sum in the definition is replaced by $\max\{|\lambda_j(A)|:j=1,\ldots,n\} $ . The boundary of the matrix ball $\B_{p,\beta}^n$ is denoted by
$$
\Sph_{p,\beta}^{n-1} := \partial \B_{p,\beta}^n  = \left\{A\in \mathscr H_n(\mathbb{F}_\beta):\sum_{j=1}^n|\lambda_j(A)|^p = 1\right\}
$$
with the same convention if $p=\infty$.
The space $\mathscr H_n(\mathbb{F}_\beta)$ admits a natural scalar product $\langle A, B\rangle = \Re \Tr (A B^*)$ so that it becomes a Euclidean space. The corresponding Riemannian volume on $\mathscr H_n(\mathbb{F}_\beta)$ is denoted by $\text{vol}_{\beta,n}$ and this measure coincides with the suitably normalized $({\beta n(n-1)\over 2} + n)$-dimensional Hausdorff measure on $\mathscr H_n(\mathbb{F}_\beta)$ as follows directly from the area-coarea formula. One can therefore define the uniform distribution on $\B_{p,\beta}^n$. The cone probability measure on $\Sph_{p,\beta}^{n-1}$ is defined as follows: the cone measure of a Borel set $K\subseteq \Sph_{p,\beta}^{n-1}$ is
$$
\frac{\text{vol}_{\beta,n} (\cup_{\lambda\in [0,1]} \lambda K)}{\text{vol}_{\beta,n}(\B_{p,\beta}^n)}.
$$

After having discussed the self-adjoint case, we now turn to the non self-adjoint one, where the eigenvalues are replaced by the singular values. For an $(n\times n)$-matrix $A\in\Mat_n(\mathbb{F}_\beta)$ with entries from the skew field $\mathbb{F}_\beta$ with $\beta\in\{1,2,4\}$, we denote by $s_1(A),\ldots,s_n(A)$ the singular values of $A$. If $\beta=1$ or $\beta=2$ these are the eigenvalues of $\sqrt{AA^*}$, while in the Hamiltonian case $\beta=4$ we refer to \cite[Corollary E.13]{AGZ2010} for a formal definition. For $0<p\leq \infty$ the Schatten $p$-ball is defined as
$$
\SSS\B_p^n := \Big\{A\in\Mat_n(\mathbb{F}_\beta):\sum_{j=1}^n|s_j(A)|^p\leq 1\Big\}
$$
with the convention that the sum is replaced by $\max\{|s_j|:j=1,\ldots,n\}$ in the case that $p=\infty$. As in the self-adjoint case, $\Mat_n(\mathbb{F}_\beta)$ can be supplied with the structure of a Euclidean space in such a way that the 
$\beta n^2$-dimensional Hausdorff measure restricted to $\SSS\B_p^n$ is finite and can thus be normalized to a probability measure. Moreover, one can also define the cone probability measure on the boundary $\partial\SSS\B_p^n$ of $\SSS\B_p^n$.

\subsection{A non-commutative Schechtman--Zinn probabilistic representation}

Let us start with the Schechtman--Zinn type probabilistic representations for random matrices chosen uniformly at random from $\B_{p,\beta}^n$ and $\SSS\B_p^n$ for $0<p<\infty$, which are of independent interest. The basic fact the authors of \cite{KPT2019_sanov} rely on is indeed a distributional representation of the empirical measure of a random matrix $Z_n$ chosen from one of the unit balls. We elaborate in slightly more detail for the case of $\B_{p,\beta}^n$, before we present the two results forming the respective backbone.

So assume for each $n\in\N$ we are given a random matrix $Z_n$ uniformly distributed in the ball $\B_{p,\beta}^n$, $\beta\in\{1,2,4\}$. Then one has a probabilistic representation of the empirical measure of $Z_n$ given as
\begin{equation*}\label{eq:schechtman_zinn_uniform}
\mu_n =
\frac 1n \sum_{i=1}^n \delta_{n^{1/p}\lambda_i(Z_n)}
\eqdistr
\frac 1n \sum_{i=1}^n \delta_{n^{1/p} U^{1 /\ell}{X_{i,n}\over\|X_n\|_p}},
\end{equation*}
where $\eqdistr$ denotes equality in distribution and
\begin{itemize}
\item[(a)] $X_n = (X_{1,n},\dots,X_{n,n})$ is a random vector on $\R^n$ with joint Lebesgue density of the form
\begin{equation}\label{eq:distribution of (X_1,...,X_n)}
\R^n \ni (x_1,\ldots,x_n) \mapsto \frac 1 {C_{n,\beta,p}} e^{-n\sum\limits_{i=1}^n|x_i|^p} \prod_{1\leq i < j \leq n}|x_i-x_j|^{\beta},
\end{equation}
with a suitable normalization constant $C_{n,\beta,p}\in(0,\infty)$ depending on $n$, $\beta$, and $p$,
\item[(b)] $U$ is a random variable with uniform distribution on $[0,1]$ that is independent of $X_n$,
\item[(c)] $\ell := \ell(n,\beta) :=  {\beta n(n-1)\over 2} + n$ is the (real) dimension of $\mathscr H_n(\mathbb{F}_\beta)$. \footnote{Let us note that there is a typo in \cite[Item (c)]{KPT2019_sanov}, where $\ell={\beta n(n-1)\over 2} + n\beta$, while it should be ${\beta n(n-1)\over 2} + n$ instead.}
\end{itemize}
Similarly, if $Z_n$ is distributed according to the cone measure on $\Sph_{p,\beta}^{n-1}$, then
\begin{equation*}\label{eq:schechtman_zinn_cone}
\mu_n
=
\frac 1n \sum_{i=1}^n \delta_{n^{1/p}\lambda_i(Z_n)}
\eqdistr
\frac 1n \sum_{i=1}^n \delta_{n^{1/p}{ X_{i,n}\over \|X_n\|_p}}.  
\end{equation*}

We now present the main probabilistic representations used to derive the level--2 LDPs in Schatten classes. The first result is essentially \cite[Corollary~4.3]{KPT2020_ensembles}.

\begin{proposition}\label{cor:EigenvaluesProbabRep}
Let $0 <  p < \infty$ and $Z_n$ be uniformly distributed in $\B_{p,\beta}^n$, $\beta\in\{1,2,4\}$. Consider a uniform random permutation $\pi$ of $\{1,\ldots,n\}$, which is independent from $Z_n$. Then
\[
\big(\lambda_{\pi(1)}(Z_n),\ldots,\lambda_{\pi(n)}(Z_n)\big) \stackrel{d}{=}U^{1\over n+m}{X_n\over\|X_n\|_p}\qquad\text{with}\qquad m:=m(n,\beta):={\beta n(n-1)\over 2},
\]
where $U$ is uniformly distributed on $[0,1]$, and, independently of $U$, the vector $X_n\in\R^n$ has joint density 
proportional to
\[
\R^n\ni(x_1,\ldots,x_n)\mapsto e^{-\sum\limits_{i=1}^n|x_i|^p}\prod_{1\leq i<j\leq n}|x_i-x_j|^\beta,
\]
where the proportionality constant only depends on $\beta$, $p$, and $n$. Also, if $Z_n\in\Sph_{p,\beta}^{n-1}$ is distributed according to the cone probability measure, one has that
$$
\big(\lambda_{\pi(1)}(Z_n),\ldots,\lambda_{\pi(n)}(Z_n)\big) \eqdistr {X_n\over\|X_n\|_{p}}.
$$
\end{proposition}

To present the non self-adjoint variant, for $\beta\in\{1,2,4\}$ let us define the function
$$
\R_+^n\to\R,\quad(x_1,\ldots,x_n)\mapsto h_\beta(x_1,\ldots,x_n):=\prod_{1\leq i<j\leq n}|x_i-x_j|^\beta\,\prod_{i=1}^nx_i^{{\beta\over 2}-1},
$$
which is homogeneous of degree
$$
{\beta n(n-1)\over 2}+n\Big({\beta\over 2}-1\Big).
$$

The following are now the Schechtman--Zinn-type distributional representations for Schatten $p$-classes and essentially \cite[Proposition 8.2]{KPT2019_sanov}; note that instead of $p$ the value $p/2$ occurs at many places since one considers the squares of the singular values and not the singular values themself.

\begin{proposition}\label{prop:schechtman-zinn-non-self-adjoint}
Let $0<p<\infty$ and $\beta\in\{1,2,4\}$. For each $n\in\N$ let $Z_n\in\SSS\B_p^n$ be uniformly distributed and, independently, let $\pi$ be a uniform random permutation of $\{1,\ldots,n\}$. Then
$$
(s_{\pi(1)}^2(Z_n),\ldots,s_{\pi(n)}^2(Z_n))\eqdistr U^{1\over n+m}{X_n\over\|X_n\|_{p/2}} \qquad\text{with}\qquad m:=m(n,\beta):={\beta n(n-1)\over 2}+n\Big({\beta\over 2}-1\Big),
$$
where $U$ is uniformly distributed on $[0,1]$, and, independently of $U$, the random vector $X_n\in\R_+^n$ has joint density proportional to
\begin{align}\label{eq:JointDensitySchatten}
\R_+^n\ni(x_1,\ldots,x_n)\mapsto e^{-n\sum_{i=1}^n|x_i|^{p/2}}\,h_\beta(x_1,\ldots,x_n),
\end{align}
where the proportionality constant only depends on $\beta$, $p$, and $n$. Also, if $Z_n\in\partial\SSS\B_p^n$ is distributed according to the cone probability measure, one has that
$$
(s_{\pi(1)}^2(Z_n),\ldots,s_{\pi(n)}^2(Z_n))\eqdistr {X_n\over\|X_n\|_{p/2}}.
$$
\end{proposition}

\subsection{Sanov-type LDPs for random matrices in Schatten unit balls}

In \cite{KPT2019_sanov} the authors study principles of large deviations in the non-commutative framework of self-adjoint and non self-adjoint Schatten $p$-classes. Let us note that an LDP for the law of the spectral measure of a Gaussian Wigner matrix had already been obtained by G.~Ben Arous and A.~Guionnet \cite[Theorem 1.1 and Theorem 1.3]{BAG1997} and a more general large deviations theorem for random measures (including the case of the empirical eigenvalue distribution of Wishart matrices) had been obtained by F.~Hiai and D.~Petz in \cite[Theorem 1]{HP1998} (see also \cite{HP2000_book}), which followed their preceding ideas from \cite{HP2000}, where the empirical eigenvalue distribution of suitably distributed random unitary matrices was shown to satisfies an LDP as the matrix size goes to infinity. In the same spirit, Z.~Kabluchko, J.~Prochno, and C.~Th\"ale obtained Sanov-type LDPs for the spectral measure of $n^{1/p}$ multiples of random matrices chosen uniformly (or with respect to the cone measure on the boundary) from the unit balls of self-adjoint and non self-adjoint Schatten $p$-classes where $0< p \leq \infty$. While the proofs roughly follow a classical strategy in large deviations theory (see, e.g., \cite{AGZ2010, BAG1997, HP1998, HP2000}), in the Schatten class case one needs to control simultaneously the deviations of the empirical measures \textit{and} their $p$-th moments towards arbitrary small balls in the product topology of the weak topology on the space of probability measures and the standard topology on $\R$ in the self-adjoint or $\R_+$ in the non self-adjoint set-up, respectively, and then prove exponential tightness. A key element is the probabilistic representation for random points in the unit balls of classical matrix ensembles which we presented in the previous subsection (see \eqref{eq:schechtman_zinn_uniform} and \eqref{eq:schechtman_zinn_cone}) and a non self-adjoint counterpart (see Proposition \ref{prop:schechtman-zinn-non-self-adjoint}). As we shall see, the good rate function governing the LDPs is essentially given by the logarithmic energy (which is remarkably the same as the negative of Voiculescu's free entropy introduced in \cite{V1993}) and, which is quite interesting, a perturbation by a constant, which is connected to the Ullman distribution. As a consequence of their LDPs, the authors obtained a law of large numbers and showed that the spectral measure converges weakly almost surely to the Ullman distribution.

\subsubsection{LDPs in the self-adjoint case}

The following theorem provides a Sanov-type LDP in the self-adjoint setting and is \cite[Theorem 1.1]{KPT2019_sanov}.

\begin{thmalpha}\label{theo:main}
Let $0< p < \infty$ and $\beta\in\{1,2,4\}$. For every $n\in\N$, let $Z_n$ be a random matrix chosen according to the uniform distribution on $\B_{p,\beta}^n$ or the cone measure on its boundary $\Sph_{p,\beta}^{n-1}$. Then the sequence of random probability measures
\[
\mu_n := \frac{1}{n}\sum_{i=1}^n \delta_{n^{1/p}\lambda_i(Z_n)},
\qquad n\in\N,
\]
satisfies an LDP on $\mathcal M_1(\R)$ with speed $n^2$ and good rate function $\mathscr I:\mathcal M_1(\R) \to [0,\infty]$ defined by
\begin{equation*}\label{eq:J_def_rate_funct}
\mathscr I(\mu) :=
\begin{cases}
- \frac \beta 2 \int_{\R}\int_{\R} \log|x-y| \, \mu(\dint x)\,\mu(\dint y) + \frac{\beta}{2p} \log\left(\frac{\sqrt{\pi}p \Gamma(\frac{p}{2})}{2^p\sqrt{e}\Gamma(\frac{p+1}{2})}\right) &\,: \int_{\R}|x|^p\mu(\dint x) \leq 1\\
\infty &\,: \int_{\R}|x|^p\mu(\dint x) > 1\,.
\end{cases}
\end{equation*}
\end{thmalpha}

\begin{rmk}
(i) The distribution of the eigenvalues of a point chosen uniformly at random in $\B_{p,\beta}^n$ can be related to the $2$-dimensional Coulomb gas of $n$ particles at inverse temperature $\beta>0$ in an external potential $V:t\mapsto |t|^p$ acting on each particle. \\
\noindent (ii) As discussed in \cite{{KPT2019_sanov}}, the additive constant to the logarithmic energy in the rate function is closely linked to the limit of the free energy whose precise value follows from results of potential theory.
\end{rmk}

The next theorem deals with the case $p=\infty$ and corresponds to \cite[Theorem 1.3]{KPT2019_sanov}.
\begin{thmalpha}\label{theo:main_p_infty}
Fix $\beta\in\{1,2,4\}$. For every $n\in\N$, let $Z_n$ be a random matrix chosen according to the uniform distribution on $\B_{\infty,\beta}^n$ or the cone measure on its boundary $\Sph_{\infty,\beta}^{n-1}$. Then the sequence of random probability measures
\[
\mu_n := \frac{1}{n}\sum_{i=1}^n \delta_{\lambda_i(Z_n)},
\qquad n\in\N,
\]
satisfies an LDP on $\mathcal M_1(\R)$ with speed $n^2$ and good rate function $\mathscr I:\mathcal M_1(\R) \to [0,\infty]$ defined by
\[
\mathscr I(\mu) =
\begin{cases}
- \frac \beta 2 \int_{\R}\int_{\R} \log|x-y| \, \mu(\dint x)\,\mu(\dint y)
 - \frac {\beta}{2} \log 2 &\,: \mu ([-1,1]) =1 \\
\infty &\,: \mu ([-1,1]) < 1\,.
\end{cases}
\]
\end{thmalpha}

As a corollary, we can derive a law of large numbers for $\mu_n$ and show that (weakly almost surely) the sequence of empirical measures converges to a non-random limiting distribution given by the Ullman measure (for $p<\infty$) or arcsine measure (for $p=\infty$). We denote the weak convergence of probability measures by ${\overset{w}{\longrightarrow}}$.
\begin{cor}\label{cor:LLN}
Fix $0< p <\infty$ and $\beta\in\{1,2,4\}$.  Let $\mu_{\infty}^{(p)}$ be the probability measure on the interval $[-b_p,b_p]$ with the Ullman density $x\mapsto h_p(x/b_p)/b_p$, where
\begin{equation*}\label{eq:ullman_def}
h_p(x):={p\over\pi}\left(\int_{|x|}^1{t^{p-1}\over\sqrt{t^2-x^2}}\,\dint t\right) \mathbbm 1_{\{|y|\leq 1\}}(x),
\qquad
b_p
:=
\left(\frac {p \sqrt \pi  \Gamma \left(\frac p2\right)} {\Gamma\left(\frac{p+1}{2}\right)}\right)^{1/p}.
\end{equation*}
(The normalization is chosen so that the $p$-th moment of $\mu_{\infty}^{(p)}$ equals $1$).  Then the random measures $\mu_n$ defined in Theorem~\ref{theo:main} satisfy
$$
\Pro \left[\mu_n \toweak \mu_\infty^{(p)}\right] = 1.
$$
The result also holds in the case $p=\infty$ with $\mu_\infty^{(\infty)}$ being the arcsine distribution with Lebesgue density $\frac {1}{\pi} (1-t^2)^{-1/2}$, $t\in (-1,1)$.
\end{cor}
For example, if $p=2$ the limiting spectral density takes the form (see~\cite[pp.~195--196]{HP2000_book}),
$$
\frac{\mu_{\infty}^{(1)}(\dint x)}{\dint x} =  \frac 1 {\pi^2}\log \frac{\pi + \sqrt{\pi^2-x^2}}{|x|}\, \mathbbm 1_{(-\pi,\pi)}(x),
$$
and if $p=1$ we get the semi-circle distribution with density
$$
\frac{\mu_{\infty}^{(2)}(\dint x)}{\dint x} = \frac{1}{2\pi} \sqrt{4-x^2} \, \mathbbm 1_{(-2,2)}(x),
$$
see Figure \ref{fig1}.

\begin{figure}[t]
\begin{center}
  \begin{tikzpicture}[scale=0.5]
      \begin{axis}[
       clip=false,
       xmin=-3.1415,xmax=3.1415,
       ymin=0, ymax=0.65,
       xtick={-3.1415,0,3.1415},
       xticklabels={$-\pi$, $0$,$\pi$}
       ]
        \addplot[domain=-3.1415:3.1415,samples=200,black,thick]{1/3.1415^2*ln((3.1415+sqrt(3.1415^2-x^2))/abs(x))};
      \end{axis}
    \end{tikzpicture}
    \qquad\qquad
  \begin{tikzpicture}[scale=0.5]
    \begin{axis}[
     clip=false,
     xmin=-2,xmax=2,
     ymin=0, ymax=0.5,
     xtick={-2,0,2},
     xticklabels={$-2$, $0$,$2$}
     ]
      \addplot[domain=-2:2,samples=200,black,thick]{1/(2*3.1415)*sqrt(4-x^2)};
    \end{axis}
  \end{tikzpicture}
\end{center}
\caption{Plots of the densities $\frac{\mu_{\infty}^{(1)}(\dint x)}{\dint x}$ (left) and $\frac{\mu_{\infty}^{(2)}(\dint x)}{\dint x}$ (right).}
\label{fig1}
\end{figure}
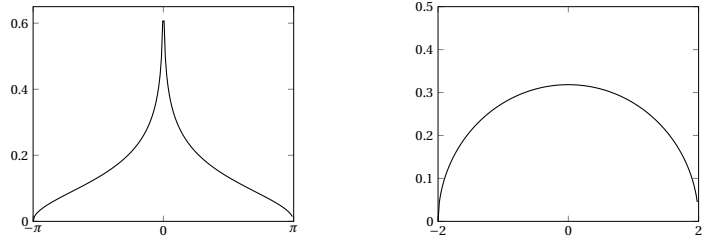

\subsubsection{LDPs in the non self-adjoint case}

After having discussed the self-adjoint case, we now turn to the non self-adjoint case.  The following Sanov-type LDP is the analogue of Theorem \ref{theo:main} and Theorem \ref{theo:main_p_infty} for the non self-adjoint case and corresponds to \cite[Theorem 1.5]{KPT2019_sanov}.

\begin{thmalpha}\label{thm:MainNonSelfAdjoint}
Fix $\beta\in\{1,2,4\}$ and $0<p<\infty$. For every $n\in\N$, let $Z_n$ be a random matrix chosen uniformly from $\SSS\B_p^n$ or according to the cone probability measure from $\partial\SSS\B_p^n$. Then the sequence of random probability measures
$$
\mu_n := {1\over n}\sum_{i=1}^n\delta_{n^{2/p}s_j^2(Z_n)},\qquad n\in\N,
$$
satisfies an LDP on the space $\mathcal{M}_1(\R_+)$ of Borel probability measures on $\mathbb{R}_+$, endowed with the weak topology, with speed $n^2$ and good rate function $\mathscr{J}:\mathcal{M}_1(\R_+)\to[0,\infty]$ given by
$$
\mathscr{J}(\mu):=\begin{cases}
-{\beta\over 2}\int_{\R_+}\int_{\R_+}\log|x-y|\mu(\dint x)\mu(\dint y)+{\beta\over p}\log\Big({\sqrt{\pi}p\Gamma({p\over 2})\over 2^p\sqrt{e}\Gamma({p+1\over 2})}\Big) &: \int_{\R_+}|x|^{p/2}\mu(\dint x)\leq 1\\ 
\infty &:\int_{\R_+}|x|^{p/2}\mu(\dint x)> 1.
\end{cases}
$$
The result continues to holds in the case $p=\infty$ if the constant term in the rate function is replaced by its limiting value, as $p\to\infty$, which is given by $-{\beta\over 2}\log 2$.
\end{thmalpha}

Again as a corollary, the authors derive a law of large numbers for the empirical singular-value distribution (and not for the squares of the singular values as in Theorem \ref{thm:MainNonSelfAdjoint}). This is the analogue of Corollary \ref{cor:LLN} for the Schatten $p$-balls.

\begin{cor}\label{cor:SLLNSchatten}
Fix $\beta\in\{1,2,4\}$ and $0<p<\infty$. Let $\eta_\infty^{(p)}$ be the probability measure on $[0,b_p]$ with density $x\mapsto 2b_p^{-1}h_p(x/b_p)$ with $h_p(x)$ and $b_p$ given by \eqref{eq:ullman_def}. Further, for each $n\in\N$, let $Z_n$ be uniformly distributed in $\SSS\B_p^n$ or distributed according to the cone probability measure on $\partial\SSS\B_p^n$. Then 
$$
\Pro\Big[{1\over n}\sum_{j=1}^n\delta_{n^{1/p}s_j(Z_n)}\toweak\eta_\infty^{(p)}\Big] = 1.
$$
The result also holds in the case $p=\infty$ with $\eta_\infty^{(\infty)}$ being the absolute arcsine distribution with density $x\mapsto{2\over\pi}(1-t^2)^{-1/2}$, $t\in(0,1)$.
\end{cor}

For example, if $p=1$ the limiting distribution has density 
$$
{\dint\eta_\infty^{(1)}\over \dint x}(x) = {2\over\pi^2}\log{\pi+\sqrt{\pi^2-x^2}\over x}\, \mathbbm 1_{(0,\pi)}(x),
$$
and if $p=2$ we get the `quater-circle distribution' with density
$$
{\dint\eta_\infty^{(2)}\over \dint x}(x) = {1\over\pi}\sqrt{4-x^2} \,\mathbbm 1_{(0,2)}(x),
$$
see Figure \ref{fig2}.

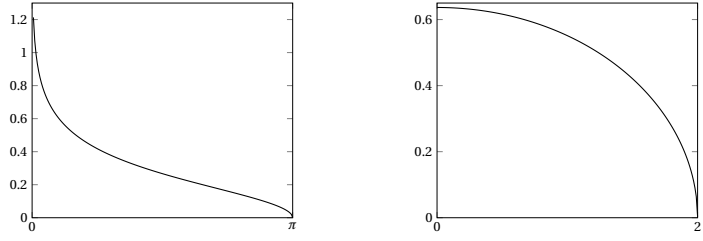
\begin{figure}[t]
\begin{center}
\begin{tikzpicture}[scale=0.5]
      \begin{axis}[
       clip=false,
       xmin=0,xmax=3.1415,
       ymin=0, ymax=1.3,
       xtick={0,3.1415},
       xticklabels={$0$,$\pi$}
       ]
        \addplot[domain=0:3.1415,samples=200,black,thick]{2/3.1415^2*ln((3.1415+sqrt(3.1415^2-x^2))/abs(x))};
      \end{axis}
    \end{tikzpicture}
    \qquad\qquad
  \begin{tikzpicture}[scale=0.5]
    \begin{axis}[
     clip=false,
     xmin=0,xmax=2,
     ymin=0, ymax=0.65,
     xtick={0,2},
     xticklabels={$0$,$2$}
     ]
      \addplot[domain=0:2,samples=200,black,thick]{1/(3.1415)*sqrt(4-x^2)};
    \end{axis}
  \end{tikzpicture}
\end{center}
\caption{Plots of the densities $\frac{\eta_{\infty}^{(1)}(\dint x)}{\dint x}$ (left) and $\frac{\eta_{\infty}^{(2)}(\dint x)}{\dint x}$ (right).}
\label{fig2}
\end{figure}

\subsubsection{Strategy of the proof}\label{Sec:strategy}

The proofs are quite technical and looking at any of the details is beyond the scope of this survey, but we shall take a glimpse at the general strategy of the proof of Theorem~\ref{theo:main} (the non self-adjoint case follows the same strategy).

\vskip 2mm
\noindent
\textit{Step 1:}
The authors prove an LDP for the sequence of pairs
\begin{equation}\label{eq:pair}
\left(\nu_n, \int_\R |x|^p \nu_n(\dint x)\right) = \left(\frac{1}{n}\sum_{i=1}^n\delta_{X_{i,n}},\frac{1}{n}\sum_{i=1}^n |X_{i,n}|^p\right),
\quad n\in\N,
\end{equation}
of empirical measures 
  \[
    \nu_n := \frac{1}{n}\sum_{i=1}^n\delta_{X_{i,n}}
  \]   
of the vectors $X_n=(X_{1,n},\ldots,X_{n,n})$, $n\in\N$, introduced in \eqref{eq:distribution of (X_1,...,X_n)}, and empirical $p$th moments of these measures.
At first it seems natural to try to apply the contraction principle to the LDP for $\nu_n$ (which is known to exist, e.g., \cite[Theorem 5.4.3]{HP2000_book}) with the mapping $\mu \mapsto (\mu, \int_\R |x|^p \mu(\dint x))$. However, this mapping is not continuous in the weak topology. While this may look like a mere technical issue, it is indeed not; one can check that the correct rate function for $(\nu_n, \int_\R |x|^p \nu_n(\dint x))$ does not coincide with what one would expect by a naive application of the contraction principle.

\vskip 2mm
\noindent
\textit{Step 2:} Using the contraction principle, the authors then derive an LDP for the sequence of random measures
\[
\frac{1}{n}\sum_{i=1}^n \delta_{n^{1/p}{ X_{i,n}\over \|X\|_p}}.
\qquad n\in\N,
\]
This proves Theorem~\ref{theo:main} in the case where $Z_n$, $n\in\N$, are sampled according to the cone measure on $\Sph_{p,\beta}^{n-1}$.

\vskip 2mm
\noindent
\textit{Step 3:}
For a uniform random variable $U$ on $[0,1]$, the sequence $(U^{1/\ell})_{n\in\N}$ (recall that $\ell = {\beta n(n-1)\over 2} + n$) satisfies an LDP with rate function 
  \[
   x\mapsto 
   \begin{cases}
     -\frac{\beta}{2}\log x &: x\in(0,1] \cr
     \infty &: \text{otherwise}.
   \end{cases}
  \]
Applying the contraction principle, the authors then derive the LDP for the sequence of random measures
\[
\frac 1n \sum_{i=1}^n \delta_{n^{1/p} U^{1 /\ell}{X_{i,n}\over\|X_n\|_p}}, \qquad n\in\N.
\]
This proves Theorem~\ref{theo:main} in the case where $Z_n$, $n\in\N$, are distributed uniformly on  the ball $\B_{p,\beta}^{n}$.

\section{Large deviations results and techniques in the world of Orlicz spaces}

Most of the large deviations results that have been obtained for quantities in geometric functional analysis belong to the framework of $\ell_p$-spaces (or their non-commutative Schatten $p$-class counterparts). As already explained in the introduction and in Section \ref{sec:ldps in lp spaces}, the reason for being able to obtain MDPs or LDPs in $\ell_p$-spaces is the existence of the Schechtman--Zinn probabilistic representation for random vectors chosen uniformly at random in $\ell_p^n$-balls (or chosen with respect to a generalized distribution as introduced by F.~Barthe, O.~Gu{\'e}don, S.~Mendelson, and A.~Naor \cite{BartheGuedonEtAl}). 
It took quite a while to understand how to go beyond this setting and obtain large or moderate deviations results in the important class of Orlicz spaces, which form a natural generalization of $\ell_p$-spaces and belong to the class of $1$-symmetric Banach spaces; those spaces have first been introduced by W.~Orlicz in 1932 and are of fundamental importance in real analysis, harmonic analysis, and functional analysis. The reason for difficulties in obtaining probabilistic results in Orlicz spaces was that in the world of Orlicz spaces there is no classical Schechtman--Zinn type probabilistic representation of the uniform distribution at our disposal. The key observation was that such a representation now involving appropriate Gibbs measures, is at least asymptotically available and this was derived by Z.~Kabluchko and J.~Prochno by means of maximum entropy considerations inspired by so-called energy constraints in statistical mechanics \cite{KP2020}. In particular, this sheds new light on the classical Schechtman--Zinn probabilistic representation. Before we dive into large deviations results for Orlicz spaces, we shall provide a formal definition and explain the heuristic idea as done in \cite[Section 1.2]{KP2020}, which is of independent interest.

\subsection{Orlicz spaces}

A convex function $M:\R\to \R$ is said to be an Orlicz function if and only if $M(t) = M(-t)$, $M(0)=0$, and $M(t)>0$ for $t\neq 0$.
The functional
\[
\|(x_1,\dots,x_d)\|_{M} := \inf \left \{ \rho > 0 \,:\, \sum_{i=1}^dM\Big(\frac{|x_i|}{\rho}\Big)  \le 1 \right\}
\]
is a norm on $\R^d$, known as Luxemburg norm, named after W.~A.~J.~Luxemburg~\cite{Lux1955}. The Orlicz space $\ell_M^d$ is $\R^d$ equipped with the Luxemburg norm and we denote by
\[
\B_M^d := \Big\{x=(x_i)_{i=1}^d\in\R^d\,:\, \|x\|_M \leq 1 \Big\}
\]
the unit ball in this space. Clearly those spaces generalize the classical $\ell_p^d$-spaces (just consider $M(t)=|t|^p$, $1\leq p<\infty$) and belong to the class of $1$-symmetric Banach spaces. One commonly just speaks of Orlicz functions, Orlicz norms, and Orlicz spaces; we refer to  \cite{KR1961} for an introduction to the theory of Orlicz spaces.

It is a simple observation (see, e.g., \cite[Lemma 2.1]{KP2020}) that $\B_M^d$ coincides with 
\[
B_M^d := \Bigg\{x=(x_i)_{i=1}^d \in\R^d\,:\, \sum_{i=1}^d M(x_i) \leq 1 \Bigg\},
\]
which simplifies computations because one does not need to work with the infimum in the Luxemburg norm.

\subsection{The maximum entropy principle \& Gibbs measures}

Let us explain here how the distributions that play a central role in the main results for Orlicz spaces obtained in \cite{KP2020} naturally appear through what is known as the maximum entropy principle. The idea behind this principle is that the probability distribution which best represents the current state of knowledge about a system is the one with largest entropy (in the physics sense). It seems that it has first been expounded by E.~T.~Jaynes in 1957 \cite{J1957_I,J1957_II} and he suggested a reinterpretation of statistical mechanics thereby connecting it to information theory. 

Our argumentation below will not be mathematically rigorous, but provides a good heuristic to show how the desired Gibbs distributions appear in our setting. We follow \cite[Section 1.2]{KP2020} and refer the interested reader to \cite{RAS2015}, \cite[Section 7.3]{DZ2010}, and \cite[Section III]{E2006} for detailed expositions regarding micro-canonical and canonical ensembles.

Let us consider a sequence of iid random variables $Y_i$, $i\in\N$, taking values in a Polish space $E$ and having distribution $\lambda\in\mathscr M_1(E)$, where we recall that $\mathscr M_1(E)$ is the space of probability measures on $E$ equipped with the weak topology. For $d\in\N$, we shall denote by $L_d:=L_d^Y\in\mathscr M_1(E)$ the empirical measure associated with the $Y_i$'s, i.e.,
  \[
    L_d := \frac{1}{d}\sum_{i=1}^d \delta_{Y_i}.
  \]
This measure is thus a random probability measure. In the setting of Sanov's theorem (see Section \ref{sec:sanovs ldp theorem} above) we know that, as $d\to\infty$, $L_d\to \lambda$ almost surely at an exponential rate. If we consider a set $C$ of probability measures whose closure does not contain $\lambda$, then by the law of large numbers, $\Pro[L_d \in C]\to 0$ as $d\to\infty$. The maximum entropy principle helps us to understand the case where we condition on the rare event that $L_d$ remains in $C$. Roughly speaking and under certain assumptions, $L_d$ converges to the element in the set $C$ that minimizes the relative entropy $H(\cdot|\lambda)$ with respect to the distribution $\lambda$, and so \emph{maximizes} what is known as thermodynamic entropy. 
Being a bit more precise, the maximum entropy principle states that if $C\subset \mathscr M_1(E)$ is closed, convex, and satisfies
\[
\inf_{\nu\in C} H(\nu|\lambda) = \inf_{\nu\in C^{\circ}} H(\nu|\lambda) <\infty,
\]
where $C^{\circ}$ denotes the interior of $C$, then there is a unique probability measure $\nu_*\in C$ minimizing $H(\cdot|\lambda)$ over the set $C$. Moreover, the conditional distributions of $L_d$ converge weakly, as $d\to\infty$, to $\delta_{\nu_*}$, i.e.,
\[
\lim_{d\to\infty} \Pro[L_d\in \cdot \,|\, L_d\in C] = \delta_{\nu_*}(\cdot)
\]
in the weak topology on $\mathscr M_1(\mathscr M_1(E))$ generated by $\mathscr C_b(\mathscr M_1(E))$. Furthermore, one can show that for any $k\in\N$, the conditional distribution of $Y_k$ (conditioned on $L_d$ being in $C$) converges weakly to the relative entropy minimizing measure $\nu_*$. An application of the maximum entropy principle now shows how a Gibbs measure arises as limiting distribution, which is exactly what happens in the case of Orlicz balls.

So let $\mathscr H:E\to \R$  be a function, also referred to as Hamiltonian or energy, and consider the corresponding average energy $\overline{\mathscr H}_d := \frac{1}{d}\sum_{i=1}^d \mathscr H(Y_i)$. Moreover, define for $R < \E_\lambda[\mathscr H]$ a set of probability measures
\[
C := \Big\{\nu \in\mathscr M_1(E)\,:\, \E_\nu[\mathscr H]\leq R \Big\}.
\]
If the set $C$ satisfies the assumptions of the maximum entropy principle, then there exists a unique probability measure $\mu_{*}\in C$ minimizing the relative entropy $H(\cdot|\lambda)$ over $C$. It is given as the following Gibbs measure at ``inverse temperature'' $\alpha_*>0$:
\[
\mu_{*}(\dint x) = \frac{e^{-\alpha_*\mathscr H(x)}}{\int_E e^{-\alpha_*\mathscr H(x)}\,\lambda(\dint x)} \lambda(\dint x),
\]
where $\alpha_*$ is such that $\E_{\mu_{*}}[\mathscr H]=R$. So wrapping everything up, the maximum entropy principle says in this case that, for each $k\in\N$ fixed,
\[
\lim_{d\to\infty} \Pro\Big[Y_k \in \cdot\,\Big|\Big.\, \overline{\mathscr H}_d\leq R\Big] = \mu_{*}.
\]

We shall now explain how this relates to our situation. As already pointed out, this derivation is not mathematically rigorous, one of the reasons being that in our setting  $\lambda$ is the Lebesgue measure, which is infinite. Let $M$ be an Orlicz function and consider, for large $d$, ``random variables'' $Y_1,Y_2,\dots,Y_d$ uniformly ``distributed'' according to the infinite Lebesgue measure $\lambda$. We are interested in the volume of the Orlicz ball
\[
B_M^d(d) := \Bigg\{ (x_1,\dots,x_d) \in\R^d \,:\, \sum_{i=1}^d M(x_i) \leq d \Bigg\}.
\]
Conditioning on $Y=(Y_1,\dots,Y_d)$ being in $B_M^d(d)$ yields the uniform distribution on $B_M^d(d)$, because for any measurable $A\subset \R^d$,
\[
\Pro\Big[Y\in A\,\Big|\Big.\, Y\in B_M^d(d)\Big] = \frac{\Pro[Y\in A\cap B_M^d(d)]}{\Pro[Y\in B_M^d(d)]} = \frac{\vol_d(A\cap B_M^d(d))}{\vol_d(B_M^d(d))}.
\]
Coming back to the maximum entropy principle, where $E=\R$, the Hamiltonian is given by the Orlicz function $M$, and $R=1$ (which is smaller than $\E_\lambda[M] = \infty$), we have, for any fixed $k\in\{1,\dots,d \}$, that
\[
\Pro\Big[Y_k\in \cdot\,\Big|\Big.\, \overline{\mathscr H}_d \leq 1 \Big]
=
\Pro\Big[Y_k \in \cdot \,\Big|\Big.\, \sum_{i=1}^d M(Y_i) \leq d \Big] \approx \mu_{*},
\]
where $\alpha_*>0$ is chosen such that $\E_{\mu_{*}}[M]=1$. So, under the ``energy constraint'' that $Y$ lies in an Orlicz ball, asymptotically the coordinates of $Y$ follow a Gibbs distribution $\mu_{*}$. So when studying random vectors in Orlicz balls, then those Gibbs distributions provide the right probabilistic set-up for investigations; let us remark that a version of Sanov's theorem with infinite underlying measure $\lambda$ has been obtained in~\cite{BS2016}.

\subsection{Volumetric properties of Orlicz balls}

As shown by Z.~Kabluchko and J.~Prochno in \cite{KP2020}, volumetric properties of unit balls in Orlicz spaces can be nicely accessed through large deviations techniques. We shall present some of the results obtained and present the idea behind their proof.

The first main result (see \cite[Theorem A]{KP2020}) contains formulas for the asymptotic logarithmic volume of $B_M^d(dR)$ and for its precise asymptotic volume. The reason is that the former result follows from an exponential tilting technique coupled with the classical central limit theorem (thus being in the spirit of the proof of Cram\'er's theorem) and provides some structural insight, which is lost in the short proof of the precise asymptotic volume where ideas and results on sharp Cram\'er large deviations by V.~V.~Petrov \cite{P1965} are employed; to be more precise, for determining the exact asymptotic volume, one use elements from the proof of \cite[Theorem 1]{P1965}.

\begin{thmalpha}\label{thm:log-volume orlicz}
Let $d\in\N$, $R\in(0,\infty)$, and $M$ be an Orlicz function. Then, as $d\to\infty$,
\[
\vol_d\big(B_M^d(dR)\big)^{1/d} \to e^{\varphi(\alpha_*)-\alpha_* R},
\]
i.e., on a logarithmic scale, we have
\[
\lim_{d\to\infty} \frac 1d \log \,\vol_d\big(B_M^d(dR)\big) = \varphi(\alpha_*)-\alpha_* R,
\]
and the  precise asymptotic volume is given by
\[
\vol_d\big(B_M^d(dR)\big) \sim \frac{1}{|\alpha_*| \sqrt{2\pi d\, \sigma_*^2}  }e^{d[\varphi(\alpha_*)-\alpha_* R]},
\]
where $\varphi:(-\infty,0)\to \R$ is given by $\varphi(\alpha) = \log \int_{\R}e^{\alpha M(x)}\,\dint x$, $\alpha_*<0$ is chosen in such a way that $\varphi'(\alpha_*) = R$, and $\sigma_*^2:=\varphi''(\alpha_*)$.
\end{thmalpha}
\begin{proof}[Idea of Proof.]
\textit{1. Asymptotic volume radius.}
The function $\varphi$ from the statement is well-defined and twice continuously differentiable on $(-\infty,0)$ with
\begin{equation}\label{eq:derivative of phi}
\varphi'(\alpha) = \frac{\int_{\R} M(x) e^{\alpha M(x)}\,\dint x }{\int_{\R}e^{\alpha M(x)}\,\dint x}\,.
\end{equation}
One can check that here exists a unique $\alpha_*:=\alpha_*(R)<0$ such that $\varphi'(\alpha_*) = R$. Then one considers iid random variables $Z_i$, $i\in\N$, with density
\[
p(x) = e^{\alpha_*M(x) - \varphi(\alpha_*)}, \qquad x\in\R,
\]
and a direct computation using \eqref{eq:derivative of phi} shows that $\E[M(Z_1)] = R$ and $\Var[M(Z_1)] = \varphi''(\alpha_*)>0$. 
Then one considers a shifted version of those random variables, more precisely, the iid random variables $Y_i := M(Z_i)-R$, $i\in\N$, which satisfy $\E[Y_1] = 0$ and $\Var[Y_1] =\Var[M(Z_1)] = \varphi''(\alpha_*) =: \sigma_*^2 $. Then
\begin{align*}
\vol_d\big(B_M^d(dR)\big) & = \int_{\R^d} \mathbbm 1_{B_M^d(dR)} (x_1,\dots,x_d)\, d\lambda^d(x_1\dots,x_d) \cr
& = \int_{\R^d}  \mathbbm 1_{B_M^d(dR)} (x_1,\dots,x_d) e^{-\alpha_*\sum_{i=1}^dM(x_i) + d\varphi(\alpha_*)} \prod_{i=1}^dp(x_i)\, d\lambda(x_1)\dots d\lambda(x_d) \cr
& = \E\Big[\mathbbm 1_{B_M^d(dR)}(Z_1,\dots,Z_d) e^{-\alpha_*\sum_{i=1}^dM(Z_i)+d\varphi(\alpha_*)}\Big] \cr
& = \E\Big[\mathbbm 1_{\{\sum_{i=1}^dY_i \leq 0\}} e^{-\alpha_*\sum_{i=1}^dY_i - d\alpha_*R+d\varphi(\alpha_*)}\Big] \cr
& = e^{d\big(\varphi(\alpha_*)-\alpha_* R\big)} \E\Big[ \mathbbm 1_{\{\sum_{i=1}^dY_i \leq 0\}} e^{-\alpha_*\sum_{i=1}^dY_i}\Big].
\end{align*}

This expression can be bounded from below and above using the central limit theorem. In fact, for every $c\in(0,\infty)$,
\begin{align*}
\E\Big[ \mathbbm 1_{\big\{\sum_{i=1}^dY_i \leq 0\big\}} e^{-\alpha_*\sum_{i=1}^dY_i}\Big] \geq \E \Big[ \mathbbm 1_{\big\{-c\sqrt{d} \leq \sum_{i=1}^dY_i \leq 0\big\}} e^{c \alpha_*\sqrt{d}} \Big]
= e^{c \alpha_*\sqrt{d}} \Pro\Big[ \frac{1}{\sqrt{d}}\sum_{i=1}^dY_i\in[-c,0]\Big],
\end{align*}
where it was used that $-\alpha_*>0$, and so one has
\begin{align*}
\vol_d\big(B_M^d(dR)\big) \geq e^{d\big(\varphi(\alpha_*)-\alpha_* R\big)} e^{c \alpha_*\sqrt{d}} \Pro\Big[ \frac{1}{\sqrt{d}}\sum_{i=1}^dY_i\in[-c,0]\Big].
\end{align*}
By the central limit theorem,
\[
\Pro\Big[ \frac{1}{\sqrt{d}}\sum_{i=1}^dY_i\in[-c,0]\Big] \stackrel{d\to\infty}{\longrightarrow} \mathscr N(0,\sigma_*^2)([0,c]).
\]
Similar to the lower bound, one can obtain the upper bound, namely,
  \[
   \E\Big[ \mathbbm 1_{\big\{\sum_{i=1}^dY_i \leq 0\big\}} e^{-\alpha_*\sum_{i=1}^dY_i}\Big] = \E\Big[ \mathbbm 1_{\big\{\frac{1}{\sqrt{d}}\sum_{i=1}^dY_i \leq 0\big\}} e^{-\alpha_*\sum_{i=1}^dY_i}\Big] \leq \Pro\Big[ \frac{1}{\sqrt{d}}\sum_{i=1}^dY_i\in (-\infty,0]\Big]
  \]
and therefore,
  \[
    \vol_d\big(B_M^d(dR)\big) \leq e^{d\big(\varphi(\alpha_*)-\alpha_* R\big)} \Pro\Big[ \frac{1}{\sqrt{d}}\sum_{i=1}^dY_i\in (-\infty,0]\Big].
  \]
The central limit theorem implies that
  \[
    \Pro\Big[ \frac{1}{\sqrt{d}}\sum_{i=1}^dY_i\in (-\infty,0]\Big]\stackrel{d\to\infty}{\longrightarrow} \mathscr N(0,\sigma_*^2)((-\infty,0]) = \frac{1}{2}.
  \]
Putting things together, for any $c\in(0,\infty)$,
  \[
    e^{d\big(\varphi(\alpha_*)-\alpha_* R\big)} e^{c \alpha_*\sqrt{d}} \Pro\Big[ \frac{1}{\sqrt{d}}\sum_{i=1}^dY_i\in[-c,0]\Big] \leq \vol_d\big(B_M^d(dR)\big) \leq e^{d\big(\varphi(\alpha_*)-\alpha_* R\big)} \Pro\Big[ \frac{1}{\sqrt{d}}\sum_{i=1}^dY_i\in (-\infty,0]\Big],
  \]
which implies the result taking the $d$-th root, letting $d\to\infty$, and using the central limit theorem.
\vskip 2mm
\noindent\textit{2. Precise asymptotic volume.}  As shown above, one has
\begin{align}\label{eq:volume of B_M^d(dR)}
\vol_d\big(B_M^d(dR)\big)
& = e^{d\big(\varphi(\alpha_*)-\alpha_* R\big)} \E\Big[ \mathbbm 1_{\{\sum_{i=1}^dY_i \leq 0\}} e^{-\alpha_*\sum_{i=1}^dY_i}\Big],
\end{align}
where $Y_i:= M(Z_i)-R$, $i\in\{1,\dots,d\}$, with $Z_1,\dots, Z_d$ independent and having Lebesgue density $p(x) = \exp(\alpha_*M(x)-\varphi(\alpha_*))$, $x\in\R$.  In particular, the distribution of $Y_1$ is not concentrated on a lattice.
Recall also that $\E[Y_1]=0$ and  $\Var[Y_1] = \varphi''(\alpha_*) =: \sigma_*^2$.
Let $\mu_d$ be the distribution of $\sum_{i=1}^dY_i$. Then one has 
\[
\E\Bigg[ \mathbbm 1_{\{\sum_{i=1}^dY_i \leq 0\}} e^{-\alpha_*\sum_{i=1}^dY_i}\Bigg] = \int_{-\infty}^0 e^{-\alpha_* y}\,\mu_d(\dint y).
\]
The integral on the right-hand side is exactly the integral that appears in \cite[Equation 4.11]{P1965}, just with a different sign. As is demonstrated in Petrov's proof using a Berry--Esseen argument, for $d\to\infty$, this integral can be evaluated (see \cite[Equation 4.19]{P1965}) as follows,
\[
\int_{-\infty}^0 e^{-\alpha_* y}\,\mu_d(\dint y) \sim \frac{1}{|\alpha_*|\sqrt{2\pi d \sigma_*^2}}.
\]
Combining this with \eqref{eq:volume of B_M^d(dR)}, one obtains that, as $d\to\infty$, 
\[
\vol_d\big(B_M^d(dR)\big) \sim \frac{1}{|\alpha_*| \sqrt{2\pi d\, \sigma_*^2}  }e^{d[\varphi(\alpha_*)-\alpha_* R]},
\]
which completes the proof.
\end{proof}

In their second main result (see \cite[Theorem B]{KP2020}), the authors determine the asymptotic behavior of the volume of intersections of two Orlicz balls when the dimension tends to infinity. This generalizes a classical result of G.~Schechtman and M.~Schmuckenschl\"ager \cite{SS1991}. For the mere fact that the quantities that appear in the following statement are all well-defined, we refer the reader directly to \cite{KP2020}. 

\begin{thmalpha}\label{thm:dichotomy}
Let $M_1$ and $M_2$ be two Orlicz functions and $R_1,R_2\in(0,\infty)$. Consider
\[
\varphi_1:=\varphi_{M_1}:(-\infty,0)\to\R,\qquad \varphi_1(\alpha) = \log \int_{\R}e^{\alpha M_1(x)}\,\dint x
\]
and choose $\alpha_*<0$ such that $\varphi_1'(\alpha_*) = R_1$. Define the Gibbs density
\[
p_1(x) := e^{\alpha_*M_1(x) - \varphi_1(\alpha_*)}, \qquad x\in\R.
\]
Then, we have
\[
\frac{\vol_d\big(B_{M_{1}}^d(dR_1) \cap B_{M_2}^d(dR_2)\big)}{\vol_d\big(B_{M_1}^d(dR_1)\big)} \stackrel{d\to\infty}{\longrightarrow}
\begin{cases}
0 & :\, \int_{\R}M_2(x)p_1(x)\,\dint x>R_2 \\
1 & :\, \int_{\R}M_2(x)p_1(x)\,\dint x<R_2\,\,\text{and}\,\,\int_{\R}M_2^2(x)p_1(x)\,\dint x<\infty
\end{cases}
\]
and the speed of convergence to zero is exponential in $d$.
\end{thmalpha}
\begin{proof}[Idea of Proof.]
Let $Z_1,\dots,Z_d$ be iid random variables with distribution given by the density
\[
p_1(x) := e^{\alpha_*M_1(x) - \varphi_1(\alpha_*)}, \qquad x\in\R,
\]
where for $\alpha<0$, we have
\[
\varphi_1(\alpha) := \log \int_{\R}e^{\alpha M_1(x)}\,\dint x
\]
and where $\alpha_*<0$ is now chosen such that $\varphi_1'(\alpha_*) = R_1$.
Then
\begin{align*}
& \vol_d\big(B_{M_1}^d(dR_1)\cap B_{M_2}^d(dR_2)\big) \cr
& = \int_{\R^d} \mathbbm 1_{B_{M_1}^d(dR_1)} (x_1,\dots,x_d)\mathbbm 1_{B_{M_2}^d(dR_2)} (x_1,\dots,x_d)\, d\lambda^d(x_1\dots,x_d) \cr
& = \int_{\R^d}  \mathbbm 1_{B_{M_1}^d(dR_1)} (x_1,\dots,x_d) \mathbbm 1_{B_{M_2}^d(dR_2)}(x_1,\dots,x_d)  e^{-\alpha_*\sum_{i=1}^dM_1(x_i) + d\varphi_1(\alpha_*)} \prod_{i=1}^dp_1(x_i)\, d\lambda(x_1)\dots d\lambda(x_d) \cr
& = \E\Big[\mathbbm 1_{B_{M_1}^d(dR_1)}(Z_1,\dots,Z_d) \mathbbm 1_{B_{M_2}^d(dR_2)}(Z_1,\dots,Z_d) e^{-\alpha_*\sum_{i=1}^dM_1(Z_i)+d\varphi_1(\alpha_*)}\Big].
\end{align*}
One now modifies the probabilistic argument seen before and defines random variables $Y_1^{(1)},\dots,Y_d^{(1)}$ and $Y_1^{(2)},\dots,Y_d^{(2)}$ letting
\[
Y_i^{(1)} := M_1(Z_i) - R_1 \qquad\text{and}\qquad Y_i^{(2)}:= M_2(Z_i) - \int_{\R}M_2(x)p_1(x)\,\dint x
\]
for $i\in\{1,\dots,d\}$. Then $Y_1^{(1)},\dots,Y_d^{(1)}$ are independent and also $Y_1^{(2)},\dots,Y_d^{(2)}$ are independent. Moreover, $\E[Y_1^{(1)}] = 0 = \E[Y_1^{(2)}]$ and $\Var[Y_1^{(1)}] = \varphi_1''(\alpha_*)$ while $\Var[Y_1^{(2)}] = \E[(Y_1^{(2)})^2] = \Var[M_2(Z_1)]$. Using those transformations of the original random variables, one can write
\begin{align*}
\vol_d\big(B_{M_1}^d(dR_1)\cap B_{M_2}^d(dR_2)\big) & = \E\Bigg[\mathbbm 1_{\big\{\sum_{i=1}^dY_i^{(1)}\leq 0 \big\}} \mathbbm 1_{\big\{\sum_{i=1}^dY_i^{(2)} \leq d[R_2 - \int_{\R}M_2(x)p_1(x)\,\dint x] \big\}} e^{-\alpha_*\sum_{i=1}^dY_i^{(1)} - d\alpha_*R_1+d\varphi_1(\alpha_*)}\Bigg] \cr
& =e^{d[\varphi_1(\alpha_*) - \alpha_*R_1]} \E\Bigg[\mathbbm 1_{\big\{\sum_{i=1}^dY_i^{(1)}\leq 0 \big\}} \mathbbm 1_{\big\{\sum_{i=1}^dY_i^{(2)} \leq d[R_2 - \int_{\R}M_2(x)p_1(x)\,\dint x] \big\}} e^{-\alpha_*\sum_{i=1}^dY_i^{(1)}}\Bigg].
\end{align*}
The idea is that  by the strong law of large numbers,
\[
\frac{1}{d}\sum_{i=1}^d Y_i^{(2)}\stackrel{a.s.}{\longrightarrow} 0,\qquad\text{as }\,d\to\infty.
\]
Thus, one can see that, as $d\to\infty$, the event
\[
\Bigg\{\sum_{i=1}^dY_i^{(2)} \leq d\Big[R_2 - \int_{\R}M_2(x)p_1(x)\,\dint x\Big] \Bigg\}
\]
occurs with probability going to $0$ (even exponentially fast) if $\int_{\R}M_2(x)p_1(x)\,\dint x>R_2$ and probability going to $1$ if
$\int_{\R}M_2(x)p_1(x)\,\dint x<R_2$.

Consider the case $\int_{\R}M_2(x)p_1(x)\,\dint x>R_2$. To obtain an upper bound for the volume of the intersection one observes that
$$
\E\Bigg[\mathbbm 1_{\big\{\sum_{i=1}^dY_i^{(1)}\leq 0 \big\}} \mathbbm 1_{\big\{\sum_{i=1}^dY_i^{(2)} \leq d[R_2 - \int_{\R}M_2(x)p_1(x)\,\dint x] \big\}} e^{-\alpha_*\sum_{i=1}^dY_i^{(1)}}\Bigg]
\leq
\Pro\Bigg[\sum_{i=1}^dY_i^{(2)} \leq d[R_2 - \int_{\R}M_2(x)p_1(x)\,\dint x]\Bigg],
$$
which goes to $0$ exponentially fast by Cram\'er's theorem (see \cite{DZ2010}) since $\E[Y_1^{(2)}]=0$ and the negative exponential moments are finite for $Y_1^{(2)}$. It follows from the lower bound for $\vol_d\big(B_{M_1}^d(dR_1)\big)$ in the proof of Theorem~\ref{thm:log-volume orlicz} that, as $d\to\infty$,
\begin{equation}\label{eq:ratio intersection}
\frac{\vol_d\big(B_{M_1}^d(dR_1)\cap B_{M_2}^d(dR_2)\big)}{\vol_d\big(B_{M_1}^d(dR_1)\big)}
\end{equation}
goes to $0$ at an exponential rate by Cram\'er's theorem. 

Next one considers the case $\int_{\R}M_2(x)p_1(x)\,\dint x < R_2$. Then
\begin{equation*}
\eqref{eq:ratio intersection}= 1- \frac{ \E\Bigg[\mathbbm 1_{\big\{\sum_{i=1}^dY_i^{(1)}\leq 0 \big\}} \mathbbm 1_{\big\{\sum_{i=1}^dY_i^{(2)} > d[R_2 - \int_{\R}M_2(x)p_1(x)\,\dint x] \big\}} e^{-\alpha_*\sum_{i=1}^dY_i^{(1)}}\Bigg]}{ \E\Bigg[\mathbbm 1_{\big\{\sum_{i=1}^dY_i^{(1)}\leq 0 \big\}} e^{-\alpha_*\sum_{i=1}^dY_i^{(1)}}\Bigg]}.
\end{equation*}
One needs to show that the quotient of expectations on the right-hand side goes to $0$. To this end, one observes that the expectation in the numerator can be estimated as follows,
$$
\E\Bigg[\mathbbm 1_{\big\{\sum_{i=1}^dY_i^{(1)}\leq 0 \big\}} \mathbbm 1_{\big\{\sum_{i=1}^dY_i^{(2)} > d[R_2 - \int_{\R}M_2(x)p_1(x)\,\dint x] \big\}} e^{-\alpha_*\sum_{i=1}^dY_i^{(1)}}\Bigg]
\leq
\Pro\Bigg[\sum_{i=1}^dY_i^{(2)} > d[R_2 - \int_{\R}M_2(x)p_1(x)\,\dint x]\Bigg],
$$
which can be estimated above by $O(1/d)$ via Chebyshev's inequality. On the other hand, the proof of Theorem \ref{thm:log-volume orlicz} shows that 
\[
\E\Bigg[ \mathbbm 1_{\{\sum_{i=1}^dY_i^{(1)} \leq 0\}} e^{-\alpha_*\sum_{i=1}^dY_i^{(1)}}\Bigg]\sim \frac{1}{|\alpha_*|\sqrt{2\pi d \sigma_*^2}},
\]
where $\sigma_*^2=\varphi_1''(\alpha_*)$.
Therefore, one obtains that the ratio in \eqref{eq:ratio intersection} goes to $1$ as $d\to\infty$, and thus one obtains the desired dichotomy.
\end{proof}

\begin{rmk}
To obtain the phase transition in Theorem \ref{thm:dichotomy}, we have seen that it is enough to know the asymptotic logarithmic volume of Orlicz balls and their intersections. To deal with the critical case at the threshold, however, the precise asymptotics would have been needed. This was an open problem, which has recently been solved by L.~Frühwirth and J.~Prochno \cite{FP2024_sharp}. In fact, the authors proved (under some mild conditions on the Orlicz functions) a number of sharp concentration phenomena in Orlicz spaces, which imply that 
\[
\frac{\vol_d\big(B_{M_{1}}^d(dR_1) \cap B_{M_2}^d(dR_2)\big)}{\vol_d\big(B_{M_1}^d(dR_1)\big)} \stackrel{d\to\infty}{\longrightarrow} \frac{1}{2}.
\]
\end{rmk}

\subsection{Sanov-type large deviations}

We shall briefly discuss here another result concerning large deviations in the world of Orlicz spaces, namely a Sanov-type large deviations principle for the empirical measure of the coordinates of vectors chosen uniformly at random from Orlicz balls. This result has been obtained by L.~Frühwirth and J.~Prochno in \cite{FP2024}; for the conditional limit theorems that can be deduced from the level--2 large deviations result and also a more general version of the result presented below, we refer the reader directly to \cite{FP2024}. 

Again, if $M: \R \rightarrow [0,\infty)$ is an Orlicz function and $ \alpha \in (- \infty, 0)$, we shall consider the corresponding Gibbs measure $\mu_{M,\alpha}$ on the Borel sets of $\R$ given by 
  \begin{equation*} 
    \mu_{M,\alpha}(\dint x) := \frac{e^{\alpha M(x)}}{\int_{\R}e^{\alpha M(t)}\,\lambda(\dint t)} \,\lambda(\dint x).
  \end{equation*}
For $\mu \in \mathscr{M} _1( \R )$, we define the Orlicz-moment-mapping 
  \begin{equation*}
    \mathbb M_M( \mu ) := \int_{\R} M(x)  \mu (dx).
  \end{equation*}

\begin{thmalpha}
	\label{ThmLDPEmpMeasuresOrlicz}
	Let $M: \R \rightarrow [0,\infty)$ be an Orlicz function, $R\in(0,\infty)$, and consider a sequence of random vectors $X^{(d)} := \big ( X^{(d)}_1,\ldots, X^{(d)}_d \big )\sim \Uni (B_M^d(dR))$, $d\in\N$. Then, there exists a unique $\alpha^* \in (-\infty,0)$ such that the sequence of empirical measures
	\begin{equation*}
	\label{EqEmpMeasureOrlicz}
	L_d := \frac{1}{d} \sum_{i=1}^{d} \delta_{ X^{(d)}_i} ,\quad{ }  d \in \N ,
	\end{equation*}
	satisfies an LDP in $\mathscr{M}_1( \R )$ with the strictly convex good rate function $ \mathbb{I}_M: \mathscr{M}_1(\R) \rightarrow [0,\infty]$, where
	\begin{equation*}
	\label{EqDefGRFEmpMeasOrlicz}
	\mathbb{I}_M(\mu) := \begin{cases} 
	H( \mu | \mu_{ M, \alpha^* }) + \alpha^* \big[ \mathbb M_M( \mu ) -R  \big] & : \mathbb M_M( \mu ) \leq R \\
	\infty &  : \text{otherwise}.
	\end{cases}
	\end{equation*}
\end{thmalpha}

\begin{rmk}
Very roughly speaking, in the proof of Theorem \ref{ThmLDPEmpMeasuresOrlicz} one combines ideas from the proof of Theorem \ref{thm:log-volume orlicz} with a general version of the Gärtner--Ellis theorem (we presented a classical version as Theorem \ref{thm:gaertnerellis}) that can be found, for instance, in \cite[Corollary 4.16.14]{DZ2010}. The latter reads as follows: 
let $( \xi_n)_{n \in \N}$ be a sequence of exponentially tight random variables on the locally convex Hausdorff topological vector space $\mathcal{E}$. Suppose the Gärtner--Ellis limit
\begin{equation*}
\Lambda( \lambda ):= \lim_{n \rightarrow \infty } \frac{1}{n} \log \mathbb{E} \Big [ e^{ n  \lambda ( \xi_n ) } \Big ], \quad \lambda \in \mathcal{E}^{*},
\end{equation*}
exists in $\R $ and is Gateaux-differentiable, where $\mathcal{E}^{*}$ is the topological dual of $\mathcal{E}$. Then, $( \xi_n)_{n \in \N}$ satisfies an LDP in $\mathcal{E}$ with the convex good rate function $\Lambda^{*}: \mathcal{E} \rightarrow [0, \infty]$, given by
\begin{equation*}
\Lambda^{*}( x) := \sup_{ \lambda \in \mathcal{E}^{*} } \big [  \lambda(x) - \Lambda( \lambda )  \big ].
\end{equation*}
\end{rmk}

\begin{rmk}
Beyond the Sanov-type LDP on $\mathscr{M}_1( \R )$ endowed with the weak topology, the LDP can actually be obtained on generalized Wasserstein spaces, which is indeed essential for the proof of the conditional limit theorem in \cite{FP2024}.
\end{rmk}

\section{Large deviations results under an asymptotic thin-shell condition}

Let us close this survey with a wonderful paper by S.~S.~Kim, Y.-T.~Liao, and K.~Ramanan \cite{KLR2019} that establishes LDPs under asymptotic thin-shell-type conditions. As already mentioned in the introduction, this can been seen as a large deviations counterpart to Klartags CLT for convex bodies. Whereas the CLT for convex bodies shows that fluctuations of most random projections of high-dimensional vectors satisfying a thin-shell condition can be characterized as almost Gaussian, the work \cite{KLR2019} characterizes tail behavior (at
the level of annealed LDPs) for projections and their associated norms onto, possibly growing, random subspaces of high-dimensional random vectors satisfying an asymptotic thin-shell condition. In particular, the authors' work unites disparate examples for LDPs in geometric functional analysis under a common framework, goeing beyond the specific setting of distributions on $\ell_p^d$-balls. Probably most important in this respect are the cases of LDPs for some general classes of Gibbs distribution with super-quadratic potentials and random vectors in super-quadratic Orlicz balls. 


Before presenting some of the results in detail, let us be a bit more specific about the contributions of \cite{KLR2019}. For any sequence of random vectors $(X^{(n)})_{n\in\N}$ whose scaled Euclidean norms satisfy a certain LDP (denoted by Assumption $A$ and its specific case Assumption $A^*$ in \cite{KLR2019}), the authors characterize
the tail behavior of the corresponding sequence of orthogonal projections of $X^{(n)}$ onto a random
$k_n$-dimensional basis, $k_n \leq n$, drawn with respect to the Haar measure on the Stiefel manifold $\mathbb V_{n,k_n}$ of orthonormal $k_n$-frames in $\R^n$, as the dimension $n\to\infty$ with $k_n/n\to \lambda\in [0, 1]$. 
A slightly stronger version of the LDP Assumption A, where one uses that the good rate function in the LDP has a unique minimum, can be viewed as an asymptotic thin-shell condition, in the sense that it implies that, for all sufficiently large $n\in\N$, the random vector $X^{(n)}$ satisfies the classical thin-shell condition \cite[Equation (1)]{ABP2003} (see the discussion in \cite[Section 2.1]{KLR2019}). It is important to note, however, that for growing subspaces, in contrast to CLT results where approximate Gaussian marginals (in bounded-Lipschitz distance) occur only for $k_n < 2 \log n/\log \log n$ \cite{ME2012} (or $k_n\sim n^\alpha$  if one assumes additional regularity of $X^{(n)}$ such as log-concavity \cite{K2007_power-law}), the annealed LDPs indicate three crucial regimes for subspace dimensions $(k_n)_{n\in\N}$, namely, constant, sublinear, and linear; the specific case of $\ell_p^n$-balls had been treated in \cite{APT2018}.

\subsection{The crucial regimes and asymptotic thin-shell-type assumptions}\label{subsec:regimes and conditions}

Before presenting some of the main results of \cite{KLR2019}, we introduce the mathematical set-up and the asymptotic-thin-shell-type conditions that are used. For $n\in\N$, we consider a random vector $X^{(n)}$ in $\R^n$, and for $k\in\N$, shall denote by $\id_k$ the $k\times k$ identity matrix. The Stiefel manifold of orthonormal $k$-frames in $\R^n$ shall be denoted by $\mathbb V_{n,k}$ and we recall that $\mathbb V_{n,k} = \{ A\in\R^{n\times k} \,:\, A^TA=\id_k \}$. 

We are interested in orthogonal projections of vectors $X^{(n)}\in\R^n$ onto random $k_n$-dimensional subspaces, where $1\leq k_n\leq n$. So for fixed $n\in\N$ and $1\leq k_n\leq n$, we let $A_{n,k_n}=(A_{n,k_n}(i,j))_{i,j=1}^{n,k_n}$ be an $n\times k_n$ random matrix  drawn with respect to the Haar probability measure on $\mathbb V_{n,k_n}$; we assume this random matrix is independent of $X^{(n)}$ and that the random objects are defined on the same probability space. The random matrix $A_{n,k_n}^T$ then linearly projects $X^{(n)}$ in $\R^n$ onto a $k_n$-dimensional vector. The goal will be to analyze the large deviations behavior of the coordinates of $A_{n,k_n}^TX^{(n)}$ (as $n\to\infty$) in the following three regimes: for a sequence $(k_n)_{n\in\N}\in\N^{\N}$, we say
\vskip 2mm
(1)  $(k_n)_{n\in\N}$ is \emph{constant} at $k\in\N$, denoted $k_n\equiv k$, if and only if $k_n=k$ for all $n\in\N$;
\vskip 1mm
(2)  $(k_n)_{n\in\N}$  \emph{grows sub-linearly}, denoted $1\ll k_n\ll n$, if and only if $k_n\to\infty$ and $\frac{k_n}{n}\to 0$;
\vskip 1mm
(3)  $(k_n)_{n\in\N}$  \emph{grows linearly} with rate $\lambda\in(0,1]$, denoted $k_n\sim \lambda n$, if and only if $\frac{k_n}{n}\to \lambda$.
\vskip 2mm

When the subspace dimensions $(k_n)_{n\in\N}$ are constant at $k\in\N$, then one clearly may investigate the large deviations behavior of the sequence $A_{n,k_n}^TX^{(n)}=A_{n,k}^TX^{(n)}$, $n\in\N$, of random vectors in $\R^k$. However, when the subspace dimensions $(k_n)_{n\in\N}$ tend to infinity (as is the case in the other two regimes), then in order to be able to even pose a sensible LDP question, one must first embed the random sequence $(A_{n,k_n}^TX^{(n)})_{n\in\N}$ of vectors in changing dimensions into a common topological space. Now, since the law of $A_{n,k_n}$ is invariant under permutation of its $k_n$ columns (which means that the coordinates of $A_{n,k_n}^TX^{(n)}$ are exchangeable), the essential distributional properties of the coordinates of the projection are encoded by the empirical measures of the coordinates of the projection. A suitable setting to prove LDPs is therefore the one of empirical measures of the coordinates of $A_{n,k_n}^TX^{(n)}$, $n\in\N$, more precisely, 
  \[
    L_{k_n} := \frac{1}{k_n} \sum_{j=1}^{k_n} \delta_{(A_{n,k_n}^TX^{(n)})_j}, \qquad n\in\N. 
  \]
  
We can now present (some of) the thin-shell-type conditions the authors introduced:
\vskip 2mm
\textbf{Assumption A.}  The random sequence $(\| X^{(n)}\|_2/\sqrt{n})_{n\in\N}$ satisfies an LDP in $\R$ at speed $s_n$ with good rate function $\mathbb J_X:\R\to [0,\infty]$.
\vskip 1mm
\textbf{Assumption A$^*$.}  The random sequence $(\| X^{(n)}\|_2/\sqrt{n})_{n\in\N}$ satisfies an LDP in $\R$ at speed $n$ with good rate function $\mathbb J_X:\R\to [0,\infty]$.
\vskip 1mm
\textbf{Assumption B.}  There exists a positive sequence $(s_n)_{n\in\N}$ with $\lim_{n\to\infty} s_n = \lim_{n\to\infty} n/s_n=\infty$ such that the random sequence $(\sqrt{s_n} \| X^{(n)}\|_2/n)_{n\in\N}$ satisfies an LDP in $\R$ at speed $s_n$ with good rate function $\mathbb J_X:\R\to [0,\infty]$.
\vskip 2mm

\begin{rmk}
Assumptions $A^*$ and $B$ are modifications of Assumption $A$, which are required to deal with different situations or regimes of subspace dimensions; in fact, the authors introduce another assumption, namely Assumption $C$, which is a refinement of Assumption $A$ required in the sub-linear regime whenever Assumption $A^*$ is not met, but we shall not discuss this here and refer to \cite[Subsection 2.3.1]{KLR2019} directly. 
\end{rmk}

\begin{rmk}
The specific case $s_n=n$, which is referred to as Assumption $A^*$, includes a number of distributions of interest in geometric functional analysis, e.g., certain types of product measures, uniform distributions on $n^{1/p}\B_p^n$ for $p\geq 2$ or, more generally, uniform distributions on super-quadratic (scaled) Orlicz balls and classes of Gibbs distributions with super-quadratic potential.

This already hints at the fact that Assumption $A^*$ is no longer met when looking at $n^{1/p}\B_p^n$ for $1\leq p<2$, which explains in part the appearance of the more general Assumption $A$; for more details we refer to \cite[Remark 2.3]{KLR2019}.
\end{rmk}

\subsection{Main result in the constant regime}

In the constant regime one requires the two assumptions Assumption $A^*$ and Assumption $B$ introduced above. The following theorem corresponds to \cite[Theorem 2.7]{KLR2019}.

  \begin{thmalpha}\label{thm:klr constant regime}
    Let $k\in\N$ and assume that $(k_n)_{n\in\N}\in\N^{\N}$ satisfies $k_n\equiv k$. If either Assumption $A^*$ or Assumption $B$ with speed $(s_n)_{n\in\N}$ are met with a good rate function $\mathbb J_X$, then $(A^T_{n,k}X^{(n)}/\sqrt{n})_{n\in\N}$ satisfies an LDP in $\R^k$ at speed $(s_n)_{n\in\N}$ with good rate function $\mathbb I_k:\R^k \to [0,\infty]$ given by
      \[
        \mathbb I_k (x) := 
        \begin{cases}
          \inf_{c\in(0,1)} \Big[ \mathbb J_X(\|x\|_2/c) - \frac{1}{2}\log(1-c^2)  \Big] &: \text{under Assumption $A^*$} \cr
           \inf_{c>0} \Big[ \mathbb J_X(\|x\|_2/c) + \frac{c^2}{2}  \Big] &: \text{under Assumption $B$}.
        \end{cases}
      \]
  \end{thmalpha}
  
  Let us continue with two remarks of which one sheds some light upon the specific form of the LDP just presented and the other presents an interesting consequence of the previous theorem.
  
 \begin{rmk}  
The two parts in the rate functions essentially correspond to the radial component of the random sequence (which is what $\mathbb J_X$ represents) and the angular component of the random sequence (which is what $- \frac{1}{2}\log(1-c^2)$ and $\frac{c^2}{2} $ represent).  
 \end{rmk} 
  
\begin{rmk} 
  As a corollary to the previous theorem, the authors are also able to deduce LDPs for general $\ell_q$-(quasi)norms with $q\in(0,\infty)$ and refer to \cite[Corollary 2.8]{KLR2019}, which is in fact a direct consequence of the contraction principle under the continuous mapping $\R^k\ni x \mapsto \|x\|_q$.  
\end{rmk}

\begin{proof}[Idea of Proof of Theorem \ref{thm:klr constant regime} under Assumption $A^*$.]
In a first step, the authors show in \cite[Lemma 4.1]{KLR2019} that the top row of $A_{n,k}$, $k,n\in\N$ with $k\leq n$, satisfies the distributional identity
  \[
    A_{n,k}(1,\cdot) := \big(A_{n,k}(1,1),\ldots,A_{n,k}(1,k)\big) \eqdistr \frac{(Z_1,\ldots,Z_k)}{\| Z^{(n)} \|_2},
  \]
where $Z^{(n)}:= (Z_1,\ldots,Z_n)\in\R^n$ for a sequence $(Z_i)_{i\in\N}$ of iid standard Gaussians. This allows one to apply a result of F.~Barthe, F.~Gamboa, L.~V.~Lozada-Chang, and A.~Rouault \cite[Theorem 3.4]{BGLR2010}, which then shows that the random sequence $(A_{n,k}(1,\cdot))_{n\in\N}$ satisfies an LDP in $\R^k$ at speed $n$ with good rate function 
  \[
    J_k(y) := 
    \begin{cases}
    -\frac{1}{2}\log(1-\|y\|_2^2) & : \|y\|_2\leq 1 \cr
    \infty &: \text{otherwise}.
    \end{cases}
  \]
From the independence of $(X^{(n)})_{n\in\N}$ and $(A_{n,k})_{n\in\N}$, and the assumption that $(\|X^{(n)}\|_2/\sqrt{n})_{n\in\N}$ satisfies an LDP at speed $n$ with good rate function $\mathbb J_X$ (Assumption $A^*$), the authors then deduce from \cite[Lemma 1.9]{KLR2019} (which is a more general version of Lemma \ref{lem: ldp pairs of independent random objects} above, which under $A^*$ is actually sufficient) that the random sequence of pairs $(A_{n,k}(1,\cdot), \|X^{(n)}\|_2/\sqrt{n})_{n\in\N}$ satisfies an LDP at speed $n$ with good rate function $J_k(y) + \mathbb J_X(\alpha)$, $y\in\R^k$, $\alpha\in[0,\infty)$.

In a second step, one can now apply the contraction principle with the continuous mapping $\R^k\times [0,\infty)\ni (y,\alpha) \mapsto \alpha y \in\R^k$ to show that the random sequence $(A^T_{n,k}X^{(n)}/\sqrt{n})_{n\in\N}$ satisfies an LDP at speed $n$ with good rate function, which, for $x\in\R^k$, is given by
  \[
    \mathbb I_k(x) := \inf_{y\in\R^k,\, z\in\R} \Big\{ -\frac{1}{2}\log(1-\|y\|^2_2) + J_X(z) \,:\, x=yz,\, \|y\|_2\leq 1,\, z\geq 0 \Big\}.
  \]
 In a third and last step, it is only required to write this rate function in the appropriate form. Here one uses that without loss of generality the range of $z$ in the infimum can be changed to $z>0$, the substitution $y=x/z$, and the fact that $\|y\|_2\leq 1$ is equivalent to $\|x\|_2\leq z$, to show 
   \[
      \mathbb I_k(x) = \inf_{z\geq \|x\|_2} \Big[ -\frac{1}{2}\log\Big(1-\frac{\|x\|^2_2}{z^2}\Big) + J_X(z)\Big] = 
      \inf_{z> \|x\|_2} \Big[ -\frac{1}{2}\log\Big(1-\frac{\|x\|^2_2}{z^2}\Big) + J_X(z)\Big].
   \]
Writing the latter in terms of $c=\|x\|_2/z$, one obtains the form of the good rate function as in the statement of the theorem. 
\end{proof}

\subsection{Main results in the sub-linear regime}
 
In the sub-linear regime, when working with $(L_{k_n})_{n\in\N}$  in spaces of probability measures, one requires Assumption $A$ to obtain the desired LDP. When considering sequences of scaled Euclidean norms of the random projections, one requires Assumption $A^*$ and what is referred to as Assumption $C$, but we shall neither introduce this assumption nor present the corresponding result \cite[Theorem 2.11]{KLR2019} here. The following theorem corresponds to \cite[Theorem 2.9]{KLR2019} and we shall write $\mathcal M_1^q(\R)$ for the subset of probability measures $\nu$ on $\R$ having finite $q$-th moments, i.e., $\int_\R |x|^q\,\nu(\dint x)<\infty$. A sequence $(\nu_n)_{n\in\N}$ in $\mathcal M_1^q(\R)$ converges to a measure $\nu\in \mathcal M_1^q(\R)$ with respect to the $q$-Wasserstein topology if and only if the sequence converges weakly and the $q$-th moments converge as well; this topology can actually be metrized. For $\sigma>0$, we shall write $\mu_{2,\sigma}$ for the Gaussian measure on $\R$ with mean $0$ and variance $\sigma^2$.

\begin{thmalpha}
 Let $(k_n)_{n\in\N}\in\N^{\N}$ be such that $1\ll k_n\ll n$ and assume that Assumption $A$ is satisfied with associated speed $s_n$ and good rate function $\mathbb J_X$. In addition, suppose that $\mathbb J_X$ has a unique minimum at $m>0$. Then, for every $q\in[1,2)$, the following hold:
 \vskip 1mm
 (1)  If $s_n\gg k_n$, then $(L_{k_n})_{n\in\N}$ satisfies an LDP in the space $\mathcal M_1^q(\R)$ at speed $k_n$ with good rate function $\rate_{L,k_n}: \mathcal M_1^q(\R)\to [0,\infty]$ given by 
   \[
     \rate_{L,k_n}(\mu) := H(\mu|\mu_{2,m}).
   \]
 \vskip 1mm
 (2) If $s_n= k_n$, then $(L_{k_n})_{n\in\N}$ satisfies an LDP in the space $\mathcal M_1^q(\R)$ at speed $k_n$ with good rate function $\rate_{L,k_n}: \mathcal M_1^q(\R)\to [0,\infty]$ given by 
   \[
     \rate_{L,k_n}(\mu) := \inf_{c>0} \Big[ H(\mu|\mu_{2,c}) + \mathbb J_X(c) \Big].
   \]
 \vskip 1mm
 (3) If $s_n\ll k_n$, then $(L_{k_n})_{n\in\N}$ satisfies an LDP in the space $\mathcal M_1^q(\R)$ at speed $k_n$ with good rate function $\rate_{L,k_n}: \mathcal M_1^q(\R)\to [0,\infty]$ given by 
   \[
     \rate_{L,k_n}(\mu) := 
     \begin{cases}
       \mathbb J_X(c) &: \mu = \mu_{2,c} \cr
       \infty &: \text{otherwise}.
     \end{cases}
   \]
\end{thmalpha}

We shall not present the idea of proof for this result and rather refer the reader to \cite[Subsection 4.3]{KLR2019}.

\begin{rmk}
Let us note that, as in Theorem \ref{thm:klr constant regime}, the rate functions that appear above can be decomposed into a radial component (represented by $\mathbb J_X$ as a consequence of Assumption $A$), and an angular component from $A_{n,k_n}$ (captured by the relative entropy term). Depending on the relative growth rate of $k_n$ and $s_n$, different parts dominate the rate function and both terms are present only when $s_n = k_n$.
\end{rmk}

\subsection{Main results in the linear regime}

In the linear regime, when working with $(L_{k_n})_{n\in\N}$ in spaces of probability measures or when considering sequences of scaled $\ell_q$-norms of the random projections, one requires Assumption $A$ to obtain the desired LDPs. 

Let us define the second-moment-mapping $\mathbb M_2:\mathcal M_1^2(\R)\to \R$, $\mathbb M_2(\nu) := \int_\R |x|^2\,\nu(\dint x)$. The following theorem corresponds \cite[Theorem 2.15]{KLR2019}. 

\begin{thmalpha}\label{thm:klr linear regime}
 Let $q\in[1,2)$, $(k_n)_{n\in\N}\in\N^{\N}$ and $\lambda\in(0,1]$ be such that $k_n\sim \lambda n$, and assume that Assumption $A$ is satisfied with associated speed $s_n$ and good rate function $\mathbb J_X$. Then $(L_{k_n})_{n\in\N}$ satisfies an LDP in the space $\mathcal M_1^q(\R)$ at speed $s_n$ and with the good rate functions $\rate_{L,\lambda}: \mathcal M_1^q(\R)\to [0,\infty]$ of the following form 
 \vskip 1mm
 (1)   If $s_n=n$, then 
   \begin{align*}
     \rate_{L,\lambda}(\mu) & :=  \inf_{c>\sqrt{\lambda \mathbb M_2(\mu)}} \Bigg[ \mathbb J_X(c) - \frac{1-\lambda}{2}\log\Big(1-\frac{\lambda \mathbb  M_2(\mu)}{c^2}\Big) + \lambda \log(c) \Bigg] \cr
     & - \lambda h(\nu) + \frac{\lambda}{2}\log(2\pi e) + \frac{1-\lambda}{2}\log(1-\lambda),
   \end{align*}
where $h(\nu) := -\int_\R \log\Big(\frac{\dint\nu}{\dint x}\Big) \,\dint\nu$ denotes the entropy of $\nu$ for those $\nu$ with density, while $h(\nu):=-\infty$ otherwise; we also use the convention $0\log0 :=0$.
  \vskip 1mm
 (2) If $s_n\ll n$, then
   \[
     \rate_{L,\lambda}(\mu): = 
     \begin{cases}
       \mathbb J_X(c) &: \mu = \mu_{2,c} \cr
       \infty &: \text{otherwise}.
     \end{cases}
   \]
\end{thmalpha}

For the LDPs for sequences of $\ell_q$-norms of projected random vectors, we refer the reader to \cite[Theorem 2.16]{KLR2019}. Let us merely note here that for $q\in[1,2)$, the results follow directly from Theorem \ref{thm:klr linear regime} via the contraction principle and that the cases for $q\geq 2$ require a different approach (see \cite[Subsection 6.2]{KLR2019}).

\subsection{Main examples satisfying asymptotic thin-shell-type assumptions}

We shall now present some examples of sequences of random vectors $(X^{(n)})_{n\in\N}$ satisfying some of the assumptions presented in Subsection \ref{subsec:regimes and conditions}.

\subsubsection{The case of product measures}

Let us consider a sequence $(X_i)_{i\in\N}$ of iid $\R$-valued random variables and $X^{(n)}:=(X_1,\ldots,X_n)$, where we suppose that
  \[
    \Lambda_{X^2}(t) := \log \E\big[ e^{tX_1^2} \big] < \infty
  \]
for all $t$ in an open neighborhood around $0$. Then by Cram\'er's theorem (e.g., \cite[Theorem 2.2.1]{DZ2010}) the random sequence $(\|X^{(n)}\|_2^2/n)_{n\in\N}$ satisfies an LDP at speed $n$ with good rate function $\Lambda_{X^2}^*$. Applying the contraction principle with the continuous mapping $x\mapsto \sqrt{x}$, we immediately see that Assumption $A^*$ is satisfied with rate function 
  \[
    \mathbb J_X(x):= 
    \begin{cases}
    \Lambda_{X^2}^*(x^2) &: x\in [0,\infty) \cr
    \infty &: \text{otherwise}.
    \end{cases}
  \]  
The existence of a unique minimizer $m=\sqrt{\E[|X_1|^2]}$ of $\mathbb J_X$ follows from the strong law of large numbers, namely $n^{-1}\sum_{i=1}^nX_i^2 \stackrel{\text{a.s.}}{\longrightarrow} \E[X_1^2]=m^2$.

\subsubsection{The case of $\ell_p^n$-balls}

In what follows, we shall always denote by $X^{(n,p)}$, $n\in\N$, $p>0$, a random vector chosen uniformly from $n^{1/p}\B_p^n$. A key element in computations and proofs is again the Schechtmann--Zinn probabilistic representation (see \eqref{eq:schechtman-zinn representation}) with an additional factor $n^{1/p}$ because of the scaling of $\ell_p^n$-balls.

We start with $p\in[1,2)$. In this case, it already follows from Theorem \ref{thm:LDPp<q} (i.e., \cite[Theorem 1.3]{KPT2019_I}) that the random sequence $(\|X^{(n,p)}\|_2/\sqrt{n})_{n\in\N}$ satisfies an LDP at speed $n^{p/2}$ and with respective rate function, i.e., $(X^{(n,p)})_{n\in\N}$ satisfies Assumption $A$ with $s_n=n^{p/2}$ and unique minimizer 
  \[
    m_p=\sqrt{M_p(2)}=\Bigg( \frac{p^{2/p}}{3}\,\frac{\Gamma(1+\frac{3}{p})}{\Gamma(1+\frac{1}{p})} \Bigg)^{1/2};
  \]   
this means that Assumption $A^*$ is not satisfied. 

It is shown, however, in \cite[Proposition 3.3]{KLR2019} that Assumption $B$ is satisfied. More precisely, the authors prove the following: for $p\in[1,2)$, the random sequence $(X^{(n,p)})_{n\in\N}$ satisfies Assumption $B$ with speed $s_n=n^{\frac{2p}{2+p}}$ and the good rate function $\mathbb J_{X,p}:\R \to [0,\infty]$ given by
  \[
    \mathbb J_{X,p}(x) =
    \begin{cases}
    \frac{x^p}{p} &: x\geq 0 \cr 
    \infty &: x<0.
    \end{cases} 
  \] 
Let us note that a normalization with $1/p$ in $p^{-1}x^p$ is in a sense natural, since in that case we have Legendre-duality with $q^{-1}x^q$, where $q$ is the Hölder-conjugate to $p$. 

We refer to \cite[Subsection 3.2.2]{KLR2019}, specifically to \cite[Proposition 3.4]{KLR2019} and \cite[Theorem 3.5]{KLR2019}, for further examples in the case $p\in[1,2)$.

\vskip 2mm

We now continue with the case $p\in[2,\infty)$, which is the rather easy one. In \cite[Proposition 3.7]{KLR2019} the authors prove the following: for $p\in[2,\infty)$, the random sequence $(X^{(n,p)})_{n\in\N}$ satisfies Assumption $A^*$ with speed $s_n=n$ and good rate functions 
  \[
    \mathbb J_{X,2}(x) :=  
    \begin{cases}
      -\log(x) &: x\in(0,1] \cr
      \infty &: \text{otherwise}.
    \end{cases}
  \]
for $p=2$ and 
  \[
    \mathbb J_{X,p}(x) := 
     \begin{cases}
      \inf_{y\geq x}\Big[ \frac{1}{2}\log\Big(\frac{y^2}{x^2} \Big) + F^*_p(y) \Big] &: x\in(0,1) \cr
      \infty &: \text{otherwise},
    \end{cases}
  \]
for $p>2$, where 
  \[ 
    F^*_p(y) = \sup_{t_1,t_2\in\R} \Bigg[ t_1y + t_2 -\log\Bigg(\int_{\R} e^{t_1x^2+t_2|x|^p}f_p(x)\,\dint x \Bigg) \Bigg],\qquad y\in\R.
   \]
Recall that $f_p$ denotes the $p$-Gaussian density as introduced in \eqref{eq:pgaussian density}. The rate function $ \mathbb J_{X,p}$ in the LDP again has a unique minimizer at $m_p=\sqrt{M_p(2)}$. When $p=2$, then the LDP arises solely from the uniform distribution and  follows from \cite[Lemma 3.3]{GKR2017}. For $p>2$ it follows from the results of Z.~Kabluchko, J.~Prochno, and C.~Thäle \cite[Theorems 1.1(a) and 1.2]{KPT2019_I} taking into account the different scaling. 

There are a number of further results that can be obtained, for instance, for Euclidean norms of the random projections, and we refer the reader to \cite[Subsection 3.2.3]{KLR2019} directly, in particular to \cite[Remark 3.9]{KLR2019}.

\subsubsection{The case of Orlicz balls}

As a last example, we shall present the case of random vectors uniformly distributed on appropriately scaled Orlicz balls, where the defining Orlicz function $M$ is assumed to be super-quadratic, meaning that
  \[
    \frac{M(x)}{x^2} \stackrel{x\to\infty}{\longrightarrow} \infty.
  \]
In what follows, we allow $M:\R \to [0,\infty]$ and denote again by $\mathscr D_M= \{x\in\R \,:\, M(x)<\infty \}$ the effective domain of $M$. We recall that
  \[
    B_M^n(n) := \Bigg\{(x_i)_{i=1}^n \in\R^n\,:\, \sum_{i=1}^n M(x_i) \leq n \Bigg\}.
  \]
Further, let us use again the Orlicz-moment-mapping (which already appeared in the Sanov-type large deviations in Orlicz spaces)
  \[
    \mathbb M_M:\mathcal M_1(\R) \to [0,\infty), \quad \nu\mapsto \int_\R M(x)\,\nu(\dint x)
  \] 
for those probability measures $\nu$ for which the integral is finite. For $\beta>0$, we work again with the probability measures
$\mu_{M,\beta}\in\mathcal M_1(\R)$  given by
  \[
    \mu_{M,\beta}(\dint x) := \frac{e^{-\beta M(x)}}{\int_{\mathscr D_M}e^{-\beta M(y)}\, \lambda(\dint y)} \lambda(\dint x).
  \]
Moreover, we define the function 
  \[
    \mathcal J(u,v) := \sup_{s<0,\, t\in\R} \Bigg[ su + tv - \log\Bigg( \int_{\mathscr D_M} e^{sM(x)+tx^2} \,\dint x\Bigg) \Bigg],\quad u,v\in[0,\infty).
  \]     
In \cite[Proposition 3.11]{KLR2019}, the authors prove the following result: if $M$ is a super-quadratic Orlicz function and, for each $n\in\N$, $X^{(n)}$ is uniformly distributed in $B_M^n(n)$, then the random sequence $(X^{(n)})_{n\in\N}$ satisfies Assumption $A^*$ with rate function 
  \[
    \mathbb J_{X,M}(z) := \mathcal J(1,z^2) - \sup_{s<0} \Bigg[s -  \log\Bigg( \int_{\mathscr D_M} e^{sM(x)}\,\dint x\Bigg) \Bigg], \quad z\in[0,\infty),
  \]  
which is $\infty$ otherwise. Moreover, there exists a unique $\beta^*>0$ such that $\mathbb M_M(\mu_{M,\beta^*})=1$ and  $\mathbb J_{X,M}$ has unique minimizer at  $m:=\mathbb M_2(\mu_{M,\beta^*})$, with $\mathbb M_2$ being the second-moment-mapping.  

Let us note that the result for super-quadratic Orlicz functions can also be used to deduce the result we presented for $\ell_p^n$-balls when $p>2$. To show this one chooses $M(x)=|x|^p$.

\subsection*{Acknowledgement}
Joscha Prochno's research is supported by the German Research Foundation (DFG) under project 516672205 and by the Austrian Science Fund (FWF) under project P-32405.

\bibliographystyle{plain}
\bibliography{ldp_gfa}

\bigskip
\bigskip
	
	\medskip
	
	\small

	\noindent \textsc{Joscha Prochno:} Faculty of Computer Science and Mathematics, University of Passau, Dr.-Hans-Kapfinger-Str. 30, 94032 Passau, Germany.
	
	\noindent
	{\it E-mail:} \texttt{joscha.prochno@uni-passau.de}

\end{document}